\documentclass[a4paper,10pt,xcolor=svgnames,dvips]{amsart}
\usepackage{amsmath,amssymb, amsbsy}
\usepackage{subfigure}
\usepackage{color,psfrag}
\usepackage{enumerate}
\usepackage[dvips]{graphicx}
\usepackage{pstricks}
\usepackage{pst-plot}
\usepackage{auto-pst-pdf}
\usepackage[utf8]{inputenc}
\usepackage{hyperref}

\usepackage{latexsym}
\usepackage{a4wide}
\usepackage{amscd}
\usepackage{mathrsfs}
\usepackage{enumerate}

\newcommand{\R}{{\mathbb R}}
\newcommand{\N}{{\mathbb N}}

\newcommand{\SN}{{\mathbb S}^{N-1}}
\newcommand{\weakly}{\rightharpoonup}
\newcommand{\eps }{\varepsilon}
\newcommand{\e }{{\mathbf e}_1}

\renewcommand{\geq }{\geqslant}

\renewcommand{\leq }{\leqslant}
\newcommand{\abs}[1]{\left\vert #1 \right\vert}
\newcommand{\Di}[1]{{\mathcal D}^{1,2}(#1)}

\newcommand{\nor}[1]{\left\| #1 \right\|}

\newenvironment{pf}{\noindent{\bf Proof}.\enspace}{\hfill\qed\medskip}
\newenvironment{pfn}[1]{\noindent{\bf Proof of
    {#1}.\enspace}}{\hfill\qed\medskip}
\newtheorem{Theorem}{Theorem}[section]
\newtheorem{Corollary}[Theorem]{Corollary}
\newtheorem{Lemma}[Theorem]{Lemma}
\newtheorem{Proposition}[Theorem]{Proposition}
\theoremstyle{definition}

\newtheorem{remark}[Theorem]{Remark}

\begin{document}

\title[Attaching a thin handle on the spectral rate of convergence]{On the sharp effect of attaching a thin handle on the spectral rate of convergence}
\author{Laura Abatangelo, Veronica Felli, Susanna Terracini}

\address{
\hbox{\parbox{5.7in}{\medskip\noindent
  L. Abatangelo, V. Felli\\
Dipartimento di Matematica e Applicazioni,\\
 Universit\`a di Milano Bicocca, \\
Piazza Ateneo Nuovo, 1, 20126 Milano (Italy)         . \\[2pt]
         {\em{E-mail addresses: }}{\tt laura.abatangelo@unimib.it, veronica.felli@unimib.it.}\\[5pt]
 S. Terracini\\
 Dipartimento di Matematica ``Giuseppe Peano'',\\
Universit\`a di Torino, \\
Via Carlo Alberto, 10,
10123 Torino (Italy). \\[2pt]
                                     \em{E-mail address: }{\tt susanna.terracini@unito.it.}}}
}

\date{\today}

\thanks{
\indent Partially supported by the PRIN2009 grant ``Critical Point Theory and
Perturbative Methods for Nonlinear \\
\indent Differential Equations''.}

\begin{abstract}
  Consider two domains connected by a thin tube: it can be shown that
  the resolvent of the Dirichlet Laplacian is continuous with respect
  to the channel section parameter. This in particular implies the
  continuity of isolated simple eigenvalues and the corresponding
  eigenfunctions with respect to domain perturbation.  Under an explicit
 nondegeneracy condition, we improve this information
  providing a sharp control of the rate of convergence of the
  eigenvalues and eigenfunctions in the perturbed domain to the
  relative eigenvalue and eigenfunction in the limit domain. As an application, we prove that, again under an explicit  nondegeneracy condition, the case of resonant domains features polinomial splitting of the two eigenvalues  and a clear bifurcation of eigenfunctions.
\end{abstract}

\maketitle

\section{Introduction and statement of the main results}

The aim of this paper is to investigate the behavior of Dirichlet
eigenvalues in varying domains, when a shrinking cylindrical handle is
attached to a smooth region, seeking not only for the rate of
convergence but also for sharp asymptotics.  Since we consider a
tubular handle with a cross-section of radius of order $\eps\to0^+$
(see Figure \ref{fig:1}), it is quite natural to expect the rate of
convergence of the eigenvalues to rely essentially on the capacity of
the junction points and hence to be of order $\eps^N$, being $N$ the
space dimension.

\begin{figure}[h]\label{fig:1}
\begin{center}
\psset{unit=0.5cm}
\begin{pspicture}(-5.5,-3.5)(5.5,3.5)
%\pspolygon[fillstyle=solid](-5.5,1)(0,1)(0,-1)(-5.5,-1)
\pscurve[fillstyle=solid]
(-0.75,1)(-0.6,3.5)(1.2,2)(1,0)(0.8,-2.4)(-0.6,-3.5)(-1.5,-1)
% \rput*[l]{180}(-10,0){\pscurve[fillstyle=solid]
% (-4.75,1)(-4.6,3.5)(-3.2,2)(-3,0)(-3.8,-2.4)(-4.6,-3.5)(-5.5,-1)}
% \psbezier[linecolor=blue,linestyle=dashed](-5.5,-1)(-5.2,0.5)(-5,-0.5)(-4.75,1) 
\psbezier(-1.5,-1)(-1.2,0.5)(-1,-0.5)(-0.75,1)
% \pscurve[linecolor=blue,fillstyle=solid]
% (0,1.5)(0.1,1.4)(0.25,0.5)(0.15,-0.4)(0,-0.5)
% \pscurve[linestyle=dashed,linecolor=blue,fillstyle=solid]
% (0,1.5)(-0.3,1.4)(-0.15,0.5)(-0.25,-0.4)(0,-0.5)
%\psline[linecolor=blue](-5.5,1)(0,1)
%\psline[linecolor=blue](-5.5,-1)(0,-1)
\psellipse[linewidth=.5pt,linestyle=dashed](-0.25,-2.4)(1.05,0.5)
\psellipse[linewidth=.5pt,linestyle=dashed](0,2.4)(1,0.5)
\pscurve[fillstyle=solid]
(-2.6,1)(-4.6,3.5)(-5.2,2)(-4.8,0)(-4.2,-3.5)(-2.8,-1)
\psbezier(-2.6,1)(-2.4,0.5)(-2.65,-0.5)(-2.8,-1)
%\pscurve[linecolor=blue,fillstyle=solid]
%(-4,2)(-4.3,1.9)(-4.25,1)(-4.15,0)(-4,0)
%\pscurve[linestyle=dashed,linecolor=blue,fillstyle=solid]
%(-4,2)(-3.8,1.9)(-3.5,1)(-3.25,0.1)(-4,0)
%\pscurve(-4,2)(-3.8,1.9)(-2,0.5)(-0.2,1.5)(0,1.5)
%\pscurve(-4,0)(-3.6,-0.1)(-1.9,-1.2)(-1.5,-1)(-0.4,-0.5)(0,-0.5)
\psellipse[linewidth=.5pt,linestyle=dashed](-4.05,-2.4)(0.8,0.4)
\psellipse[linewidth=.5pt,linestyle=dashed](-4.25,2.4)(0.95,0.5)
\uput[90](-4.05,-1.4){$D^-$}
\uput[90](-0.05,-1.4){$D^+$}
\end{pspicture}
\qquad
 \psset{unit=0.5cm}
\begin{pspicture}(-5.5,-3.5)(5.5,3.5)
%\pspolygon[fillstyle=solid](-5.5,1)(0,1)(0,-1)(-5.5,-1)
\pscurve[fillstyle=solid]
(-0.75,1)(-0.6,3.5)(1.2,2)(1,0)(0.8,-2.4)(-0.6,-3.5)(-1.5,-1)
% \rput*[l]{180}(-10,0){\pscurve[fillstyle=solid]
% (-4.75,1)(-4.6,3.5)(-3.2,2)(-3,0)(-3.8,-2.4)(-4.6,-3.5)(-5.5,-1)}
% \psbezier[linecolor=blue,linestyle=dashed](-5.5,-1)(-5.2,0.5)(-5,-0.5)(-4.75,1) 
\psbezier(-1.5,-1)(-1.2,0.5)(-1,-0.5)(-0.75,1)
\pscurve[fillstyle=solid]
(0,0.5)(0.2,0.4)(0.25,0.25)(0.15,-0.4)(0,-0.5)
\pscurve[linewidth=.5pt,linestyle=dashed,fillstyle=solid]
(0,0.5)(-0.1,0.4)(-0.15,0.25)(-0.25,-0.4)(0,-0.5)
%\psline[linecolor=blue](-5.5,1)(0,1)
%\psline[linecolor=blue](-5.5,-1)(0,-1)
\psellipse[linewidth=.5pt,linestyle=dashed](-0.25,-2.4)(1.05,0.5)
\psellipse[linewidth=.5pt,linestyle=dashed](0,2.4)(1,0.5)
\pscurve[fillstyle=solid]
(-2.6,1)(-4.6,3.5)(-5.2,2)(-4.8,0)(-4.2,-3.5)(-2.8,-1)
\psbezier(-2.6,1)(-2.4,0.5)(-2.65,-0.5)(-2.8,-1)
\pscurve[fillstyle=solid]
(-4,1)(-4.3,0.9)(-4.25,0.8)(-4.15,0)(-4,0)
\pscurve[linewidth=.5pt,linestyle=dashed,fillstyle=solid]
(-4,1)(-3.8,0.9)(-3.7,0.8)(-3.65,0.1)(-4,0)
\pscurve(-4,1)(-3.5,0.9)(-2,-0.5)(-0.2,0.5)(0,0.5)
\pscurve(-4,0)(-3.6,-0.1)(-1.9,-1.2)(-1.5,-1)(-0.4,-0.5)(0,-0.5)
\psellipse[linewidth=.5pt,linestyle=dashed](-4.05,-2.4)(0.8,0.4)
\psellipse[linewidth=.5pt,linestyle=dashed](-4.25,2.4)(0.95,0.5)
\uput[90](-4.05,-1.4){$D^-$}
\uput[90](-0.05,-2){$D^+$}
\uput[90](-2,-0.5){$\mathcal C_\eps$}
\qdisk(0,0){1pt}
\uput[0](0,0){$\mathbf e_1$}
\psellipse[fillstyle=solid,fillcolor=lightgray,linestyle=dashed,linewidth=.5pt](-2,-0.85)(0.2,0.4)
\uput[180](-2,-2){$\Sigma$}
\pscurve[linewidth=.5pt]{<-}(-2,-0.85)(-2.3,-1)(-2.5,-1.65)
\end{pspicture}
\end{center}
\caption{The disconnected domain $D^-\cup D^+$ becomes connected by the attachment of the handle $\mathcal C_\eps$.}
\end{figure}
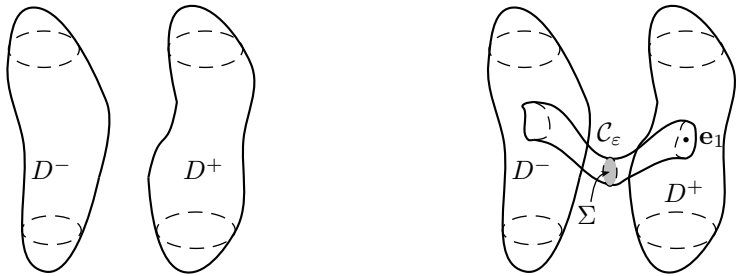

Referring to Figure \ref{fig:1}, let $u_0$ be the $k$-th eigenfunction on the limit (disconnected) domain
$D^-\cup D^+$ completely supported only in the connected component
$D^+$. By the attachment of a handle $\mathcal C_\eps$ with cross
section of radius of order $\eps$, its mass  will be pushed 
into the channel in order to spread over the new entire domain
$D^-\cup \mathcal C_\eps\cup D^+$. 
Besides the tubular shape of the  connecting tube, we require, as a
basic assumption to start our analysis,
that  the handle is attached at a point $\e$
of $\partial D^+$ where $u_0$ has a zero of order one, i.e. its
normal derivative is different from zero. If moreover $u_0$ is 
simple and suitably normalized, the corresponding eigenvalue $\lambda_k$ can be continued into a
family $\lambda_k^\eps$ of eigenvalues corresponding to normalized
eigenfunctions $u_\eps$ on the perturbed domain.

We prove that, in such a setting, there exists the limit
\begin{equation}\label{eq:58}
\lim_{\eps\to0} \eps^{-N}(\lambda_k-\lambda_k^\eps)=
\bigg(\frac{\partial u_0}{\partial \nu}(\e)\bigg)^{\!\!2} \mathfrak
C(\Sigma),
\end{equation}
where $\mathfrak C(\Sigma)$ is a positive constant
depending only the geometry of the junction section $\Sigma$ (see
Theorems \ref{teo_asintotico_autovalori} and
\ref{teo_asintotico_autofunzioni} below).  Thinking to the
eigenfunction $u_0$ as pushed into the channel, we can imagine that a
force acts over the junction between the channel and the domain where
$u_0$ is supported.  The constant $\mathfrak C(\Sigma)$ represents
indeed the \emph{compliance} of the channel's junction, under a
constant force concentrated at the junction section; the compliance,
which can be expressed as the $L^1$-norm of the trace of a suitable
harmonic function over the channel section, measures the faculty of an
elastic membrane to adjust or to resist to a force applied on the
section, see \eqref{def_compliance} for a precise definition.  Our
proof consists in a sharp differentiation with respect to the
parameter, which requires first a careful analysis of the transition functions which have
to be attached to $u_0$ in order to push it over the channel. In this
way, we will prove that, once more,
\begin{equation*}
\lim_{\eps\to 0^+}
\eps^{-N}\|u_\eps-u_0\|^2_{\Di{\R^N}}
=
\bigg(\frac{\partial u_0}{\partial \nu}(\e)\bigg)^{\!\!2}
\mathfrak C(\Sigma),
\end{equation*}
where $u_\eps$ and $u_0$ are trivially extended to the whole $\R^N$.

As an application of the sharp asymptotics \eqref{eq:58} we are able
to treat also the resonant case: if $\lambda_k=\lambda_{k+1}$ is a double
eigenvalue on the limit disconnected domain which is a simple
eigenvalue both on $D^-$ and on $D^+$, an asymptotics for eigenvalues
of type \eqref{eq:58} still holds if the limit problem is
asymmetrical, e.g.  under the assumption that the normal
derivatives of the limit eigenfunctions at the junctions are different
from each other (see Theorem \ref{t:reso}). In this
case, it turns out that the splitting of the two subsequent eigenvalues 
$\lambda_k^\eps,\lambda_{k+1}^{\eps}$ has the polynomial vanishing
order $\eps^N$ (see Remark \ref{r:splitting_polinomiale}); 
 such result complements those in 
 \cite{BHM}, where it was proved that, in a symmetric dumbbell domain
 with a shrinking handle, the
 splitting of the first two eigenvalues vanishes with exponential rate.
Moreover, in contrast with \cite{BHM}, we can localize each
approximating 
eigenfunction on its corresponding region, up to an exponentially
vanishing tail, see Theorem \ref{t:reso_autofun}.

For expository reasons, the present paper discusses the effect of
attaching a thin handle on the spectral rate of convergence only for
dumbbell domains. However, up to minor modifications, the results
obtained here hold true in quite general contexts, since they rely
essentially on the attachment of a shrinking handle at a point in
which the limit eigenfunction has a zero of order $1$; therefore the
presence/lack of a second domain beyond the channel and its shape seem
to be irrelevant for the validity of the asymptotics we are going to
derive. The choice of focusing on the dumbbell structure is motivated
not only by the large attention devoted to this peculiar case of
singularly perturbed domain in the literature, due to the many
interesting related spectral phenomena (see  \S
\ref{sec:motiv-refer-liter} below), but also by the fact that some
preliminary results required in our analysis have been obtained for
dumbbell domains in \cite{AFT12,FT12}, where the singular asymptotic
behavior of eigenfunctions at the second junction of the tube is
described.

\subsection{Dumbbell domains}

As a paradigmatic example, we consider a dumbbell domain where each
``chamber'' has a constant section, namely we straighten out the
handle and assume its section $\Sigma$ to be constant along its whole
length, whereas we spread out the two domains $D^+$ and $D^-$ assuming
they are two entire half-spaces, see Figure \ref{fig:dd}. We observe
that such a simplification of the domain's geometry does not imply a
substantial loss of generality if a suitable weight is introduced in
the eigenvalue problem under investigation: indeed, the effect of a
diffeomorphism transforming a generic dumbbell in a dumbbell with two
half-spaces as chambers is the transformation of the eigenvalue
problem into a weighted one.

Let $N\geq 3$. We denote
\begin{equation*}
D^-=\{(x_1,x')\in \R\times \R^{N-1}:x_1<0\},\quad
D^+=\{(x_1,x')\in \R\times \R^{N-1}:x_1>1\},
\end{equation*}
and, 
 for all $t>0$,
\begin{equation*}
B^+_t:=D^+\cap B({\mathbf e}_1,t),\quad
B^-_t:=D^-\cap B({\mathbf 0},t),\quad
\Gamma_t^+=D^+\cap \partial B^+_{t},\quad 
\Gamma_t^-=D^-\cap \partial B^-_{t},\quad 
\end{equation*}
where ${\mathbf e}_1
=(1,0,\dots,0)\in \R^N$, ${\mathbf 0}
=(0,0,\dots,0)\in \R^N$,
and $B(P,t):=\{x\in\R^N:|x-P|<t\}$ denotes the
ball of radius $t$ centered at $P$.
Let 
$\Sigma\subset \R^{N-1}$ be an open bounded  set with
$C^{2,\alpha}$-boundary containing $0$. For simplicity of notation, we
assume that $\Sigma$ satisfies 
\begin{equation}\label{eq:condsigma}
\big\{x'\in\R^{N-1}:|x'|\leq\tfrac1{2}\big\}\subset \Sigma\subset
\{x'\in\R^{N-1}:|x'|<1\}.
\end{equation}
Let $p\in C^1(\R^N,\R)\cap L^{\infty}(\R^N)$ be a weight  satisfying
\begin{align}
\label{eq:p} & p\geq 0\text{ a.e. in }\R^N,\
p\in L^{N/2}(\R^N),\ \nabla p(x)\cdot x\in L^{N/2}(\R^N),
\ \frac{\partial p}{\partial x_1}\in L^{N/2}(\R^N),\\
\label{eq:p2} & 
\begin{cases}
p\not\equiv 0\text{ in }D^-,\quad
 p\not\equiv 0\text{ in }D^+,\\
 p(x)=0\text { for all } 
x\in \{(x_1,x')\in \R\times \R^{N-1}:\frac12\leq x_1\leq1,\ x'\in\Sigma\}
\cup B^+_3.
\end{cases}
\end{align}
Assumption \eqref{eq:p2} is required for technical reasons as in
\cite{AFT12,FT12}; it is used in \S \ref{sec:point-wise-energy} to
prove some preliminary estimates of eigenfunctions on the perturbed domain.
 
For every open set $\Omega\subset\R^N$,  we denote as $\sigma_p(\Omega)$ the
 set of the diverging eigenvalues 
\[
\lambda_1(\Omega)\leq
\lambda_2(\Omega)\leq\cdots\leq\lambda_k(\Omega)\leq \cdots
\]
 (where each $\lambda_k(\Omega)$ is repeated as many times as its multiplicity) of
 the weighted eigenvalue problem
\begin{equation}\label{eq:30}
\begin{cases}
-\Delta \varphi=\lambda p \varphi,&\text{in }\Omega,\\
\varphi=0,&\text{on }\partial \Omega.
\end{cases}
\end{equation}
It is easy to verify that $\sigma_p(D^-\cup D^+)=\sigma_p(D^-)\cup \sigma_p(D^+)$.

Let $\Omega^\eps\subset\R^N$ be the domain  formed by
connecting the two half-spaces $D^+,D^-$ with a tube of length $1$ and
cross-section $\eps\Sigma$, i.e.
\begin{equation}\label{eq:31}
\Omega^\eps=D^-\cup \mathcal C_\eps\cup D^+,
\end{equation}
where    $\eps\in (0,1)$ and $\mathcal C_\eps=\{(x_1,x')\in \R\times
\R^{N-1}:0\leq x_1\leq1,\ \frac{x'}{\eps}\in\Sigma\}$, see Figure \ref{fig:dd}.
\begin{figure}[h]
 \centering
   \begin{psfrags}
     \psfrag{D-}{$D^-$}
     \psfrag{D+}{$D^+$}
\psfrag{e}{${\scriptsize{\eps}}$}
\psfrag{1}{${\scriptsize{1}}$}
\psfrag{C}{$\mathcal C_\eps$}
     \includegraphics[width=7cm]{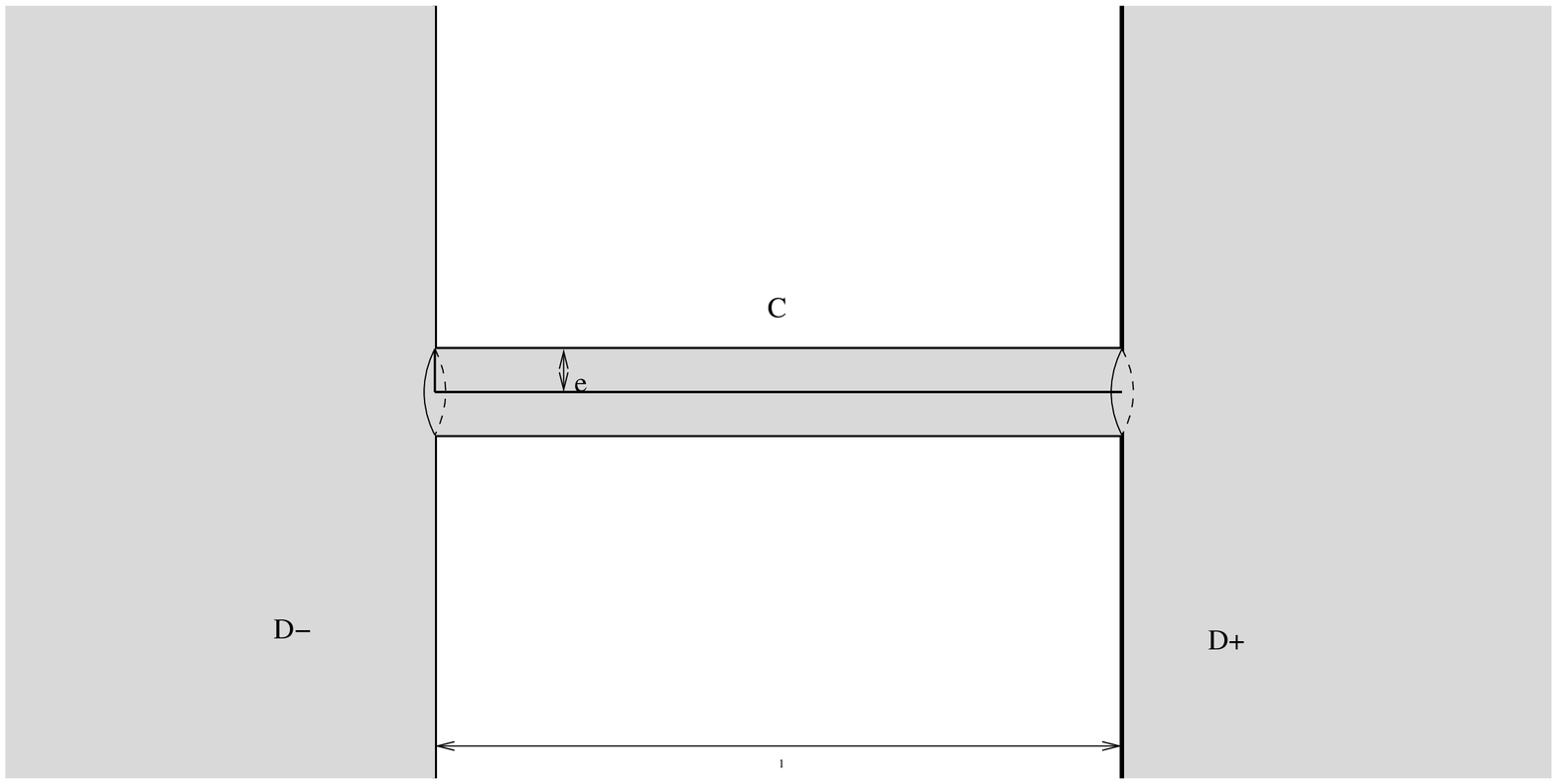}
   \end{psfrags}
 \caption{The domain $\Omega^\eps$.}\label{fig:dd}
\end{figure}

Here and in the sequel, for every open set $\Omega\subseteq\R^N$, ${\mathcal
    D}^{1,2}(\Omega)$ denotes the functional space obtained as completion of
$C^\infty_{\rm c}(\Omega)$
with respect to the Dirichlet norm $\|u\|_{{\mathcal
    D}^{1,2}(\Omega)}=\big(\int_{\Omega}|\nabla u|^2dx\big)^{1/2}$.

\subsection{Main results}

By standard minimization methods, it is easy to prove that the minimum 
\begin{equation}\label{eq:m(sigma)}
m(\Sigma)=\min_{w\in \mathcal
  D^{1,2}(\widetilde D)} J_\Sigma(w),
\end{equation}
is achieved, where 
  \[
\widetilde D =D^+\cup T_1^-, \quad
  T_1^-=\{(x_1,x'):x'\in\Sigma,\ x_1\leq1\},
\]
and 
\begin{align}\label{eq:Jsigma}
&J_\Sigma:{\mathcal
   D}^{1,2}(\widetilde D)\to\R,\\
\notag& J_\Sigma(w)= \frac12\int_{\widetilde
    D}|\nabla w(x)|^2\,dx-\int_{\Sigma}
 w(1,x')\,dx',\quad\text{for every }w\in \mathcal
  D^{1,2}(\widetilde D).
\end{align}
It is easy to verify that 
\[
m(\Sigma)<0,
\]
see Corollary \ref{cor:min_phi1}.
Moreover, we notice that, denoting as $(\mathcal D^{1,2}(\R^N))^\star$
the dual space of $\mathcal D^{1,2}(\R^N)$ and 
letting $\mathcal F\in (\mathcal D^{1,2}(\R^N))^\star$ defined as 
\[
\phantom{a}_{(\mathcal D^{1,2}(\R^N))^\star}\big\langle \mathcal F,
w\big\rangle_{\mathcal D^{1,2}(\R^N)} =
\int_{\Sigma}w(1,x')\,dx',
\]
$m(\Sigma)$ can be rewritten as 
\begin{equation}\label{eq:legame_m_c}
 m(\Sigma)=\min_{w\in \mathcal
  D^{1,2}(\widetilde D)} \bigg(\frac12\int_{\widetilde
    D}|\nabla w(x)|^2\,dx-
\phantom{a}_{(\mathcal D^{1,2}(\R^N))^\star}\big\langle \mathcal F,
w\big\rangle_{\mathcal D^{1,2}(\R^N)}\bigg) 
=-\frac{\mathfrak{C}(\Sigma)}2
\end{equation}
where 
\begin{equation}\label{def_compliance}
  \mathfrak{C}(\Sigma):=
\max_{w\in \mathcal
  D^{1,2}(\widetilde D)} \bigg(2
\phantom{a}_{(\mathcal D^{1,2}(\R^N))^\star}\big\langle \mathcal F,
w\big\rangle_{\mathcal D^{1,2}(\R^N)}-\int_{\widetilde
    D}|\nabla w(x)|^2\,dx\bigg) 
\end{equation}
represents the compliance functional
associated to the force $\mathcal F$ concentrated on the section
$\Sigma$ in the flavor of \cite{BS,BSV}.  In general, the compliance
functional measures the rigidity of a membrane subject to a given
(vertical) force: the maximal rigidity is obtained by minimizing the
compliance functional $\mathfrak C(\Sigma)$ in the class of admissible
regions $\Sigma$.  With this notation and concepts in mind we state
our main results.

Let us first assume that there exists $k_0\geq 1$ such that
\begin{align}
  \label{eq:53} \lambda_{k_0}(D^+)&\text{ is simple and the
    corresponding eigenfunctions}\\
\notag&\text{ have in ${\mathbf e}_1
    $ a zero of order $1$},\\
  \label{eq:54} 
 \lambda_{k_0}(D^+)&\not\in \sigma_p(D^-).
  \end{align}
We can then fix an eigenfunction $u_0\in{\mathcal
    D}^{1,2}(D^+)\setminus\{0\}$ associated to $\lambda_{k_0}(D^+)$,
  i.e. solving
\begin{equation}\label{eq:u0}
\begin{cases}
-\Delta u_0=\lambda_{k_0}(D^+) p u_0,&\text{in }D^+,\\
u_0=0,&\text{on }\partial D^+,
\end{cases}
\end{equation}
such that
\begin{equation}\label{eq:13}
\frac{\partial u_0}{\partial x_1}({\mathbf e}_1
)>0\quad\text{and}\quad \int_{D^+}p(x)u_0^2(x)\,dx=1.
\end{equation}
From \cite[Example 8.2, Corollary 4.7, Remark 4.3]{daners} (see also
\cite[Lemma 1.1]{FT12}), it follows that, letting
\begin{equation*}
\lambda_\eps=\lambda_{\bar k}(\Omega^\eps)
\end{equation*}
where $\bar k=k_0+
\mathop{\rm card}\big\{
j\in\N\setminus\{0\}:\lambda_j(D^-)\leq \lambda_{k_0}(D^+)\}$,
so that $\lambda_{k_0}(D^+)=\lambda_{\bar k}(D^-\cup D^+)$, there
holds
\begin{equation}\label{eq:52}
\lambda_\eps\to \lambda_{k_0}(D^+)\quad\text{as }\eps\to0^+.
\end{equation}
We will denote
\[
\lambda_0=\lambda_{k_0}(D^+).
\]
Furthermore, for every $\eps$ sufficiently small, $\lambda_\eps$
is simple and there exists an eigenfunction $u_\eps$
associated to $\lambda_\eps$, i.e. satisfying
\begin{equation}\label{problema}
\begin{cases}
-\Delta u_\eps=\lambda_\eps p u_\eps,&\text{in }\Omega^\eps,\\
u_\eps=0,&\text{on }\partial \Omega^\eps,
\end{cases}
\end{equation}
such that
\begin{equation}\label{convergenza_u_0}
\int_{\Omega^\eps}p(x)u_\eps^2(x)\,dx=1\quad\text{and}\quad  u_\eps\to
u_0\text{ in }{\mathcal D}^{1,2}(\R^N)
\text{ as }\eps\to0^+,
\end{equation}
where in the above formula we mean the functions $u_\eps,u_0$ to be trivially extended to the whole $\R^N$.
We refer to \cite[\S 5.2]{bucur2006} for uniform convergence of eigenfunctions.

\begin{Theorem}\label{teo_asintotico_autovalori}
  Under assumptions \eqref{eq:condsigma}, \eqref{eq:p}, \eqref{eq:p2},
  \eqref{eq:53}, and \eqref{eq:54}, let
  $\lambda_0=\lambda_{k_0}(D^+)=\lambda_{\bar k}(D^- \cup D^+)$ be the
  $\bar k$-th eigenvalue of problem \eqref{eq:30} on $D^- \cup D^+$
  (which is equal to the simple $k_0$-th eigenvalue on $D^+$) and
  $\lambda_\eps=\lambda_{\bar k}(\Omega^\eps)$ be the $\bar k$-th
  eigenvalue of problem \eqref{problema} on the domain $\Omega^\eps$
  defined in \eqref{eq:31}. Then
\begin{equation}\label{eq:36}
\lim_{\eps\to 0^+}\frac{\lambda_0 - \lambda_\eps}{\eps^N}=
\bigg(\frac{\partial u_0}{\partial x_1}(\e)\bigg)^{\!\!2}
\mathfrak C(\Sigma),
\end{equation}
with $u_0$  as in \eqref{eq:u0} and \eqref{eq:13}, and $\mathfrak C(\Sigma)$ 
as in \eqref{def_compliance}.
\end{Theorem}

\begin{Theorem}\label{teo_asintotico_autofunzioni}
Under assumptions
\eqref{eq:condsigma}, \eqref{eq:p}, \eqref{eq:p2}, \eqref{eq:53}, and
\eqref{eq:54}, let $u_\eps$ and $u_0$ as in \eqref{eq:u0},
\eqref{eq:13}, \eqref{eq:52},
\eqref{problema}, \eqref{convergenza_u_0}. 
Then 
\begin{equation}\label{eq:as_eigenfunctions}
\lim_{\eps\to 0^+}
\eps^{-N}\|u_\eps-u_0\|^2_{\Di{\R^N}}
=
\bigg(\frac{\partial u_0}{\partial x_1}(\e)\bigg)^{\!\!2}
\mathfrak C(\Sigma),
\end{equation}
where $u_\eps$ and $u_0$ are trivially extended to the whole $\R^N$
and $\mathfrak C(\Sigma)$ is defined in \eqref{def_compliance}.
\end{Theorem}

We observe that, once the measure of the section $\Sigma$ is fixed,
the shape minimizing $m(\Sigma)$ and hence maximizing both the limits
$\lim_{\eps\to 0^+}\eps^{-N}(\lambda_0 - \lambda_\eps)$ and
$\lim_{\eps\to 0^+}\eps^{-N}\|u_\eps-u_0\|^2_{\Di{\R^N}}$ is the
spherical one, as we will show in Proposition \ref{p:steiner} by
Steiner rearrangement. Hence, the disk-shaped section of the tube is
the one which makes as slow as possible the convergence of the
eigenvalues on the perturbed domain to the eigenvalues on the limit
domain, as the handle thickness shrinks to zero. In other words, this
means that among all the admissible sections $\Sigma$, the disk
attains the minimum of the rigidity of the domain $\widetilde D$: from
the opposite point of view, in the case of a round section, the
eigenfunctions located in the right domain $D^+$ are the most
sensitive to the attachment of the thin handle if compared to the case
of more indented sections.  This phenomenon can be read in Theorems
\ref{teo_asintotico_autovalori} and \ref{teo_asintotico_autofunzioni},
since the limits in \eqref{eq:36} and \eqref{eq:as_eigenfunctions} attain their maximal (positive)
constant at a disk-shaped section: symmetrization of the section makes
the difference $\lambda_0 - \lambda_\eps$ and
$\|u_\eps-u_0\|^2_{\Di{\R^N}}$ drift away from being $o(\eps^N)$.

The proof of Theorem \ref{teo_asintotico_autovalori}, which is presented in
Section \ref{sec:proof-theor-reft}, 
is based on the Courant-Fisher \emph{minimax characterization} of
eigenvalues: the estimates from above and below of the Rayleigh
quotient used to prove the theorem are based on the analysis of 
proper test functions introduced in Section
\ref{sec:prel-notat-techn}. 
Theorem \ref{teo_asintotico_autofunzioni} is proved in Section
\ref{sec:rate-conv-eigenf}, using some blow-up analysis
developed in Section \ref{sec:prel-notat-techn} and the invertibility
of an operator associated to the eigenvalue problem on $D^+$ (see \eqref{def_operatore_F}).

\medskip In section \ref{sec:resonant-case} we drop  assumption \eqref{eq:54}
and assume that  $\lambda_{\bar k}(D^-\cup D^+)\in
\sigma_p(D^-)\cap\sigma_p(D^+)$  is 
a simple eigenvalue on $D^-$, a simple eigenvalue on $D^+$, and a
double eigenvalue on $D^-\cup D^+$. 
We prove that by attaching the shrinking handle 
at two points where the normal derivatives of the limit eigenfunctions
are different from each other, the double eigenvalue $\lambda_{\bar k}(D^-\cup D^+)$ on the limit domain
is approximated by two different branches of eigenvalues on the
perturbed domain as 
$\eps\to 0^+$.

\begin{Theorem}\label{t:reso}
Let us assume that \eqref{eq:condsigma}, \eqref{eq:p}, \eqref{eq:p2}
hold and $p(x)=0$ for all $x\in 
B^-_3 \cup \mathcal C_\eps \cup B^+_3$. Let 
\[
\lambda_{\bar k}(D^-\cup D^+)=\lambda_{\bar k+1}(D^-\cup D^+)\in
\sigma_p(D^-)\cap\sigma_p(D^+)
\]
 be 
a simple eigenvalue on $D^-$  with
    corresponding eigenfunctions having in ${\mathbf 0}$ a zero of
    order $1$, a simple eigenvalue on $D^+$  with
    corresponding eigenfunctions having in ${\mathbf e}_1$ a zero of
    order $1$, and a
double eigenvalue on $D^-\cup D^+$. 
Let $u_0^+\in{\mathcal
    D}^{1,2}(D^+)$ and $u_0^-\in{\mathcal
    D}^{1,2}(D^-)$ be the  eigenfunctions 
associated to 
$\lambda_{\bar k}(D^-\cup D^+)=\lambda_{\bar k+1}(D^-\cup D^+)$ on $D^+$ and $D^-$
respectively satisfying 
\begin{equation*}
 \int_{D^+}p(x)|u_0^+(x)|^2\,dx=\int_{D^-}p(x)|u_0^-(x)|^2\,dx=1.
\end{equation*}
If 
\begin{equation}\label{eq:61}
 \abs{\dfrac{\partial u_0^+}{\partial x_1}(\mathbf e_1)} >
 \abs{\dfrac{\partial u_0^-}{\partial x_1}(\mathbf 0)}, 
\end{equation}
then 
\begin{equation*}
\lim_{\eps\to 0^+}\frac{\lambda_{\bar k}(D^-\cup D^+)
-\lambda_{\bar
    k}(\Omega^\eps)}{\eps^N}=
\mathfrak C(\Sigma) \bigg(\dfrac{\partial u_0^+}{\partial
  x_1}(\mathbf e_1)\bigg)^{\!\!2},
\end{equation*}
and 
\begin{equation*}
\lim_{\eps\to 0^+}\frac{\lambda_{\bar k+1}(D^-\cup D^+)-\lambda_{\bar
    k+1}(\Omega^\eps)}{\eps^N}=
\mathfrak C(\Sigma) \bigg(\dfrac{\partial u_0^-}{\partial
  x_1}(\mathbf e_1)\bigg)^{\!\!2},
\end{equation*}
where $\Omega^\eps$ is defined in \eqref{eq:31} and $\mathfrak C(\Sigma)$ is defined in \eqref{def_compliance}.
\end{Theorem}

In section \ref{sec:resonant-case}, we also prove that, in the
resonant case, under condition \eqref{eq:61}, each
approximating eigenfunction is localized as $\eps\to 0^+$ on the
corresponding component of the limit domain, i.e. an asymmetrical
limit configuration prevents dumbbell eigenfunctions  from spreading
their mass over both components and forces them to concentrate in one
of the two regions.

\begin{Theorem}\label{t:reso_autofun}
Under the same assumptions and with the same notations of Theorem
\ref{t:reso},
there exist two continuously parametrized  families 
$v_{\bar k}^\eps, v_{\bar k+1}^\eps  \in{\mathcal
    D}^{1,2}(\Omega^\eps)$ of eigenfunctions  on $\Omega^\eps$, i.e. soutions to
    \begin{equation}\label{problemabark}
\begin{cases}
-\Delta v^\eps=\lambda p v^\eps,&\text{in }\Omega^\eps,\\
v^\eps =0,&\text{on }\partial \Omega^\eps,\\
\int_{\Omega^\eps}p(x) |v^\eps(x)|^2\,dx=1,
\end{cases}
\end{equation}

for $\lambda=\lambda_{\bar k}(\Omega^\eps)$ and $\lambda=\lambda_{\bar k+1}(\Omega^\eps)$ respectively, such that 
\begin{equation}
v_{\bar k}^\eps\to
u_0^+ \quad\text{and}\quad  v_{\bar k+1}^\eps\to
u_0^-\text{ in }{\mathcal D}^{1,2}(\R^N)
\text{ as }\eps\to0^+.
\end{equation}
%and 
%a continuously paramterized  family of eigenfunctions 
%$v_{\bar k+1}^\eps \in{\mathcal
%    D}^{1,2}(\Omega^\eps)$ on $\Omega^\eps$
%associated to 
%$\lambda_{\bar k+1}(\Omega^\eps)$ such that 
%\begin{equation}\label{problemabark+1}
%\begin{cases}
%-\Delta v_{\bar k+1}^\eps=\lambda_{\bar k+1}(\Omega^\eps) p v_{\bar k+1}^\eps,&\text{in }\Omega^\eps,\\
%v_{\bar k+1}^\eps =0,&\text{on }\partial \Omega^\eps,\\
%\int_{\Omega^\eps}p(x) |v_{\bar k+1}^\eps(x)|^2\,dx=1,
%\end{cases}\quad\text{and}\quad  v_{\bar k+1}^\eps\to
%u_0^-\text{ in }{\mathcal D}^{1,2}(\R^N)
%\text{ as }\eps\to0^+.
%\end{equation}
Moreover, for $v^\eps=v_{\bar k}^\eps$ and $v^\eps= v_{\bar k+1}^\eps$ there holds
\begin{equation}\label{eq:62}
\int_{D^*\cup([0,1/8)\times(\eps\Sigma))}|\nabla v^\eps|^2dx=O\big
  (\eps^{-(N+1)}e^{-\frac{\sqrt{\lambda_1(\Sigma)}}{16\eps}}\big),\quad \text{as }\eps\to 0^+.
\end{equation}
%and 
%\begin{equation}\label{eq:7}
%\int_{D^+\cup((7/8,1]\times(\eps\Sigma))}|\nabla v_{\bar
%  k+1}^\eps|^2dx=O\big
%  (\eps^{-(N+1)}e^{-\frac{\sqrt{\lambda_1(\Sigma)}}{16\eps}}\big),\quad \text{as }\eps\to 0^+,
%\end{equation}
where $D^*=D^-$ and $D^*=D^+$ respectively, and $\lambda_1(\Sigma)$ denotes the first eigenvalue of the Laplace
 operator on $\Sigma$ under null Dirichlet boundary conditions.
\end{Theorem}

For the two families of eigenfunctions $v_{\bar k}^\eps, v_{\bar k+1}^\eps$ we provide a sharp
asymptotics, extending the result of Theorem~\ref{teo_asintotico_autofunzioni} in the resonant asymmetrical case.

\begin{Theorem}\label{t:reso_autofun_conv}
Under the same assumptions and with the same notations of Theorem
\ref{t:reso}, let $v_{\bar k}^\eps, v_{\bar k+1}^\eps \in{\mathcal
    D}^{1,2}(\Omega^\eps)$ be as in Theorem \ref{t:reso_autofun}. Then 
\begin{align}
\label{eq:6}&\lim_{\eps\to 0^+}
\eps^{-N}\|v_{\bar k}^\eps -u_0^+\|^2_{\Di{\R^N}}
=
\bigg(\frac{\partial u_0^+}{\partial x_1}(\e)\bigg)^{\!\!2}
\mathfrak C(\Sigma),\\
\label{eq:16}&\lim_{\eps\to 0^+}
\eps^{-N}\|v_{\bar k+1}^\eps -u_0^-\|^2_{\Di{\R^N}}
=
\bigg(\frac{\partial u_0^-}{\partial x_1}(\e)\bigg)^{\!\!2}
\mathfrak C(\Sigma),
\end{align}
where $v_{\bar k}^\eps, v_{\bar k+1}^\eps , u_0^+$, and $u_0^-$  are trivially extended to the whole $\R^N$
and $\mathfrak C(\Sigma)$ is defined in \eqref{def_compliance}.
\end{Theorem}

\subsection{Motivations and references to the literature}\label{sec:motiv-refer-liter}

The continuity of  eigenvalues and eigenfunctions of the Laplace
operator under Dirichlet boundary conditions in varying domains
including the dumbbell case has
been studied in  \cite{bucur2006,daners}. We also refer to \cite{BV} for
a first result about spectral continuity for less general domain's
perturbations and to \cite{hale} (and references therein) for a detailed survey.

As far as estimates
of the rate of convergence are concerned, we mention \cite{HM}, where, among other results, the authors
prove that, in the case of a Helmholtz resonator with a cavity, 
the effect of adding a tubular region with a section of radius of order $\eps$ is to shift the
eigenvalues by a small amount of order at most $\eps^{1/2}$. This generalizes a
previous result of \cite{Ar} where an
$\eps^{1/2}$-rate of convergence for resonances of a Helmholtz
resonator was obtained in dimension $3$. 
We stress that  the case treated in present paper does not allow continuos spectrum for the Dirichlet Laplacian.
As far as we know, no sharp estimates similar to ours can be found in the literature. 
Similar to our settings, we mention  
 \cite{taylor} which contains  an $\eps^a$-bound from above for the a rate of convergence, but not 
the exact asymptotics.
Some other estimates on the rate of convergence of Dirichlet eigenvalues for different domain's
perturbations can be found
in \cite{davies,pang}.

We note that there exists an extensive literature
dealing with Neumann boundary conditions, but, in the case of  dumbbell
domains with thin handles,   the eigenvalues of the Laplacian
may not be continuous, as observed in \cite{arrietaJDE,CH,jimbo} (see
also \cite{nazarov}).

Spectral analysis in thin branching domains arises naturally in the
study of models of propagation of waves in quasi one-dimensional
systems: in this framework we meet the theory of quantum graphs which
provide simplified models of quantum wires, photonic crystals, carbon
nano-structures, thin waveguides and many other problems,
see e.g. \cite{BK,kuchment} for details.
Similarly as in quantum graph theory, in this paper we address to
systems which are composed by different chambers communicating by
connecting regions and which are governed by certain differential
equations.  We mention that, besides their own theoretical interest in
the framework of spectral theory for elliptic operators, such issues
are also related to some engineering problems: elasticity problems in
heterogeneous materials and  limit problems at
the junctions of several domains with different limit dimensions
(namely thin plates with beams or rods),  see e.g. \cite{ciarlet}.

\section{Preliminaries, notation and technical lemmas}\label{sec:prel-notat-techn}

The proof of Theorem \ref{teo_asintotico_autovalori}
is based on the Courant-Fisher \emph{minimax characterization} of
eigenvalues and some estimates from above and below on the
associated Rayleigh quotient computed at suitable test
functions. In this section we introduce the proper test functions on
which the Rayleigh quotient will be estimated to prove upper/lower
bounds, and prove some properties (i.e. point-wise estimates, blow-up
analysis) of such test functions and of eigenfunctions on the domain~$\Omega^\eps$.

\subsection{Transition functions }

We start by  introducing some functions
describing  the domain's change of geometry at the junction, which will
be used for the construction of super-solutions needed for deriving
point-wise estimates on eigenfunctions  and for estimating the
Rayleigh quotient associated to the eigenvalue problem. More
precisely, we consider
\begin{itemize}\itemsep4pt
\item the unique function $\Phi$  which is harmonic in the domain $\widetilde D$,
 has finite energy in
  $T_1^-$, and behaves as $(x_1-1)^+$ as $|x-{\mathbf e}_1|\to+\infty$
  in $D^+$ (here $s^+=\max\{s,0\}$ denotes the positive part of $s$
  for all $s\in\R$);
\item for every $R>2$, the function $z_R$ defined as  the harmonic
  extension  of $\Phi\big|_{\Gamma_R^+}$ in the domain $B_R^+$
  vanishing on $\partial B_R^+\cap \partial D^+$;
\item for every $R>2$, the function $v_R$ defined as  the harmonic
  extension  of $(x_1-1)^+\big|_{\Gamma_R^+}$ in the domain $T_1^-\cup
  B_R^+$
  vanishing on $\partial (T_1^-\cup D^+)$.
\end{itemize}
For all $R>1$, we denote as
$\mathcal H_R$  the completion of
$C^\infty_{\rm c}\big(\big((-\infty,1)\times\R^{N-1}\big)\cup \overline{B_R^+}\big)$
with respect to the norm $\big(\int_{((-\infty,1]\times\R^{N-1})\cup
  B_R^+}|\nabla v|^2dx\big)^{1/2}$, i.e.  $\mathcal H_R$ is the space of functions with
finite energy in $((-\infty,1]\times\R^{N-1})\cup \overline{B_R^+}$ vanishing on
$\{(1,x')\in \R\times\R^{N-1}:|x'|\geq R\}$.

In the sequel, we also denote as  $\lambda_1(\Sigma)$ the first eigenvalue of the Laplace
 operator on $\Sigma$ under null Dirichlet boundary conditions, and as
 $\psi_1^\Sigma(x')$ the corresponding positive
$L^2(\Sigma)$-normalized eigenfunction, so that 
\[
 \begin{cases}
 -\Delta_{x'}\psi_1^\Sigma(x')=
 \lambda_1(\Sigma)\psi_1^\Sigma(x'),&\text{in }\Sigma,\\
 \psi_1^\Sigma=0,&\text{on }\partial\Sigma,
 \end{cases}
\]
 being $\Delta_{x'}=\sum_{j=2}^N\frac{\partial^2}{\partial x_j^2}$,
 $x'=(x_2,\dots,x_N)$.

\smallskip \subsubsection{The function $\Phi$.}
In \cite[Lemma 2.4]{FT12}, it is proved that there exists a unique
function $\Phi$ satisfying
\begin{equation}\label{eq_Phi_1}
\begin{cases}
\int_{T_1^-\cup B^+_{R}}\Big(|\nabla \Phi(x)|^2
+|\Phi(x)|^{2^*}\Big)
\,dx<+\infty\text{ for all }R>1,\\[5pt]
-\Delta \Phi=0\text{ in a distributional sense in }\widetilde D,
\quad \Phi=0\text{ on }\partial \widetilde D,\\[5pt]
\int_{D^+}|\nabla (\Phi-(x_1-1))(x)|^2\,dx<+\infty.
\end{cases}
\end{equation}
Furthermore $\Phi>(x_1-1)^+$ in $\widetilde D$ (by the Strong
Maximum Principle) and, by \cite[Lemma 2.9]{FT12}, there holds
\begin{align}\label{eq:Phiinfty}
  &\Phi(x)=(x_1-1)^++O(|x-{\mathbf e}_1|^{1-N})\quad\text{in $D^+$ as
  }|x-{\mathbf e}_1|\to +\infty,\\
\label{eq:phiminusinfty}&\Phi(x)=O(e^{\sqrt{\lambda_1(\Sigma)}\frac{x_1}{2}})\quad\text{as
  }x_1\to-\infty\text{ uniformly with respect to }x'\in\Sigma.
\end{align}
Let
\begin{align*}
  \Psi: {\mathbb S}^{N-1}\to\R, \quad \Psi (\theta_1,\theta_2,\dots,\theta_N)=\frac{\theta_1}{\Upsilon_N},
\end{align*}
being ${\mathbb S}^{N-1}=\{(\theta_1,\theta_2,\dots,\theta_N)\in
\R^N:\sum_{i=1}^N\theta_i^2=1\}$ the unit $(N-1)$-dimensional sphere
and
\begin{align}\label{eq:upsilonN}
\Upsilon_N=\sqrt{\tfrac12{\textstyle{\int}}_{{\mathbb
        S}^{N-1}}\theta_1^2d\sigma(\theta)}=\sqrt{\frac{\omega_{N-1}}{2N}},
\end{align}
where $\omega_{N-1}$ denotes the volume of the unit sphere ${\mathbb
  S}^{N-1}$, i.e. $\omega_{N-1}=\int_{{\mathbb S}^{N-1}}d\sigma(\theta)$. 
Here
and in the sequel, the notation $d\sigma$ is used to denote the volume
element on $(N-1)$-dimensional surfaces.  
We
notice that, letting  
$${\mathbb
  S}^{N-1}_+:=\{\theta=(\theta_1,\theta_2,\dots,\theta_N)\in{\mathbb
  S}^{N-1}:\theta_1>0\},
$$
$\Psi^+=\Psi\big|_{{\mathbb
  S}^{N-1}_+}=\frac{\theta_1}{\Upsilon_N}$ is the
first positive $L^2({\mathbb S}^{N-1}_+)$-normalized eigenfunction of
$-\Delta_{{\mathbb S}^{N-1}}$ on ${\mathbb S}^{N-1}_+$ under null
Dirichlet boundary conditions satisfying
\begin{equation}\label{eq:eigenpsi+}
-\Delta_{{\mathbb
    S}^{N-1}}\Psi^+=(N-1) \Psi^+\quad\text{on }{\mathbb S}^{N-1}_+.
\end{equation}

\begin{Lemma}\label{lemma_Phi_1}
 Let $\Phi$, $\Psi^+=\frac{\theta_1}{\Upsilon_N}$, and $\Upsilon_N$ be
  as in \eqref{eq_Phi_1}, \eqref{eq:eigenpsi+}, and
  \eqref{eq:upsilonN} respectively, and, for every $r>1$, let us define
  \begin{equation}\label{eq:varphi}
\varphi(r)=
 \int_{\SN_+}\Phi(\mathbf e_1+r\theta)\Psi^+(\theta)\,d\sigma(\theta).  
  \end{equation}
Then
\begin{align}
&\tag{i} \varphi(r)=\varphi(1)r^{1-N}+\Upsilon_N(r-r^{1-N}),
\quad\text{for every $r>1$};\\
&\tag{ii}\int_{\Gamma_r^+}\frac{\partial
  \Phi}{\partial\nu}(x_1-1)\,d\sigma=\Upsilon_N r^N\bigg(N\Upsilon_N+(1-N)\frac{\varphi(r)}{r}\bigg) ,
\quad\text{for every $r>1$}.
\end{align}
\end{Lemma}
\begin{pf}
Part (i) is proved in \cite[Lemma 2.2]{AFT12}. To prove (ii) we
observe that 
\[
\int_{\Gamma_r^+}\frac{\partial
  \Phi}{\partial\nu}(x_1-1)\,d\sigma=\Upsilon_N r^N \varphi'(r)
\]
so that the thesis immediately follows from differentiation of (i) and
simple calculations.  
\end{pf}

\begin{Lemma}\label{l:secondo_lemma_Phi_1}
 Let $\Phi$ as in \eqref{eq_Phi_1} and $J_\Sigma:{\mathcal
   D}^{1,2}(\widetilde D)\to\R$  as in \eqref{eq:Jsigma}. 
Then 
\begin{align*}
{\rm (i)}\quad& r^{N-1}\Big(\Phi({\mathbf e}_1+r\theta)-r\theta_1\Big)\to
\frac{1}{\Upsilon_N^2}\bigg(\int_{{\mathbb S}^{N-1}_+}(\Phi({\mathbf
  e}_1+\theta)-\theta_1)\theta_1\,d\sigma\bigg)\theta_1\\
& \text{ in $C^{1,\alpha}({\mathbb S}^{N-1}_+)$ as $r\to+\infty$, 
 for every $\alpha\in(0,1)$;}\\[5pt]
{\rm (ii)}\quad&  J_\Sigma (\Phi-(x_1-1)^+)=\min_{w\in \mathcal
  D^{1,2}(\widetilde D)} J_\Sigma(w);\\
{\rm (iii)}\quad&  \int_{\widetilde
  D}\nabla(\Phi-(x_1-1)^+)\cdot\nabla v(x)\,dx=\int_\Sigma
v(1,x')\,dx',\quad\text{for every }v\in \mathcal
  D^{1,2}(\widetilde D).
\end{align*}
\end{Lemma}
\begin{pf}
Let $w:D^+\to\R$, $w(x)=\Phi(x)-(x_1-1)$. From \eqref{eq_Phi_1} we
have that $-\Delta w=0$ in $D^+$, $w>0$ in $D^+$, $w=0$ on $\{(1,x'):|x'|>1\}$, and
$\int_{D^+}|\nabla w(x)|^2\,dx<+\infty$. Then, (i) follows
from \cite[Theorem 1.5]{FFT} applied to the function $w$. 

Statements (ii) and (iii) are  contained in \cite[Lemma 2.4]{FT12} and \cite{AT12}.
\end{pf}

\noindent As a consequence of the previous lemma, it is possible to characterize 
the minimum $m(\Sigma)$ defined in \eqref{eq:m(sigma)} in terms of the function $\Phi$.

\begin{Corollary}\label{cor:min_phi1}
Let $J_\Sigma$ as in \eqref{eq:Jsigma}, $m(\Sigma)$ as in
\eqref{eq:m(sigma)}, and $\Phi$ as in \eqref{eq_Phi_1}. Then 
\[
m(\Sigma)=-\frac12 \int_\Sigma \Phi(1,x')\,dx'.
\]  
\end{Corollary}
\begin{pf}
From Lemma \ref{l:secondo_lemma_Phi_1}(ii), we have that 
\[
m(\Sigma)=   \frac12\int_{\widetilde
    D}|\nabla 
(\Phi-(x_1-1)^+)|^2\,dx-\int_{\Sigma}
 \Phi(1,x')\,dx'
\]
and the conclusion follows taking $v=\Phi-(x_1-1)^+$ in  Lemma \ref{l:secondo_lemma_Phi_1}(iii).
\end{pf}

\smallskip \subsubsection{The function $z_R$.}
For every $R>1$, we denote as $z_R$ the
unique solution to the minimization problem
\begin{equation*}
  \int_{B_{R}^+}|\nabla z_R|^2\,dx 
\\=\min\bigg\{
\int_{B_{R}^+}|\nabla v|^2\,dx:
v\in  H^1(B_{R}^+),\ 
v=0\text{ on }\partial D^+,\text{ and }v=\Phi\text{ on }\Gamma^+_{R}\bigg\},
\end{equation*}
which then solves
\begin{equation}\label{eq:z_R}
 \begin{cases}
  -\Delta z_R =0, &\text{ in } B_{R}^+,\\
  z_R=\Phi, &\text{ on }\Gamma_{R}^+,\\
  z_R = 0, &\text{ on }\partial D^+,
 \end{cases}
\end{equation}

\begin{Lemma}\label{lemma_z_R}
Let $z_R$, $\Psi^+=\frac{\theta_1}{\Upsilon_N}$, and $\Upsilon_N$ be
  as in \eqref{eq:z_R}, \eqref{eq:eigenpsi+} and
  \eqref{eq:upsilonN} respectively, and, for every $R>2$ and $r\in(0,R]$, let us define
\begin{equation}\label{eq:phi}
\phi_R(r)=
 \int_{\SN_+}z_R(\mathbf e_1+r\theta)\Psi^+(\theta)\,d\sigma(\theta).  
  \end{equation}
Then 
\begin{equation*}
 \int_{\Gamma_r^+}\frac{\partial
  z_R}{\partial\nu}(x_1-1)\,d\sigma=\Upsilon_N r^N \dfrac{\phi_R(R)}{R},
\quad\text{for every $r\in(0,R]$}.
\end{equation*}
\end{Lemma}
\begin{pf}
We first observe that 
\begin{equation}\label{eq:23}
 \phi_R'(r)=\dfrac{1}{\Upsilon_N r^N} \int_{\Gamma_r^+}\frac{\partial
  z_R}{\partial\nu}(x_1-1)\,d\sigma,\quad\text{for all }r\in(0,R]. 
\end{equation}
Since $z_R$ is harmonic in $B_R^+$, there exists $C_R\in\R$ such that 
$\big(\frac{\phi_R(r)}{r}\big)'=\frac{C_R}{r^{N+1}}$ in $(0,R]$, so
that
\begin{equation*}
\phi_R(r)= r \frac{\phi_R(R)}{R} + \frac{C_R}{N}r R^{-N}
-\frac{C_R}{N}r^{1-N}, \quad \text{for every $r\in(0,R]$.}  
\end{equation*}
From the regularity of $z_R$ in ${\mathbf e}_1$, we deduce $C_R=0$ and then 
$\phi_R(r)=r\frac{\phi_R(R)}{R}$. Hence 
\begin{equation}\label{eq:22} 
\phi_R'(r)=\frac{\phi_R(R)}{R},\quad\text{for every }r\in (0,R],
\end{equation}
The thesis follows from \eqref{eq:23} and \eqref{eq:22}.
\end{pf}

\smallskip \subsubsection{The function $v_R$.}
For every $R>1$, we denote as $v_R$ 
 the
unique solution to the minimization problem
\[
\int_{T_1^-\cup B_R^+}|\nabla v_R|^2\,dx 
=\min\bigg\{
\int_{T_1^-\cup B_R^+}|\nabla v|^2\,dx:
v\in  \mathcal H_{0,R}\text{ and }v=(x_1-1)\text{ on }\Gamma^+_{R}\bigg\},
\]
where
$\mathcal H_{0,R}$ is the completion of
$\big\{v\in C^\infty_{\rm c}\big(\,\overline{T_1^-\cup B_R^+ }\,\big):
v=0\text{ on }\partial(T_1^-\cup D^+)\big\}$ with respect to the norm
$\big(\int_{T_1^-\cup B_R^+}|\nabla v|^2dx\big)^{1/2}$. This function then solves
\begin{equation}\label{eq:v_R}
 \begin{cases}
  -\Delta v_R =0, &\text{ in } T_1^-\cup B_R^+,\\
  v_R=(x_1-1), &\text{ on }\Gamma_{R}^+,\\
  v_R = 0, &\text{ on }\partial(T_1^-\cup D^+).
 \end{cases}
\end{equation}

\begin{Lemma}\label{lemma_v_R}
  For every $R>2$, let $v_R$, $\Psi^+=\frac{\theta_1}{\Upsilon_N}$, and $\Upsilon_N$ be
  as in \eqref{eq:v_R}, \eqref{eq:eigenpsi+}, and
  \eqref{eq:upsilonN} respectively, and, for every $r\in(1,R]$, let us define
  \begin{equation*}
\chi_R(r)=
 \int_{\SN_+}v_R(\mathbf e_1+r\theta)\Psi^+(\theta)\,d\sigma(\theta).  
  \end{equation*}
Then
\begin{align}
 &\tag{i} v_R \to \Phi \quad \text{as $R\to+\infty$ in $\mathcal H_t$ for all $t>2$;} \\
 &\tag{ii} \int_{\Gamma_R^+}\frac{\partial
  v_R}{\partial\nu}(x_1-1)\,d\sigma=\Upsilon_N \,
\frac{\Upsilon_N (R^N+N-1)-N\chi_R(1)}{1-R^{-N}}.
\end{align}
\end{Lemma}
\begin{proof}
To prove  (i), for any $t<R$, we estimate
\begin{align*}
 \int_{T_1^- \cup B_t^+} \abs{\nabla (v_R - \Phi)}^2dx
&\leq \int_{T_1^- \cup B_R^+} \abs{\nabla (v_R - \Phi)}^2\,dx \\
&\leq \int_{T_1^- \cup B_R^+} \abs{\nabla (\eta_R(x_1-1 - \Phi))}^2 \,dx
\end{align*}
via the Dirichlet Principle since $v_R - \Phi$ is harmonic in $T_1^-
\cup B_R^+$ and $(v_R - \Phi)\big|_{\Gamma_R^+}=(x_1-1)-\Phi
\big|_{\Gamma_R^+}$, being $\eta_R$ a smooth cut-off function such
that 
\begin{equation}\label{eq:cutoff}
\eta_R\equiv 0 \text{ in }T_1^- \cup B_{R/2}^+,\quad
\eta\equiv 1 \text{ in }D^+\setminus B_R^+,\quad
0\leq\eta_R\leq1, \quad|\nabla \eta_R|\leq \frac 4R
\text{ in }B_R^+
\setminus B_{R/2}^+. 
\end{equation}
In view of  \eqref{eq_Phi_1} and \eqref{eq:Phiinfty}
\begin{align}\label{eq:25}
  \int_{T_1^- \cup B_R^+} &\abs{\nabla (\eta_R(x_1-1 - \Phi))}^2 dx\\
 \notag &\leq 2 \int_{B_R^+\setminus B_{R/2}^+ } \abs{\nabla \eta_R}^2(x_1-1 - \Phi)^2dx 
+ 2 \int_{D^+\setminus B_{R/2}^+} \eta_R^2 \abs{\nabla(x_1-1 - \Phi)}^2 dx\\
  &\notag\leq {\rm const\,} R^{-2} R^{2-2N} R^N + o(1) = o(1)
\end{align}
as $R\to+\infty$, thus proving (i).

 To prove (ii), we observe that 
\begin{equation}\label{eq:21}
\int_{\Gamma_r^+}\frac{\partial
  v_R}{\partial\nu}(x_1-1)\,d\sigma
= r^N\Upsilon_N \chi_R'(r), \quad \text{for every }r\in(1,R]. 
\end{equation}
From \eqref{eq:v_R}, $\chi_R(r)$ solves the equation
$\big( r^{N+1}\big(\frac{\chi_R(r)}{r}\big)' \big)'=0$ in the interval
$(1,R]$, hence by integration we obtain that there exists $C_R\in\R$
such that 
\[
\frac{\chi_R(r)}{r} - \chi_R(1)=\frac{C_R}{N}
(1-r^{-N}),\quad\text{for every }r\in(1,R].
\]
Replacing $r=R$ in the above identity and observing that the boundary
condition in \eqref{eq:v_R} implies that 
\[
\frac{\chi_R(R)}R=\Upsilon_N,
\]
we obtain that 
\[
 \frac{C_R}{N} = \dfrac{1}{1-R^{-N}} \left( \Upsilon_N -
   \chi_R(1) \right).
\]
Hence 
\begin{equation*}
\chi_R(r) = r\,\frac{\Upsilon_N-\chi_R(1) R^{-N}}{1-R^{-N}}
- \frac{\Upsilon_N - \chi_R(1)}{1-R^{-N}}\,r^{1-N}, \quad r\in(1,R],
\end{equation*}
and then
\begin{equation}\label{eq:24}
 \chi_R'(r) = \frac{\Upsilon_N-\chi_R(1) R^{-N}}{1-R^{-N}}
+ \frac{(N-1)(\Upsilon_N - \chi_R(1))}{1-R^{-N}}\,r^{-N}, \quad r\in(1,R].
\end{equation}
The conclusion follows by plugging $r=R$ in \eqref{eq:21} and \eqref{eq:24}. 
\end{proof}

\subsection{Point-wise and energy control for eigenfunctions on the
  varying domain}\label{sec:point-wise-energy}
In order to prove Theorem \ref{teo_asintotico_autovalori}, quite
precise decaying estimates of eigenfunctions on the varying domain
are needed.  In this subsection we pursue this analysis.

\begin{Lemma}\label{controllo_supersol}
Let $j\in\N$ and, for all $\eps\in(0,1)$, let $v^\eps\in \mathcal D^{1,2}(\Omega^\eps)$ solve 
 \begin{equation}\label{eq:equagenereigenf}
\begin{cases}
-\Delta v_\eps=\lambda_j(\Omega^\eps) p v_\eps,&\text{in }\Omega^\eps,\\
v_\eps=0,&\text{on }\partial \Omega^\eps,\\
\int_{\Omega^\eps}pv_\eps^2\,dx=1.
\end{cases}
\end{equation}
\begin{enumerate}[\rm (i)]
\item For every sequence $\eps_n \to 0$ there exist a subsequence $\eps_{n_k}$ 
and $v_0\in \Di{D^-\cup D^+}$ solving 
 \begin{equation*}
\begin{cases}
-\Delta v_0=\lambda_j(D^+\cup D^-) p v_0,&\text{in }D^+\cup D^-,\\
v_0=0,&\text{on }\partial (D^+\cup D^-),\\
\int_{D^-\cup D^+} p v_0^2\, dx=1,
\end{cases}
\end{equation*}
 such that 
$ v_{\eps_{n_k}} \to v_0$ in $\Di{\R^N}$.
\item There exists $\eps_0\in(0,1)$ and $C_1,C_2>0$ such that, for all
  $\eps\in(0,\eps_0)$ and $R,R_1,R_2>1$ with $R_1>R_2$, there holds
\begin{align}
& \label{vjeps_unif_lim} |v_\eps(x)|\leq C_1,\quad \text{for all }x\in\Omega^\eps,\\
 & \label{vjeps_bucur} \lim_{\eps\to 0^+}\sup_{B_{R\eps}^+ \cup
   B_{R\eps}^-\cup \mathcal C_\eps} |v_{\eps}| = 0, \\
 & \label{vjeps_grad_bucur} \sup_{(B_{R_1\eps}^+ \setminus B_{R_2\eps}^+) 
\cup (B_{R_1\eps}^- \setminus B_{R_2\eps}^-)} 
\abs{\nabla v_{\eps}} = O(1/\eps), \quad \text{as }\eps\to 0^+, \\
& \label{vjeps_soprasol} |v_{\eps}(x)|\leq C_2 \Big(\sup_{\partial \mathcal C_{\eps}}|v_\eps|\Big) 
\left( e^{-\frac{\sqrt{\lambda_1(\Sigma)}}{4{\eps}} x_1} + e^{-\frac{\sqrt{\lambda_1(\Sigma)}}{4{\eps}} (1-x_1)}\right), 
\quad \text{for all }x\in\mathcal C_{\eps}.
\end{align}
\end{enumerate}
\end{Lemma}
\begin{pf}
From the spectral continuity analyzed in \cite{daners}, 
$\lambda_j(\Omega^\eps)\to \lambda_j(D^+\cup D^-)$; hence the proof of (i)
follows easily from classical compactness argument in view of the
compactness of the map 
 $\Di{\R^N}\to (\Di{\R^N})^\star$, $u\mapsto pu$. 
Estimate \eqref{vjeps_unif_lim} follows by an iterative Brezis-Kato type
argument (see e.g. \cite[Lemma 2.2]{FT12}). 

From \cite[Lemma 5.2]{bucur2006} it follows that solutions to 
\eqref{eq:equagenereigenf} converging in $\Di{\R^N}$ actually converge
in $L^{\infty}_{\rm loc}(\R^N)$; then \cite{bucur2006} and part (i)
imply that for
every sequence $\eps_n \to 0^+$ there exist a subsequence $\eps_{n_k}$ 
 such that $v_{\eps_{n_k}} \to v_0$ in $L^{\infty}_{\rm
   loc}(\R^N)$ and hence, for every $R>1$, 
\[
\lim_{k\to+\infty}\sup_{B_{R\eps_{n_k}}^+ \cup B_{R\eps_{n_k}}^-\cup \mathcal C_{\eps_{n_k}}}
|v_{\eps_{n_k}}| = 0 .
\]
Since the  above
limit  depends neither on the sequence nor on the subsequence, we
deduce the limit as $\eps\to0^+$ thus proving \eqref{vjeps_bucur}.

Estimate \eqref{vjeps_unif_lim}, together with classical elliptic
estimates for $x\mapsto v_\eps (\eps x)$
over an annulus 
$B_{R_1}^- \setminus B_{R_2}^-$ 
and $x\mapsto  v_\eps (\mathbf e_1 +\eps (x-\mathbf e_1)$ over
$B_{R_1\eps}^+ \setminus B_{R_2\eps}^+$,
yield \eqref{vjeps_grad_bucur}.

To prove estimate \eqref{vjeps_soprasol}, let us consider the function
\begin{equation}\label{Psi_eps}
\Psi_{\eps}(x_1,x') = C_{\eps} 
\left( e^{-\frac{\sqrt{\lambda_1(\Sigma)}}{4{\eps}} x_1} + e^{-\frac{\sqrt{\lambda_1(\Sigma)}}{4{\eps}} (1-x_1)}\right) 
\psi_1^\Sigma\Big(\frac{x'}{2{\eps}}\Big)
\end{equation}
where
\begin{equation}\label{Ceps}
 C_{\eps} = \Big( \min_{|y'|\leq 1/2} \psi_1^\Sigma(y') \Big)^{-1} 
\sup_{\partial \mathcal C_{\eps}}\abs{v_{\eps}}.
\end{equation}
We note that $C_{\eps} = o(1)$ as ${\eps}\to 0^+$ in view of estimate \eqref{vjeps_bucur}.
For all $x'\in \eps\Sigma$,
\begin{equation}\label{boundary_cond1} 
\Psi_{\eps}(0,x') \geq C_\eps \psi_{1}^\Sigma \Big(\frac{x'}{2{\eps}}\Big) 
\geq \sup_{\partial \mathcal C_{\eps}}\abs{v_{\eps}}, \quad
 \Psi_{\eps}(1,x') \geq C_\eps \psi_{1}^\Sigma \Big (\frac{x'}{2{\eps}}\Big) 
\geq \sup_{\partial \mathcal C_{\eps}}\abs{v_{\eps}},
\end{equation}
so that $\Psi_\eps\geq|v_\eps|$ on $\partial \mathcal C_\eps$.
Moreover 
\[
-\Delta \Psi_{\eps} =
\frac{3}{16}\frac{\lambda_1(\Sigma)}{{\eps}^2}
\Psi_{\eps},\quad\text{in }\mathcal C_\eps,
\] 
whereas, via Kato's inequality \cite{kato}, 
\[
-\Delta |v_{\eps}| \leq \lambda_j(\Omega^{\eps}) p |v_{\eps}|,\quad\text{in }\mathcal C_\eps,
\] 
so that there exists a constant $c>0$ independent of $\eps$ such that,
for $\eps$ sufficiently small, $\Psi_{\eps} -\abs{v_{\eps}}$ weakly solves 
\begin{equation}\label{eq:46}
 -\Delta (\Psi_{\eps} -\abs{v_{\eps}}) - c (\Psi_{\eps}
 -\abs{v_{\eps}}) \geq 0,\quad\text{in }\mathcal C_\eps. 
\end{equation}
The boundary conditions \eqref{boundary_cond1} imply that $(\Psi_{\eps} -\abs{v_{\eps}})^- =\max\{0,-(\Psi_{\eps} -\abs{v_{\eps}})\}
\in H_0^1(\mathcal C_{\eps})$; hence testing \eqref{eq:46} with 
$-(\Psi_{\eps} -\abs{v_{\eps}})^-$ and using H\"older and Sobolev
inequalities, we obtain that
\begin{align*}
  \int_{\mathcal C_{\eps}} \abs{\nabla(\Psi_{\eps}
    -\abs{v_{\eps}})^-}^2 dx \leq c \int_{\mathcal
    C_{\eps}}\abs{(\Psi_{\eps} -\abs{v_{\eps}})^-}^2
  dx\\
  \leq c \left(\int_{\mathcal C_{\eps}}\abs{(\Psi_{\eps}
      -\abs{v_{\eps}})^-}^{2^*}dx
  \right)^{\!\!2/2^*}\big(\eps^{N-1}|\Sigma|\big)^{2/N}\\
  \leq c \,|\Sigma|^{2/N}\eps^{\frac{2(N-1)}{N}} S^{-1} \int_{\mathcal C_{\eps}} \abs{\nabla (\Psi_{\eps}
    -\abs{v_{\eps}})^-}^2dx 
\end{align*}
where $S$ denotes   the best constant in the Sobolev
inequality
$ S\|u\|_{L^{2^*}(\R^N)}^2\leq \|u\|_{{\mathcal D}^{1,2}(\R^N)}^2$ and 
$|\Sigma|$  is the Lebesgue $(N-1)$-dimensional measure of
$\Sigma$. If $(\Psi_{\eps} -\abs{v_{\eps}})^-\not\equiv
0$, the above estimate would imply that $1\leq c
\,|\Sigma|^{2/N}\eps^{\frac{2(N-1)}{N}} S^{-1}$ for $\eps$ small, thus
giving rise to a contradiction. Then $(\Psi_{\eps}
-\abs{v_{\eps}})^-\equiv 0$ and $|v_\eps|\leq \Psi_{\eps}$ in
$\mathcal C_\eps$, which implies estimate \eqref{vjeps_soprasol}.
\end{pf}

\begin{Corollary}\label{stime_integrali_vjeps}
 Let $v_\eps$ be as in Lemma \ref{controllo_supersol}.
Then for every $\delta \in (0,1/2)$ and  $R>1$ 
\begin{align}
  & \label{stima_puntuale_canale_vjeps} \abs{v_\eps(x_1,x')} =
  O(e^{-\frac{\sqrt{\lambda_1(\Sigma)}}{4\eps}\delta}), \quad \text{as
    $\eps\to 0^+$ uniformly in }(\delta,1-\delta)\times(\eps \Sigma),
  \\ & \label{stima_int_mezzapalletta_vjeps} \int_{B_{R{\eps}}^-} p
  v_{\eps}^2 dx= o({\eps}^N), \quad \text{as $\eps\to
    0^+$}\\ & \label{stima_int_canale_vjeps} \int_{\mathcal C_{\eps}}
  p v_{\eps}^2 dx= o({\eps}^N), \quad \text{as $\eps\to 0^+$}.
\end{align}
\end{Corollary}
\begin{pf}
  Estimate \eqref{stima_puntuale_canale_vjeps} is a a straightforward
  consequence of  \eqref{vjeps_soprasol}, whereas 
\eqref{stima_int_mezzapalletta_vjeps} follows directly from \eqref{vjeps_bucur}.
In order to prove \eqref{stima_int_canale_vjeps}, we observe that
estimates \eqref{vjeps_soprasol} and \eqref{vjeps_bucur} imply that
\begin{align*}
 \int_{\mathcal C_{\eps}} p v_{\eps}^2 dx
&\leq {\rm const\,} \Big(\sup_{\partial \mathcal C_{\eps}}|v_\eps|\Big)^{\!2}\eps^{N-1} \int_0^1 
\left( e^{-\frac{\sqrt{\lambda_1(\Sigma)}}{4{\eps}} x_1} + e^{-\frac{\sqrt{\lambda_1(\Sigma)}}{4{\eps}} (1-x_1)}\right)^{\!2} dx_1 \\
&\leq {\rm const\,} \Big(\sup_{\partial \mathcal C_{\eps}}|v_\eps|\Big)^{\!2}\eps^{N-1}\ {\eps} 
= o({\eps}^N)
\end{align*}
 as ${\eps}\to0^+$. 
\end{pf}

\noindent For $r\in (0,+\infty)\setminus
(1,1+\eps) $, 
we define
\begin{align*}
\Omega_r^\eps&=
\begin{cases}
D^-\cup\{(x_1,x')\in \mathcal C_\eps:x_1<r\},&\text{if }0< r\leq1,\\
D^-\cup\mathcal C_\eps\cup B^+_{r-1},&\text{if } r\geq \eps+1.
\end{cases}
\end{align*}
We recall the following result stated in \cite{FT12}.
\begin{Lemma}\label{lemma_zero}
Let $u_\eps$ be as in \eqref{problema} under the assumptions \eqref{convergenza_u_0}, \eqref{eq:52},
\eqref{eq:53}, \eqref{eq:54}, \eqref{eq:u0}, \eqref{eq:13}.
Then
for every $f\in L^{N/2}(\R^N)$ and $M>0$, there exists
 $\bar\eps_{M,f}>0$ such that for all $r\in(0,1)$ and $\eps\in
  (0,\bar\eps_{M,f})$
$$
\int_{\Omega_{r}^\eps}
|\nabla u_\eps(x)|^2dx\geq M \int_{\Omega_{r}^\eps}
|f(x)| u_\eps^2(x)dx.
$$
\end{Lemma}

\begin{Corollary}\label{motivazione_O_grande}
Let $u_\eps$ be as in Lemma \ref{lemma_zero}.
Then, 
for every $\delta\in (0,1/4)$
\[
\int_{\Omega^\eps_{1-2\delta} } \abs{\nabla u_\eps}^2 dx
= O\Big(\eps^{N-1}e^{-\frac{\sqrt{\lambda_1(\Sigma)}}{2\eps}
  \delta}\Big)\quad\text{as }\eps\to0^+. 
\]
\end{Corollary}
\begin{pf}
 Let $\eta$ be a smooth cut-off function such that $\eta\equiv 1$ in
 $\Omega^\eps_{1-2\delta}$, $\eta\equiv 0$ in
 $\Omega^\eps\setminus \Omega^\eps_{1-\delta}$, $0\leq\eta\leq1$, and
 $|\nabla\eta|\leq 2\delta$. Let
 $w_\eps= \eta u_\eps$. Then, taking into account \eqref{eq:p2},
\begin{multline*}
 \int_{\Omega^\eps_{1-2\delta} } \big(\abs{\nabla u_\eps}^2 -\lambda_{\eps} pu_\eps^2\big)\,dx
\leq \int_{\Omega^\eps_{1-\delta}}\big( \abs{\nabla w_\eps}^2 -\lambda_{\eps} pw_\eps^2\big)\,dx\\
=-\int_{
\Omega^\eps_{1-\delta}\setminus \Omega^\eps_{1-2\delta}}
\eta(\Delta \eta) u_\eps^2\,dx+\frac12
\int_{
\Omega^\eps_{1-\delta}\setminus \Omega^\eps_{1-2\delta}}
u_\eps^2\Delta(\eta^2)\,dx
\end{multline*}
from which the thesis follows invoking estimate
\eqref{stima_puntuale_canale_vjeps}  of Corollary \ref{stime_integrali_vjeps} 
 and Lemma \ref{lemma_zero}.
\end{pf}

\subsection{Point-wise estimates and blow-up analysis of  the test functions}

For every $\eps\in(0,\eps_0)$ and $R>1$, let $\bar u_{\eps,R}$ be the
unique solution to the minimization problem
\begin{equation}\label{eq:min2}
\int_{\Omega^\eps_{1+R\eps}}|\nabla \bar u_{\eps,R}|^2\,dx 
=\min\bigg\{
\int_{\Omega^\eps_{1+R\eps}}|\nabla v|^2\,dx:
v\in  \mathcal H_R^\eps\text{ and }v=u_0\text{ on }\Gamma^+_{R\eps}\bigg\},
\end{equation}
where
$\mathcal H_R^\eps$ is the completion of
$\{v\in C^\infty_{\rm c}\big(\,\overline{\Omega^\eps_{1+R\eps} }\,\big):
v=0\text{ on }\partial\Omega^\eps\}$ with respect to the norm
$\big(\int_{\Omega^\eps_{1+R\eps} }|\nabla v|^2dx\big)^{1/2}$; in
particular $\bar u_{\eps,R}$ solves
\begin{equation}\label{eq:1}
 \begin{cases}
  -\Delta \bar u_{\eps,R} =0, &\text{ in } \Omega^\eps_{1+R\eps},\\
  \bar u_{\eps,R}= u_0, &\text{ on }\Gamma_{R\eps}^+,\\
  \bar u_{\eps,R} = 0, &\text{ on }\partial\Omega^\eps.
 \end{cases}
\end{equation}
In a similar way, 
for every 
$\eps\in(0,\eps_0)$ and $R>1$, we denote as $\bar v_{\eps,R}$  the
unique solution to the minimization problem
\begin{equation}\label{eq:min1}
\int_{B_{R\eps}^+}\!\!|\nabla \bar v_{\eps,R}|^2dx 
=\min\bigg\{
\int_{B_{R\eps}^+}\!\!|\nabla v|^2dx:
v\in  H^1(B_{R\eps}^+),\, 
v=0\text{ on }\partial D^+,\text{ and }v= u_\eps\text{ on }\Gamma^+_{R\eps}\bigg\}.
\end{equation}
In
particular $\bar v_{\eps,R}$ solves
\begin{equation*}
 \begin{cases}
  -\Delta \bar v_{\eps,R} =0, &\text{ in } B_{R\eps}^+,\\
  \bar v_{\eps,R}=u_\eps, &\text{ on }\Gamma_{R\eps}^+,\\
  \bar v_{\eps,R} = 0, &\text{ on }\partial D^+.
 \end{cases}
\end{equation*}
Hence, for every $\eps\in(0,\eps_0)$ and $R>1$, we
define  $\widetilde u_{\eps,R}\in \mathcal D^{1,2}(\Omega^\eps)$ as 
\begin{equation}\label{u_0_entra_nel_canale}
 \widetilde u_{\eps,R}:= 
\begin{cases}
  u_0, &\text{ in $D^+ \setminus B_{R\eps}^+$},\\
  \bar u_{\eps,R}, &\text{ in }\Omega^\eps_{1+R\eps},
 \end{cases}
\end{equation}
and $\widehat {u}_{\eps,R}\in \mathcal D^{1,2}(D^+)$ as
\begin{equation}\label{v_j^eps_zero_al_giunto}
 \widehat {u}_{\eps,R}:= 
\begin{cases}
  u_\eps, &\text{ in $D^+ \setminus B_{R\eps}^+$}\\
  \bar v_{\eps,R},&\text{ in $B_{R\eps}^+$}.
 \end{cases}
\end{equation}
\begin{Lemma}\label{l:stimaeste}
Let $R>1$ and $\widetilde u_{\eps,R}$ be as in
\eqref{u_0_entra_nel_canale}. Then,
 for every $\delta \in(0,1/4)$,
\begin{equation}\label{eq:4}
|\widetilde u_{\eps,R} (x)|=O\Big(e^{-\frac{\sqrt{\lambda_1(\Sigma)}}{4\eps}\delta}\Big)
\end{equation}
as $\eps\to0^+$ uniformly in $\big\{ (x_1,x')\in \mathcal C_\eps: x_1 \in (\delta,1-\delta) \big\}$.
Moreover 
\begin{equation}\label{eq:5}
  \int_{\Omega^\eps_{1-2\delta} } |\nabla \widetilde u_{\eps,R}|^2dx=
 O\Big(\eps^{N-1}e^{-\frac{\sqrt{\lambda_1(\Sigma)}}{2\eps}\delta}\Big)\quad\text{as }\eps\to0^+. 
\end{equation}
\end{Lemma}
\begin{pf}
From \eqref{eq:1} and the maximum principle, it follows that, for
$\eps$ sufficiently small, 
\begin{equation}\label{eq:2}
0\leq   \widetilde u_{\eps,R}(x)\leq
\|u_0\|_{L^{\infty}(D^+)},\quad\text{for all }x\in \Omega^{\eps}_{1+R\eps}.
\end{equation}
We argue as in the proof of estimate \eqref{vjeps_soprasol}
in Lemma \ref{controllo_supersol}. Let us consider the function
\[ 
\widetilde \Psi_{\eps}(x_1,x') = 
\Big( \min_{|y'|\leq 1/2} \psi_1^\Sigma(y') \Big)^{-1} 
\nor{u_0}_{L^\infty} 
\left( e^{-\frac{\sqrt{\lambda_1(\Sigma)}}{4{\eps}} x_1} + e^{-\frac{\sqrt{\lambda_1(\Sigma)}}{4{\eps}} (1-x_1)}\right) 
\psi_1^\Sigma\Big(\frac{x'}{2{\eps}}\Big). 
\]
From \eqref{eq:2} we obtain that, for all $x'\in\eps\Sigma$,
\begin{equation*}
\widetilde \Psi_{\eps}(0,x') 
\geq \nor{u_0}_{L^\infty}
\geq \sup_{\partial \mathcal C_{\eps}}\abs{\widetilde u_{\eps,R}} \quad\text{and}\quad \widetilde \Psi_{\eps}(1,x') 
\geq \nor{u_0}_{L^\infty}
\geq \sup_{\partial \mathcal C_{\eps}}\abs{\widetilde u_{\eps,R}},
\end{equation*}
so that $\widetilde \Psi_\eps\geq \widetilde u_{\eps,R}$ on $\partial \mathcal C_\eps$.
Moreover, 
\[-\Delta \widetilde\Psi_{\eps} =
\frac{3}{16}\frac{\lambda_1(\Sigma)}{{\eps}^2}
\widetilde\Psi_{\eps}\geq 0,\quad\text{in }\mathcal C_\eps,
\] 
whereas $\widetilde u_{\eps,R}$ is nonnegative and harmonic in
$\mathcal C_\eps$, so that 
$-\Delta (\widetilde \Psi_\eps - \widetilde u_{\eps,R}) \geq 0$ on
$\mathcal C_\eps$ and, by the Maximum Principle, we deduce that
\begin{equation*}
 0\leq \widetilde u_{\eps,R}(x) \leq \widetilde \Psi_{\eps}(x), \quad \text{for all }x\in\mathcal C_\eps,
\end{equation*}
from which estimate \eqref{eq:4} follows.
 
 Let $\eta$ be a smooth cut-off function such that $\eta\equiv 1$ in
 $\Omega^\eps_{1-2\delta}$, $\eta\equiv 0$ in
 $\Omega^\eps\setminus \Omega^\eps_{1-\delta}$, $0\leq\eta\leq1$, and
 $|\nabla\eta|\leq 2\delta$. Let
 $w_\eps= \eta \widetilde u_{\eps,R}$. Then
\begin{multline*}
 \int_{\Omega^\eps_{1-2\delta}} |\nabla \widetilde u_{\eps,R}|^2dx
 \leq
 \int_{\Omega^\eps_{1-\delta} } |\nabla w_\eps|^2dx
=-\int_{
\Omega^\eps_{1-\delta}\setminus \Omega^\eps_{1-2\delta}}
\eta(\Delta \eta) \widetilde u_{\eps,R}^2\,dx+\frac12
\int_{
\Omega^\eps_{1-\delta}\setminus \Omega^\eps_{1-2\delta}}
\widetilde u_{\eps,R}^2\Delta(\eta^2)\,dx
\end{multline*}
from which \eqref{eq:5} follows invoking \eqref{eq:4}.
\end{pf}

\noindent For all $R>1$ and $\eps\in(0,\eps_0)$, let us define
\begin{align}
\label{eq:19}&U_\eps(x):=\frac{u_\eps(\e + \eps (x-\e))}{\eps},\qquad  
u_{0,\eps}(x):=\frac{u_0(\e + \eps (x-\e))}{\eps},\\
\label{eq:20}&Z_{\eps}^R(x):=\frac{\widehat {u}_{\eps,R}
(\e + \eps
  (x-\e))}{\eps},\quad 
V_\eps^R:=\dfrac{\widetilde u_{\eps,R}(\e + \eps (x-\e))}{\eps}.
\end{align}

\begin{Lemma}\label{blow_up_vari}
The following convergences hold as $\eps\to 0^+$:
\begin{align}
\label{Phi}
U_\eps\to \big(\tfrac{\partial u_0}{\partial x_1}(\e)\big)\Phi \quad & \text{in
  $\mathcal H_R$ for any $R>2$},\\
\label{x_1} u_{0,\eps} \to \big(\tfrac{\partial u_0}{\partial x_1}(\e)\big) (x_1 -1) \quad & \text{in
  $C^2_{\rm loc}(\overline{D^+})$},\\
\label{z_h} Z_\eps^R\to \big(\tfrac{\partial u_0}{\partial x_1}(\e)\big) z_R \quad & \text{in
  $H^1(B_R^+)$ for any $R>2$},\\
\label{v_k} V_\eps^R\to \big(\tfrac{\partial u_0}{\partial x_1}(\e)\big) v_R \quad & \text{in
  $\mathcal H_R$ for any $R>2$}
\end{align}
with $z_R$ and $v_R$ being as in \eqref{eq:z_R} and \eqref{eq:v_R} respectively.
\end{Lemma}
\begin{pf}
 The convergence \eqref{Phi} follows from \cite[Lemma 4.1 and
 Corollary 4.4]{FT12} and \cite[Lemmas 2.1 and 2.4]{AFT12}.
In order to prove \eqref{x_1}, we notice that
\[
 u_{0,\eps} \to \nabla u_0(\e) \cdot (x-\e)=\frac{\partial u_0}{\partial x_1}(\e) (x_1 -1),
\quad \text{for all }x\in D^+.
\] 
Furthermore, for every $t>0$, 
\[
 \int_{B_t^+} \abs{\nabla u_{0,\eps}}^2dx = \int_{B_t^+} \abs{\nabla u_0(\e+\eps(x-\e))}^2\,dx
= \eps^{-N} \int_{B_{t\eps}^+} \abs{\nabla u_0}^2dx \leq  {\rm
  const\,}t^N
\]
for some ${\rm
  const\,}>0$ independent of $\eps$ and $t$. Then, by a diagonal
process, one can easily prove that,  up to
subsequences, $u_{0,\eps}$ weakly converges in $H^1(B_t^+)$ for all
$t>0$. By elliptic regularity theory we conclude that  $u_{0,\eps}$
converges  to
its point-wise limit in $C^2_{\rm loc}(\overline{D^+})$ (since such a
limit does not depend on the subsequence, the convergence actually
holds as $\eps\to 0^+$).

In order to prove \eqref{z_h}, we notice that $Z_\eps^R - \big(\tfrac{\partial u_0}{\partial x_1}(\e)\big) z_R$ solves
\begin{equation*}
 \begin{cases}
   -\Delta \big(Z_\eps^R - \big(\tfrac{\partial u_0}{\partial x_1}(\e)\big) z_R\big) =0, &\text{in $B_R^+$}\\
   Z_\eps^R - \big(\tfrac{\partial u_0}{\partial x_1}(\e)\big) z_R =
   U_\eps - \big(\tfrac{\partial u_0}{\partial
     x_1}(\e)\big)\Phi, &\text{on $\Gamma_R^+$},
 \end{cases}
\end{equation*}
and, by the Dirichlet principle and \eqref{Phi},
\begin{align*}
 \int_{B_R^+} &\abs{\nabla(Z_\eps^R -\big(\tfrac{\partial u_0}{\partial x_1}(\e)\big) z_R) }^2dx
\leq \int_{B_R^+} \abs{\nabla\Big(\eta
\Big(U_\eps - \big(\tfrac{\partial u_0}{\partial x_1}(\e)\big)\Phi\Big)\Big)}^2 dx\\
&\leq 2 \left(\int_{B_R^+}\abs{\nabla \eta}^2\Big|
U_\eps - \big(\tfrac{\partial u_0}{\partial x_1}(\e)\big)\Phi\Big|^2dx 
+\int_{B_R^+}
\eta^2\Big|\nabla\big(U_\eps - \big(\tfrac{\partial u_0}{\partial x_1}(\e)\big)\Phi\big)\Big|^2dx\right)\\
&= o(1)\quad \text{as }\eps\to 0^+,
\end{align*}
where $\eta$ is a smooth cut-off function such that 
$\eta\equiv 0$ in $B_{R/2}^+$,
$\eta\equiv 1$ in $D^+\setminus B_{R}^+$. Then 
\[
Z_\eps^R \to \big(\tfrac{\partial u_0}{\partial x_1}(\e)\big)z_R
\]
as $\eps\to 0^+$ in
$H^{1}(B_R^+)$ and
convergence \eqref{z_h} is proved. 

In order to prove \eqref{v_k}, we first notice that, in view of \eqref{eq:min2},
\begin{equation*}
\|V_\eps^R\|^2_{\mathcal H_R}=\eps^{-N} \int_{\Omega^\eps_{1+\eps
    R}}|\nabla \widetilde u_{\eps,R}|^2dx =
\eps^{-N} \int_{\Omega^\eps_{1+\eps
    R}}|\nabla \bar u_{\eps,R}|^2dx
\leq \eps^{-N} \int_{B_{\eps R}^+} \abs{\nabla
  u_0}^2 dx\leq {\rm const}
\end{equation*}
some ${\rm const\,}>0$ independent of $\eps$. Then, up to
subsequences, $V_\eps^R \weakly w$ weakly in $\mathcal H_R$ and
strongly in $L^2(\Gamma_R^+)$ as $\eps \to 0^+$ for some $w\in \mathcal
H_{0,R}$ which is is harmonic in $T_1^-\cup B_R^+$. Since, by
\eqref{x_1}, $V_\eps^R\big|_{\Gamma_R^+}=u_{0,\eps}\to
\big(\tfrac{\partial u_0}{\partial x_1}(\e)\big)(x_1-1)$ in
$L^2(\Gamma_R^+)$, we conclude that $w=\big(\tfrac{\partial u_0}{\partial x_1}(\e)\big)v_R$; in particular, since the
weak $\mathcal H_R$-limit of $V_\eps^R$ does not depend on the
subsequence, the convergence actually holds as $\eps\to 0^+$.

Moreover, by standard
interior elliptic estimates, it is easy to prove that the convergence
is strong in $\mathcal H_r$ for every $r\in(1,R)$. 
 In addition, we
can prove that 
\[
V_\eps^R \to \big(\tfrac{\partial u_0}{\partial x_1}(\e)\big)v_R \quad \text{in $H^1(B_R^+\setminus
\overline{B_{R/2}^+})$.}
\]
  Indeed, since $V_\eps^R - \big(\tfrac{\partial u_0}{\partial x_1}(\e)\big)v_R$ is harmonic on
$B_R^+\setminus \overline{B_{R/2}^+}$, then its energy is less or
equal to the energy of any other $H^1$-function with the same boundary
conditions on $\partial (B_R^+\setminus \overline{B_{R/2}^+})$.
In particular, letting $\eta$ be a smooth cut-off function such that 
$\eta\equiv 0$ in $B_{R/2}^+$ and
$\eta\equiv 1$ in $D^+\setminus B_{R}^+$, 
and $\varphi$ be a smooth cut-off function such that 
$\varphi\equiv 1$ in $B_{R/2}^+$ and
$\varphi\equiv 0$ in $D^+\setminus B_{(3/4)R}^+$, in view of
\eqref{x_1} and  $\mathcal H_r$-convergence  for $r\in(1,R)$, we obtain that
\begin{align*}
  & \int_{B_R^+\setminus B_{R/2}^+}|\nabla (V_\eps^R
  -\big(\tfrac{\partial u_0}{\partial x_1}(\e)\big) v_R)|^2\,dx\\
&\leq
  \int_{B_R^+\setminus {B_{R/2}^+}}\Big|\nabla
  \Big(\eta\Big(u_{0,\eps}-\big(\tfrac{\partial u_0}{\partial
    x_1}(\e)\big)(x_1-1)\Big)\Big)+\varphi( V_\eps^R -
  \big(\tfrac{\partial u_0}{\partial x_1}(\e)\big)v_R)\big)\Big|^2\,dx \\
  &\leq 4\int_{B_R^+\setminus
    {B_{R/2}^+}}|\nabla\eta|^2\big|u_{0,\eps} -\big(\tfrac{\partial u_0}{\partial x_1}(\e)\big)(x_1-1)\big|^2dx\\
&\qquad  +4\int_{B_R^+\setminus {B_{R/2}^+}}\eta^2 \big|\nabla
  \big(u_{0,\eps}-\big(\tfrac{\partial u_0}{\partial x_1}(\e)\big)(x_1-1)\big)\big|^2dx\\
  &\qquad + 4\int_{B_R^+\setminus {B_{R/2}^+}}|\nabla\varphi|^2(
  V_\eps^R -\big(\tfrac{\partial u_0}{\partial x_1}(\e)\big) v_R)^2\,dx\\
  &\qquad + 4\int_{B_{(3/4)R}^+\setminus
    {B_{R/2}^+}}\varphi^2|\nabla( V_\eps^R - \big(\tfrac{\partial u_0}{\partial x_1}(\e)\big)
  v_R)|^2\,dx=o(1)\quad\text{as }\eps\to 0^+.
\end{align*}
Hence $V_\eps^R \to \big(\tfrac{\partial u_0}{\partial x_1}(\e)\big) v_R$ in $H^1(B_R^+\setminus
\overline{B_{R/2}^+})$, which, together with 
 $\mathcal H_r$-convergence  for $r\in(1,R)$, implies \eqref{v_k}.~\end{pf}

\begin{remark}\label{rem:hdiv}
  Convergences \eqref{z_h} and \eqref{v_k} together 
  with  the normal trace embedding theorem for $H(\mathop{\rm
    div};\Omega)$ (see e.g. \cite[Chapter 20]{tartar}),
imply that, for all $R>2$, 
\begin{align*}
& \dfrac{\partial Z_\eps^R}{\partial \nu} \to \bigg(\frac{\partial u_0}{\partial x_1}(\e)\bigg)\dfrac{\partial
   z_R}{\partial \nu} \quad 
\text{in $H^{-1/2}(\Gamma_R^+)$}\quad \text{as $\eps\to 0$},\\
& \dfrac{\partial V_\eps^R}{\partial \nu} \to \bigg(\frac{\partial u_0}{\partial x_1}(\e)\bigg)\dfrac{\partial
   v_R}{\partial \nu} \quad 
\text{in $H^{-1/2}(\Gamma_R^+)$}\quad \text{as $\eps\to 0$},
\end{align*}
where  $\nu=\nu(x)=\frac{x}{|x|}$ is  the normal external unit vector to $\Gamma_R^+$.
\end{remark}

\noindent As a straightforward corollary of the blow-up analysis
performed in Lemma \ref{blow_up_vari}, we obtain the
following result, which will play a crucial role in the proof of
Theorem \ref{teo_asintotico_autofunzioni}.
\begin{Corollary}\label{stima_autofunzioni_via_blow_up}
 Under assumptions {\rm (\ref{eq:p}--\ref{eq:54})}, let $u_\eps$ and
 $u_0$ as in {\rm (\ref{problema}--\ref{convergenza_u_0})}
 and {\rm (\ref{eq:u0}--\ref{eq:13})}.
Then
\begin{equation}
\lim_{\eps\to0^+}\dfrac{1}{\eps^N} \int_{\Omega^\eps_{1+R\eps}}
\abs{\nabla (u_\eps - u_0)}^2dx=
\big(\tfrac{\partial
  u_0}{\partial x_1}(\e)\big)^2
\int_{T_1^-\cup B_R^+} \abs{\nabla (\Phi - (x_1-1)^+)}^2dx
\end{equation}
for all $R>2$.
\end{Corollary}
\begin{pf}
The thesis follows from \eqref{Phi} and \eqref{x_1} through a change of variable.~\end{pf}

\section{Proof of Theorem \ref{teo_asintotico_autovalori}}\label{sec:proof-theor-reft}

Let us recall and fix some notation we are going to use throughout this
section.
We recall that $\lambda_\eps=\lambda_{\bar k}(\Omega^\eps)$ denotes
the $\bar k$-th eigenvalue of problem \eqref{eq:30} on the domain
$\Omega^\eps$ and $\lambda_0=\lambda_{\bar k}(D^-\cup D^+)$
denotes the $\bar k$-th eigenvalue on $D^- \cup D^+$ which is equal to the
simple $k_0$-th eigenvalue on $D^+$. Let $u_\eps$ be the eigenfunction on
$\Omega^\eps$ associated to $\lambda_\eps$ satisfying \eqref{problema}
and \eqref{convergenza_u_0}. 

For every 
$j=1,2,\dots,\bar k-1$, we fix an eigenfunction $v_j^\eps\in \mathcal
D^{1,2}(\Omega^\eps)$ associated to $\lambda_j(\Omega^\eps)$  on
$\Omega^\eps$  such that
$\int_{\Omega^\eps}p|v_j^\eps|^2dx =1$
and an eigenfunction $v_j^0\in \mathcal
D^{1,2}(D^-\cup D^+)$ associated to the
eigenvalue $\lambda_j(D^-\cup D^+)$ on $D^-\cup D^+$ such that 
$\int_{D^-\cup D^+}p|v_j^0|^2dx =1$. 
 In particular, we can choose such
eigenfunctions in such a way that
\begin{align*}
&\int_{D^-\cup D^+}\nabla v_j^0\cdot \nabla v_i^0\,dx=0,\quad
\int_{\Omega^\eps}\nabla v_j^\eps\cdot \nabla v_i^\eps\,dx=0,
\quad\text{if
}i\neq j,\ 1\leq i,j\leq \bar k-1,\\
&\int_{D^-\cup D^+}\nabla v_j^0\cdot \nabla u_0\,dx=0,\quad
\int_{\Omega^\eps}\nabla v_j^\eps\cdot \nabla u_\eps\,dx=0,
\quad\text{for all
}1\leq j\leq \bar k-1.
\end{align*}
In the sequel we will denote $\lambda_j(D^-\cup D^+)$ as $\lambda_j^0$
and $\lambda_j(\Omega^\eps)$ as $\lambda_j^\eps$ (we recall that 
the eigenvalues are repeated as many times as their own multiplicity).

The proof of Theorem \ref{teo_asintotico_autovalori} is based on the
following  preliminary result.

\begin{Theorem}\label{t:asintotico_autovalori1}
Under assumptions
\eqref{eq:condsigma}, \eqref{eq:p}, \eqref{eq:p2}, \eqref{eq:53}, and
\eqref{eq:54}, let $\lambda_\eps=\lambda_{\bar k}(\Omega^\eps)$ be 
the $\bar k$-th eigenvalue of problem \eqref{problema}  on
the domain $\Omega^\eps$ defined in \eqref{eq:31} and
$\lambda_0=\lambda_{k_0}(D^+)=\lambda_{\bar k}(D^- \cup D^+)$ be the $\bar k$-th eigenvalue
of problem \eqref{eq:30}
on $D^- \cup D^+$ (which is equal to the simple $k_0$-th eigenvalue on
$D^+$). Then 
\[
\lim_{\eps\to 0^+}\frac{\lambda_0 - \lambda_\eps}{\eps^N}=
\bigg(\frac{\partial u_0}{\partial x_1}(\e)\bigg)^{\!\!2} N
\int_{\SN_+} \left(\Phi(\e+\theta) - \theta_1 \right)\theta_1
\,d\sigma(\theta)\in(0,+\infty),
\]
where $\Phi$  is defined in  
\eqref{eq_Phi_1}.
\end{Theorem}
\begin{pf}
  We observe that a straightforward consequence of the minimax
  principle for eigenvalues is that $\lambda_\eps \leq \lambda_0$. We
  are going to prove first two estimates for the quantity
  $\frac{\lambda_0 - \lambda_\eps}{\eps^N}$, one from below and one
  from above, in order to reach, for every $R>2$,  an estimate of the type
\[
K_1(\eps,R)\leq \frac{\lambda_0 - \lambda_\eps}{\eps^N} \leq  
K_2(\eps,R)
\]
for some constants $K_1(\eps,R),K_2(\eps,R)>0$ depending $\eps$ and $R$; secondly, we will 
prove that 
\[
\lim_{R\to+\infty}\lim_{\eps\to 0^+}K_1(\eps,R)=
\lim_{R\to+\infty}\lim_{\eps\to 0^+}K_2(\eps,R)=\bigg(\frac{\partial u_0}{\partial x_1}(\e)\bigg)^{\!\!2} N
\int_{\SN_+} \left(\Phi(\e+\theta) - \theta_1 \right)\theta_1
\,d\sigma(\theta)
\]
thus implying the stated asymptotics.

\smallskip\noindent
{\bf Step 1:} estimate from below.
From the Courant-Fisher \emph{minimax characterization} of the
Dirichlet eigenvalues, we have that
\begin{equation}
  \lambda_\eps = \min\bigg\{\max_{u\in E\setminus \{0\}}
\dfrac{\int_{\Omega^\eps} \abs{\nabla u}^2dx}{\int_{\Omega^\eps}
  pu^2\,dx}:E \text{ is a subspace of $\Di{\Omega^\eps}$ \text{such
    that }$\dim E=\bar k$}\bigg\}.
\end{equation}
Let $R>2$. If we choose the space $E=\mathop{\rm span} \{ v_1^0,v_2^0,\dots,
v^0_{\bar k-1},\widetilde u_{\eps,R}\}$ (where the functions $v_j^0$
are trivially extended to the whole $\Omega^\eps$), 
we have that $\dim E=\bar k$ and then 
\begin{align*}
 \lambda_\eps &\leq
\max_{\substack{(\alpha_1,\dots, \alpha_{\bar k-1},\beta)\in
    \R^{\bar k}\\
(\sum_{j=1}^{\bar k-1}\alpha_j^2) + \beta^2 =1}}
 \dfrac{\int_{\Omega^\eps} |\nabla (\sum_{j=1}^{\bar k-1}\alpha_j v_j^0 + \beta \widetilde u_{\eps,R})|^2}
{\int_{\Omega^\eps} p(\sum_{j=1}^{\bar k-1}\alpha_j v_j^0 + \beta \widetilde u_{\eps,R})^2}\\
&= \max_{\substack{(\alpha_1,\dots, \alpha_{\bar k-1},\beta)\in
    \R^{\bar k}\\
(\sum_{j=1}^{\bar k-1}\alpha_j^2) + \beta^2 =1}}
\dfrac{\sum_{j=1}^{\bar k-1}\alpha_j^2 \int_{\Omega^\eps} |\nabla v_j^0|^2
+ \beta^2 \int_{\Omega^\eps}\abs{\nabla \widetilde u_{\eps,R}}^2 + 2
\sum_{j=1}^{\bar k-1}\alpha_j\beta
\int_{\Omega^\eps}
\nabla v_j^0\cdot\nabla\widetilde u_{\eps,R}}
{\sum_{j=1}^{\bar k-1}\alpha_j^2\int_{\Omega^\eps} p |v_j^0|^2
+ \beta^2 \int_{\Omega^\eps}p {\widetilde u_{\eps,R}}^2 + 2 \sum_{j=1}^{\bar k-1}\alpha_j\beta
\int_{\Omega^\eps}p v_j^0 \widetilde u_{\eps,R}}\\
&= \max_{\substack{(\alpha_1,\dots, \alpha_{\bar k-1},\beta)\in
    \R^{\bar k}\\
(\sum_{j=1}^{\bar k-1}\alpha_j^2) + \beta^2 =1}}
\dfrac{\sum_{j=1}^{\bar k-1}\alpha_j^2 \int_{\Omega^\eps} |\nabla v_j^0|^2
+ \beta^2 \int_{\Omega^\eps}\abs{\nabla \widetilde u_{\eps,R}}^2 + 2
\sum_{j=1}^{\bar k-1}\alpha_j\beta
\int_{\Omega^\eps}
\nabla v_j^0\cdot\nabla\widetilde u_{\eps,R}}
{1+\beta^2\int_{\Omega^\eps_{1/2}}p {\bar u_{\eps,R}}^2 + 2 \sum_{j=1}^{\bar k-1}\alpha_j\beta
\int_{\Omega^\eps}p v_j^0 \widetilde u_{\eps,R}}\\
&\leq  \max_{\substack{(\alpha_1,\dots, \alpha_{\bar k-1},\beta)\in
    \R^{\bar k}\\
(\sum_{j=1}^{\bar k-1}\alpha_j^2) + \beta^2 =1}}
\frac{\sum_{j=1}^{\bar k-1}\alpha_j^2\lambda_j^0
+ \beta^2 \int_{\Omega^\eps}\abs{\nabla \widetilde u_{\eps,R}}^2 + 2
\sum_{j=1}^{\bar k-1}\alpha_j\beta
\int_{\Omega^\eps}
\nabla v_j^0\cdot\nabla\widetilde u_{\eps,R}}{1+ o(\eps^N)}
\end{align*}
in view of the estimate 
\[
\bigg|\int_{\Omega^\eps_{1/2}}p v_j^0 \widetilde u_{\eps,R}\bigg|\leq
\nor{p}_{L^{N/2}(\R^N)} \nor{v_j^0}_{L^{2^*}(\R^N)} \nor{\widetilde u_{\eps,R}}_{L^{2^*}(\Omega^\eps_{1/2})}
=o(\eps^N)
\]
which holds by Lemma \ref{l:stimaeste} and Sobolev inequality. 
Then 
\begin{align}\label{eq:9}
& \lambda_\eps - \lambda_0 \\
\notag\leq&
\!\!\!\!\max_{\substack{(\alpha_1,\dots, \alpha_{\bar k-1},\beta)\\
(\sum_{j=1}^{\bar k-1}\alpha_j^2) + \beta^2 =1}}
 \!\!\!\!\!\bigg\{\!\textstyle\sum\limits_{j=1}^{\bar k-1}\alpha_j^2 \lambda_j^0
\!+\! \beta^2 \!\int_{\Omega^\eps}\abs{\nabla \widetilde u_{\eps,R}}^2 \!+ \!2
\sum\limits_{j=1}^{\bar k-1}\alpha_j\beta
\int_{\Omega^\eps}
\nabla v_j^0\cdot\nabla\widetilde u_{\eps,R}- \Big(\sum\limits_{j=1}^{\bar k-1}\alpha_j^2 +\beta^2\Big)\lambda_0
\bigg\} \\
&\notag\qquad\qquad+ o(\eps^N)  \\
\notag=&
\!\!\!\!\max_{\substack{(\alpha_1,\dots, \alpha_{\bar k-1},\beta)\\
(\sum_{j=1}^{\bar k-1}\alpha_j^2) + \beta^2 =1}}
 \!\!\!\!\!\bigg\{\!\textstyle\sum\limits_{j=1}^{\bar k-1}\alpha_j^2 (\lambda_j^0-\lambda_0)
\!+\! \beta^2 \big( \int_{\Omega^\eps}\!\abs{\nabla \widetilde u_{\eps,R}}^2\! \!-\!\! \int_{D^+}\abs{\nabla u_0}^2 \big)\!+ 2
\sum\limits_{j=1}^{\bar k-1}\alpha_j\beta
\int_{\Omega^\eps}\!
\nabla v_j^0\!\cdot\!\nabla\widetilde u_{\eps,R}\bigg\}  \\
&\notag\qquad\qquad+ o(\eps^N).
\end{align}
We observe that, since $\lambda_0$ is simple, for all $i=1,\ldots,\bar k-1$ 
\begin{equation}\label{primo_coeff_sotto} 
a_i:=\lambda_i^0-\lambda_0 <0.
\end{equation}
From convergences \eqref{x_1} and \eqref{v_k} established
in Lemma \ref{blow_up_vari} and Remark \ref{rem:hdiv}, it follows that 
\begin{align}\label{eq:8}
b_{\eps,R}:&=\int_{\Omega^\eps}\abs{\nabla \widetilde u_{\eps,R}}^2dx-
\int_{D^+}\abs{\nabla u_0}^2dx
= \int_{\Omega^\eps_{1+R\eps}}\abs{\nabla \widetilde u_{\eps,R}}^2dx-  \int_{B_{R\eps}^+}\abs{\nabla u_0}^2dx\\
\notag&= \int_{\Gamma_{R\eps}^+} u_0 \bigg(\frac{\partial \widetilde
  u_{\eps,R}}{\partial \nu} - 
\frac{\partial u_0}{\partial \nu} \bigg)\,d\sigma\\
\notag&= \eps^{N-1} \int_{\Gamma_{R}^+} u_0(\e + \eps(x-\e))
\bigg(\dfrac{\partial \widetilde u_{\eps,R}}{\partial \nu} -
\dfrac{\partial u_0}{\partial \nu} \bigg)
(\e + \eps(x-\e))\,d\sigma(x)\\
\notag&= \eps^{N} \int_{\Gamma_{R}^+} u_{0,\eps}(x)
\bigg(\dfrac{\partial V_\eps^R}{\partial \nu} - \dfrac{\partial u_{0,\eps}}{\partial \nu}\bigg)(x)\,d\sigma(x) \\
\notag&= \eps^{N} (b_R+o(1)),\quad\text{as }\eps\to0^+,
\end{align}
where
\begin{equation}\label{eq:bR}
b_R=\bigg(\frac{\partial u_0}{\partial x_1}(\e)\bigg)^{\!\!2}\int_{\Gamma_{R}^+} (x_1-1)
\bigg(\dfrac{\partial v_R}{\partial \nu} - \dfrac{\partial
  (x_1-1)}{\partial \nu}\bigg)(x)\,d\sigma(x)
\end{equation}
For every $i=1,\ldots,\bar k-1$ let us denote 
\[
 c_{\eps,R}^i=\int_{\Omega^\eps} \nabla v_i^0(x)\cdot\nabla\widetilde
  u_{\eps,R}(x)\,dx.
\]
In view of the orthogonality in $D^+\cup D^-$ between $v_j^0$ and $u_0$
\begin{align}\label{eq:12}
 c_{\eps,R}^i &= \int_{D^+\cup D^-} \nabla v_i^0(x)\cdot\nabla\widetilde
  u_{\eps,R}(x)\,dx \\
\notag&= \int_{D^-\cup B_{R\eps}^+} \nabla v_i^0(x)\cdot\nabla\widetilde
  u_{\eps,R}(x)\,dx
+ \int_{D^+ \setminus B_{R\eps}^+} \nabla v_i^0(x)\cdot\nabla u_0(x)\,dx  \\
\notag&= \int_{D^-} \nabla v_i^0(x)\cdot\nabla\widetilde
  u_{\eps,R}(x)\,dx + \int_{B_{R\eps}^+} \nabla v_i^0(x)\cdot\nabla\widetilde
  u_{\eps,R}(x)\,dx - \int_{B_{R\eps}^+} \nabla v_i^0(x)\cdot\nabla u_0(x)\,dx \\
\notag &= O(\eps^N)\quad\text{as }\eps\to 0^+,
\end{align}
taking into account Lemma \ref{l:stimaeste} and the fact that 
\[ 
\int_{B_{R\eps}^+} \abs{\nabla \widetilde
  u_{\eps,R}(x)}^2 dx
\leq \int_{\Omega^{\eps}_{1+R\eps}} \abs{\nabla \widetilde
  u_{\eps,R}(x)}^2 dx
\leq \int_{B_{R\eps}^+} \abs{\nabla u_0(x)}^2 dx
\]
by Dirichlet Principle and  \eqref{u_0_entra_nel_canale}.

We claim that 
\begin{equation}\label{eq:11}
\max_{\substack{(\alpha_1,\dots, \alpha_{\bar k-1},\beta)\\
(\sum_{j=1}^{\bar k-1}\alpha_j^2) + \beta^2 =1}}
 \bigg\{\sum_{j=1}^{\bar k-1}\alpha_j^2 a_j
\!+\! \beta^2 b_{\eps,R}+ 2
\sum\limits_{j=1}^{\bar k-1}\alpha_j\beta c_{\eps,R}^j\bigg\} =\eps^N(b_R+o(1)).
\end{equation}
To prove \eqref{eq:11}, let $\beta_\eps\in\R$, $\alpha_{j,\eps}\in\R$,
$j=1,\dots,\bar k-1$, be such that $\sum_{j=1}^{\bar k-1}\alpha_{j,\eps}^2 + \beta_\eps^2 =1$
and 
\begin{equation}\label{eq:maxass}
\sum_{j=1}^{\bar k-1}\alpha_{j,\eps}^2 a_j
\!+\! \beta_\eps^2 b_{\eps,R}+ 2
\sum\limits_{j=1}^{\bar k-1}\alpha_{j,\eps}\beta_\eps c_{\eps,R}^j=\max_{\substack{(\alpha_1,\dots, \alpha_{\bar k-1},\beta)\\
(\sum_{j=1}^{\bar k-1}\alpha_j^2) + \beta^2 =1}}
 \bigg\{\sum_{j=1}^{\bar k-1}\alpha_j^2 a_j
\!+\! \beta^2 b_{\eps,R}+ 2
\sum\limits_{j=1}^{\bar k-1}\alpha_j\beta c_{\eps,R}^j\bigg\}.
\end{equation}
We first prove that 
\begin{equation}\label{eq:14}
\beta_\eps=1+o(\eps^N),\quad\text{as }\eps\to 0^+.
\end{equation}
Indeed from 
\begin{equation}\label{eq:absu}
b_{\eps,R}\leq \sum_{j=1}^{\bar k-1}\alpha_{j,\eps}^2 a_j
\!+\! \beta_\eps^2 b_{\eps,R}+ 2
\sum\limits_{j=1}^{\bar k-1}\alpha_{j,\eps}\beta_\eps c_{\eps,R}^j,
\end{equation}
\eqref{eq:8}, and \eqref{eq:12}, it follows that 
\[
\eps^{N} (b_R+o(1)) (1-\beta_\eps^2)\leq (1-\beta_\eps^2)\max_{i=1,\dots,\bar
  k-1}a_i+O(\eps^N)
\]
which implies that $1-\beta_\eps^2=O(\eps^N)$. Assuming by
contradiction that \eqref{eq:14} does not hold, there should exists a
sequence $\eps_n\to 0^+$ such that
$\lim_{n\to\infty}\eps_n^{-N}(1-\beta_{\eps_n}^2)=\ell\in(0,+\infty)$.
Then, up to subsequences, there would exist $L<0$ such that
$\lim_{n\to\infty}\eps_n^{-N} \sum_{j=1}^{\bar k-1}\alpha_{j,\eps_n}^2
a_j=L$. Therefore \eqref{eq:absu} and \eqref{eq:8} would imply
\[
\eps_n^N(b_R+o(1))\leq \eps_n^N(L+o(1))+
\eps_n^N(b_R+o(1))(1-\ell\eps_n^N+o(\eps_n^N))+o(\eps_n^N),\quad\text{as
}n\to\infty,
\]
i.e. $b_R+o(1)\leq L+b_R+o(1)$ as $n\to\infty$, thus contradicting
$L<0$. Estimate \eqref{eq:14} is thereby proved.

From \eqref{eq:14} we deduce that 
$\alpha_{j,\eps}=o(\eps^{N/2})$ as $\eps\to 0^+$, then from
\eqref{eq:8}, \eqref{eq:14}, \eqref{eq:12}, and \eqref{eq:maxass} it follows that 
\begin{multline*}
\max_{\substack{(\alpha_1,\dots, \alpha_{\bar k-1},\beta)\\
(\sum_{j=1}^{\bar k-1}\alpha_j^2) + \beta^2 =1}}
 \bigg\{\sum_{j=1}^{\bar k-1}\alpha_j^2 a_j
\!+\! \beta^2 b_{\eps,R}+ 2
\sum\limits_{j=1}^{\bar k-1}\alpha_j\beta c_{\eps,R}^j\bigg\} \\
=
\eps^N(1+o(\eps^N))(b_R+o(1))+o(\eps^N)
=\eps^N(b_R+o(1)),\quad\text{as }\eps\to0^+,
\end{multline*}
thus proving claim \eqref{eq:11}.
From \eqref{eq:9} and \eqref{eq:11}, we deduce that 
\[
\lambda_\eps - \lambda_0\leq \eps^N(b_R+o(1)) ,\quad\text{as
}\eps\to0^+,
\]
and hence, 
for every $R>2$, 
\[
\frac{\lambda_0 - \lambda_\eps}{\eps^N}\geq K_1(\eps,R),
\]
where, for every $R>2$, 
\begin{equation*}
\lim_{\eps\to 0^+}K_1(\eps,R)=
\bigg(\frac{\partial u_0}{\partial x_1}(\e)\bigg)^{\!\!2}\int_{\Gamma_{R}^+} (x_1-1)
\bigg(\dfrac{\partial
  (x_1-1)}{\partial \nu}-\dfrac{\partial v_R}{\partial \nu}\bigg)(x)\,d\sigma(x).
\end{equation*}
{\bf Step 2:} estimate from above.
By the Courant-Fisher \emph{minimax characterization} of the eigenvalue
 $\lambda_0=\lambda_{\bar k}(D^+ \cup D^-)$, we have that
\begin{equation}\label{eq:10}
  \lambda_0 = \min\bigg\{\max_{u\in F\setminus \{0\}}
\dfrac{\int_{D^+ \cup D^-} \abs{\nabla u}^2dx}{\int_{D^+ \cup D^-}
  pu^2\,dx}:F \text{ is a subspace of $\Di{D^+ \cup D^-}$, $\dim F= \bar k$}\bigg\}.
\end{equation}
Let $R>2$ and $\eta_{\eps,R}$ be a smooth cut-off function such that $\eta_{\eps,R}\equiv 1$ in
 $(D^+\setminus B_{R\eps}^+)\cup(D^-\setminus B_{R\eps}^-)$, $\eta_{\eps,R}\equiv 0$ in
 $B_{(R/2)\eps}^+ \cup B_{(R/2)\eps}^-$, $0\leq\eta_{\eps,R}\leq1$ and $|\nabla \eta_{\eps,R}|\leq
 4/(R\eps)$  in
 $D^-\cup D^+$.
We choose the $\bar k$-dimesional space $F=
\mathop{\rm span} \{ \eta_{\eps,R} v_1^\eps,\ldots,\eta_{\eps,R} v_{\bar k-1}^\eps, \widehat u_{\eps,R} \}$ in
\eqref{eq:10}. Then
\begin{align*}
  \lambda_0 &\leq \max_
  {\substack{(\alpha_1,\dots, \alpha_{\bar k})\in \R^{\bar k}\\
      \sum_{j=1}^{\bar k}\alpha_j^2 =1}}
  \dfrac{\int_{D^+ \cup D^-} |\nabla (\sum_{j=1}^{\bar k-1}\alpha_j
    \eta_{\eps,R} v_j^\eps + \alpha_{\bar k}\widehat u_{\eps,R})|^2dx}
  {\int_{D^+\cup D^-} p (\sum_{j=1}^{\bar k-1}\alpha_j \eta_{\eps,R}
    v_j^\eps + \alpha_{\bar k}\widehat u_{\eps,R})^2dx}.
\end{align*}
We notice that 
\begin{align*}
&\int_{D^+ \cup D^-} \bigg|\nabla \bigg(\sum_{j=1}^{\bar k-1}\alpha_j
    \eta_{\eps,R} v_j^\eps + \alpha_{\bar k}\widehat
    u_{\eps,R}\bigg)\bigg|^2dx\\
&=
\sum\limits_{j=1}^{\bar k -1}\alpha_j^2 \int_{D^+\cup D^-} |\nabla (\eta_{\eps,R} v_j^\eps)|^2 dx
+ \alpha_{\bar k}^2 \int_{D^+} |\nabla \widehat u_{\eps,R}|^2dx\\
&\quad+ 
\sum\limits_{\substack{i,j<\bar k\\i\neq j}}\alpha_i\alpha_j
\int_{D^+\cup D^-}
\nabla {(\eta_{\eps,R} v_i^\eps)}\cdot
\nabla {(\eta_{\eps,R} v_j^\eps)}dx
+ 
\sum\limits_{\substack{j=1}}^{\bar k-1}\alpha_j\alpha_{\bar k}
\int_{D^+}
\nabla {(\eta_{\eps,R} v_j^\eps)}\cdot
\nabla {\widehat u_{\eps,R}}dx,
\end{align*}
while, from Lemma \ref{lemma_zero}, Corollary \ref{motivazione_O_grande}, 
assumption \eqref{eq:p2}, and Corollary \ref{stime_integrali_vjeps} it
follows that 
\begin{align*}
\int_{D^+\cup D^-} &p \bigg(\sum_{j=1}^{\bar k-1}\alpha_j \eta_{\eps,R}
    v_j^\eps + \alpha_{\bar k}\widehat u_{\eps,R}\bigg)^{\!\!2}dx= 
\sum\limits_{j=1}^{\bar k -1}\alpha_j^2 \int_{D^+\cup D^-} p (\eta_{\eps,R} v_j^\eps)^2 dx
+ \alpha_{\bar k}^2 \int_{D^+\cup D^-} p {\widehat u_{\eps,R}}^2dx\\
&\ + 
\sum\limits_{\substack{i,j<\bar k\\i\neq j}}\alpha_i\alpha_j
\int_{D^+\cup D^-}
p\,\eta_{\eps,R}^2 v_i^\eps v_j^\eps dx
+ 
\sum\limits_{\substack{j=1}}^{\bar k-1}\alpha_j\alpha_{\bar k}
\int_{D^+\cup D^-}
p\,\eta_{\eps,R} v_j^\eps \widehat u_{\eps,R}\,dx \\
& = \sum\limits_{j=1}^{\bar k -1}\alpha_j^2 \bigg( 1 +
  \int_{B_{R\eps}^-\cup \mathcal C_\eps} p (\eta_{\eps,R}^2-1) 
| v_j^\eps|^2dx \bigg) 
+ \alpha_{\bar k}^2 \bigg(1 - \int_{\Omega^\eps_{1+R\eps}} p u_\eps^2dx\bigg)\\
&\ + 
\sum\limits_{\substack{i,j<\bar k\\i\neq j}}\alpha_i\alpha_j
\int_{B_{R\eps}^-\cup \mathcal C_\eps}
p\,(\eta_{\eps,R}^2-1) v_i^\eps v_j^\eps \,dx-
\sum\limits_{\substack{j=1}}^{\bar k-1}\alpha_j\alpha_{\bar k}
\int_{\Omega^{\eps}_{1/2}}
p\, v_j^\eps u_\eps\,dx\\
& = 1 + o(\eps^N), \quad\text{as }\eps\to 0^+.
\end{align*}
Then
\begin{equation}\label{eq:9bis}
  \lambda_0 - \lambda_\eps \leq  
\max_{\substack{(\alpha_1,\dots, \alpha_{\bar k})\in \R^{\bar k}\\
      \sum_{j=1}^{\bar k}\alpha_j^2 =1}}
\Bigg\{ \sum_{j=1}^{\bar k} \alpha_j^2 \,a_{j,R}^\eps +
\sum_{\substack{i,j=1\\i\neq j}}^{\bar k} \alpha_i\alpha_j\,
c_{i,j,R}^\eps 
\Bigg\}
+ o(\eps^N)
\end{equation}
where we have set
\begin{align*}
& a_{\bar k,R}^\eps=\int_{D^+} |\nabla \widehat u_{\eps,R}|^2dx
- \int_{\Omega^\eps}\abs{\nabla {u_\eps}}^2dx, \\
& a_{j,R}^\eps= \int_{D^+\cup D^-} |\nabla {(\eta_{\eps,R} {v}_{j}^{\eps})}|^2dx
-\int_{\Omega^\eps} |\nabla  u_{\eps}|^2dx, \quad \text{for every }j=1,\ldots,\bar k-1,\\
& c_{i,j,R}^\eps=\int_{D^+\cup D^-}
\nabla {(\eta_{\eps,R} v_i^\eps)}\cdot
\nabla {(\eta_{\eps,R} v_j^\eps)} \,dx,\quad \text{for every
  $i,j=1,\ldots,\bar k-1$, $i\neq j$},\\
& c_{j,\bar k,R}^\eps=c_{\bar k,j,R}^\eps=\int_{D^+}
\nabla {(\eta_{\eps,R} v_j^\eps)}\cdot
\nabla {\widehat u_{\eps,R}} \,dx,\quad \text{for every }j=1,\ldots,\bar k-1.
\end{align*}
Let us study each coefficient of the quadratic form above.
From \eqref{eq:19}, \eqref{eq:20}, and Corollary
\ref{motivazione_O_grande}, we have that 
\begin{align*}
a_{\bar k,R}^\eps&=-\int_{\Omega^\eps_{1+R\eps}}|\nabla u_\eps|^2dx+
 \int_{B_{R\eps}^+}|\nabla \widehat u_{\eps,R}|^2dx\\
&= \int_{\Gamma_{R\eps}^+} u_\eps \bigg( \frac{\partial \widehat u_{\eps,R}}{\partial\nu} 
- \frac{\partial u_\eps}{\partial\nu} \bigg)d\sigma
-\lambda_\eps\int_{\Omega^\eps_{1/2}}p u_\eps^2dx
 \\
&= \eps^N \int_{\Gamma_{R}^+} U_\eps\bigg( \frac{\partial{Z^R_\eps}}{\partial\nu} 
- \frac{\partial U_\eps}{\partial\nu}
\bigg)\,d\sigma(\theta)+o(\eps^N)\quad\text{as }\eps\to0^+,
\end{align*}
and hence, in view of Lemma \ref{blow_up_vari} and Remark \ref{rem:hdiv},
\begin{equation}\label{secondo_coeff_sopra}
a_{\bar k,R}^\eps=\eps^N(a_R+o(1)), \quad \text{as }\eps\to0^+,
\end{equation}
where 
\begin{equation}\label{eq:aR}
a_R=\bigg(\frac{\partial u_0}{\partial x_1}(\e)\bigg)^{\!\!2}\int_{\Gamma_{R}^+} \Phi(x)
\bigg(\dfrac{\partial z_R}{\partial \nu}-\dfrac{\partial
  \Phi}{\partial \nu}\bigg)(x)\,d\sigma(x).
\end{equation}
For every $j=1,\ldots,\bar k-1$, in view of estimates
\eqref{vjeps_unif_lim} and \eqref{vjeps_grad_bucur}  of Lemma
\ref{controllo_supersol}, we have that 
\begin{align}\label{eq:3}
&\int_{B_{R\eps}^-\cup B_{R\eps}^+} \abs{\nabla (\eta_{\eps,R} v_j^\eps) }^2dx \leq 2
 \int_{B_{R\eps}^+\setminus B_{(R/2)\eps}^+}|\nabla \eta_{\eps,R}|^2 |v_{j}^{\eps}|^2dx +\
 \int_{B_{R\eps}^+\setminus B_{(R/2)\eps}^+}\eta_{\eps,R}^2|\nabla
 v_{j}^{\eps}|^2dx\\
&\notag+ 2
 \int_{B_{R\eps}^-\setminus B_{(R/2)\eps}^-}|\nabla \eta_{\eps,R}|^2 |v_{j}^{\eps}|^2dx +\
 \int_{B_{R\eps}^-\setminus B_{(R/2)\eps}^-}\eta_{\eps,R}^2|\nabla
 v_{j}^{\eps}|^2dx
=O(\eps^{N-2}) ,\quad
\text{as }\eps\to 0^+.
\end{align}
From \eqref{eq:3}, estimate
\eqref{vjeps_grad_bucur}  of Lemma \ref{controllo_supersol} and
estimates 
\eqref{stima_int_mezzapalletta_vjeps} and \eqref{stima_int_canale_vjeps}
of Corollary \ref{stime_integrali_vjeps} we deduce that
\begin{align}\label{primo_coeff_sopra}
  a_{j,R}^\eps &= \lambda_j^\eps -\lambda_\eps -
  \int_{B_{R\eps}^-\cup\mathcal C_\eps\cup B_{R\eps}^+} \abs{\nabla
    v_j^\eps }^2dx
  + \int_{B_{R\eps}^-\cup B_{R\eps}^+} \abs{\nabla (\eta_{\eps,R} v_j^\eps) }^2dx \\
  \notag& = \lambda_j^\eps -\lambda_\eps -\lambda_j^\eps
  \int_{B_{R\eps}^-\cup\mathcal C_\eps\cup B_{R\eps}^+} p {v_j^\eps}^2 dx-\int_{\Gamma_{R\eps}^+\cup \Gamma_{R\eps}^-} {v_j^\eps}
  \dfrac{\partial {v_j^\eps}}{\partial \nu}\,d\sigma
  +O(\eps^{N-2})\nonumber \\
  & = \lambda_j^\eps -\lambda_\eps + O(\eps^{N-2})= \lambda_j^0 - \lambda_0 +o(1), \quad \text{as }\eps \to
  0^+. \nonumber
\end{align}
For every $i,j=1,\ldots,\bar k-1$ such that $i\neq j$, from
\eqref{eq:3}, estimates
\eqref{vjeps_bucur} and \eqref{vjeps_grad_bucur} of Lemma \ref{controllo_supersol} and the
orthogonality of $v_i^\eps$ and $v_j^\eps$ in
$\Di{\Omega^\eps}$, it follows that
\begin{align}\label{coeff_misto_sopra}
 c_{i,j,R}^\eps &= - \int_{B_{R\eps}^-\cup\mathcal C_\eps\cup B_{R\eps}^+}\nabla v_i^\eps \cdot \nabla v_j^\eps\,dx 
+  \int_{B_{R\eps}^-\cup B_{R\eps}^+}\nabla (\eta_\eps v_i^\eps) \cdot
\nabla (\eta_\eps v_j^\eps)\,dx \\
&\notag= - \int_{B_{R\eps}^-\cup\mathcal C_\eps\cup B_{R\eps}^+} p\,v_i^\eps v_j^\eps\,dx-
 \int_{\Gamma_{R\eps}^+ \cup \Gamma_{R\eps}^-} v_i^\eps
 \dfrac{\partial v_j^\eps}{\partial \nu}\,d\sigma+ O(\eps^{N-2})\\
& = O(\eps^{N-2}) \quad \text{as }\eps \to 0. \nonumber
\end{align}
From \eqref{eq:min1}, \eqref{v_j^eps_zero_al_giunto}, and \cite[Lemma 2.10 and 2.11]{FT12},
 we have that
\begin{multline}\label{eq:17}
 \int_{B_{R\eps}^+}| \nabla \widehat u_{\eps,R}|^2dx=
 \int_{B_{R\eps}^+}| \nabla \bar v_{\eps,R}|^2dx
\leq 
 \int_{B_{R\eps}^+}|\nabla(\eta_{\eps,R} u_{\eps})|^2dx\\\leq 2
 \int_{B_{R\eps}^+\setminus B_{(R/2)\eps}^+}|\nabla \eta_{\eps,R}|^2 |u_{\eps}|^2dx +2
 \int_{B_{R\eps}^+\setminus B_{(R/2)\eps}^+}\eta_{\eps,R}^2|\nabla u_{\eps}|^2dx=O(\eps^N) ,\quad
\text{as }\eps\to 0^+.
\end{multline}
From \eqref{eq:3}, \eqref{eq:17}, \eqref{vjeps_unif_lim},
\eqref{vjeps_grad_bucur}, and \cite[Lemma 2.10]{FT12} (which in
particular yields $\sup_{\Gamma_{R\eps}^+}|u_\eps|=O(\eps)$ as $\eps\to0^+$),
it follows that, for every $j=1,\ldots,\bar k-1$,
\begin{align}\label{coeff_misto_sopra_bark}
 c_{j,\bar k,R}^\eps & = \int_{B_{R\eps}^+}\nabla (\eta_\eps v_j^\eps) \cdot \nabla \widehat u_{\eps,R}\,dx
- \int_{\Omega^\eps_{1+R\eps}}\nabla v_j^\eps \cdot \nabla u_{\eps}\,dx\\
& = O(\eps^{(N-2)/2})  O(\eps^{N/2})
- \lambda_j^\eps\int_{\Omega^\eps_{1/2}} p v_j^\eps \,u_\eps\, dx 
- \int_{\Gamma_{R\eps}^+ }u_\eps \dfrac{\partial v_j^\eps}{\partial \nu}\,d\sigma \nonumber \\
& = O(\eps^{N-1}) \quad \text{as }\eps \to 0^+. \nonumber
\end{align}
We claim that 
\begin{equation}\label{eq:11bis}
\max_{\substack{(\alpha_1,\dots, \alpha_{\bar k})\in \R^{\bar k}\\
      \sum_{j=1}^{\bar k}\alpha_j^2 =1}}
 \bigg\{
\sum_{j=1}^{\bar k}a_{j,R}^\eps\alpha_j^2
+ 
\sum\limits_{\substack{i,j=1 \\i\neq j}}^{\bar k}c_{i,j,R}^\eps\alpha_i\alpha_j
\bigg\} =\eps^N(a_R+o(1)),
\end{equation}
as $\eps\to0^+$.
To prove \eqref{eq:11bis}, let $\alpha_{j,\eps}\in\R$,
such that $\sum_{j=1}^{\bar k}\alpha_{j,\eps}^2=1$
and 
\begin{equation}\label{eq:maxassbis}
  \sum_{j=1}^{\bar k}a_{j,R}^\eps\alpha_{j,\eps}^2
  + 
  \sum\limits_{\substack{i\neq j}}c_{i,j,R}^\eps\alpha_{i,\eps}\alpha_{j,\eps}
  =
  \max_{\substack{(\alpha_1,\dots, \alpha_{\bar k})\in \R^{\bar k}\\
      \sum_{j=1}^{\bar k}\alpha_j^2 =1}}
  \bigg\{
  \sum_{j=1}^{\bar k}a_{j,R}^\eps\alpha_j^2
  + 
  \sum\limits_{\substack{i\neq j}}c_{i,j,R}^\eps\alpha_i\alpha_j
  \bigg\}.
\end{equation}
From 
\begin{equation*}
a_{\bar k,R}^\eps\leq 
\sum_{j=1}^{\bar k}a_{j,R}^\eps\alpha_{j,\eps}^2
+ 
\sum\limits_{\substack{i\neq j}}c_{i,j,R}^\eps\alpha_{i,\eps}\alpha_{j,\eps},
\end{equation*}
it follows that 
\begin{equation*}
  \Big(-\max_{j<\bar k}a_{j,R}^\eps+a_{\bar
    k,R}^\eps\Big)(1-\alpha_{\bar k,\eps}^2)\leq \sum\limits_{\substack{i\neq j}}c_{i,j,R}^\eps\alpha_{i,\eps}\alpha_{j,\eps},
\end{equation*}
and hence, in view of 
\eqref{secondo_coeff_sopra} and \eqref{primo_coeff_sopra},
\begin{equation}\label{eq:absubis}
  \Big(-\max_{j<\bar k}(\lambda_j^0-\lambda_0)+o(1)\Big)(1-\alpha_{\bar k,\eps}^2)\leq \sum\limits_{\substack{i\neq j}}c_{i,j,R}^\eps\alpha_{i,\eps}\alpha_{j,\eps}.
\end{equation}
From \eqref{eq:absubis}, 
\eqref{coeff_misto_sopra}, and \eqref{coeff_misto_sopra_bark}, it follows that 
\begin{equation}\label{alpha_bark_eps}
1-\alpha_{\bar k,\eps}^2=O(\eps^{N-2}), \quad \text{as }\eps \to 0^+.
\end{equation}
Since $1-\alpha_{\bar k,\eps}^2= \sum_{j<\bar k} \alpha_{j,\eps}^2$, we obtain that,
for every $j=1,\ldots,\bar k-1$,
\begin{equation}\label{alpha_j_eps}
 \alpha_{j,\eps}^2=O(\eps^{N-2}), \quad \text{as }\eps \to 0^+.
\end{equation}
From \eqref{coeff_misto_sopra}, \eqref{alpha_j_eps},
\eqref{coeff_misto_sopra_bark} and \eqref{alpha_bark_eps}, it follows
that
\begin{equation}\label{pezzo_misto}
 \sum\limits_{\substack{i\neq j}}c_{i,j,R}^\eps\alpha_{i,\eps}\alpha_{j,\eps} 
= O(\eps^{2(N-2)}) + O(\eps^{\frac{N-2}{2}+N-1}) 
= 
\begin{cases}
 O(\eps^N), &\text{if }N\geq 4, \\
 O(\eps^{2}),  &\text{if }N=3,
\end{cases}
\quad \text{as }\eps \to 0^+.
\end{equation}
In the case $N=3$,~\eqref{eq:absubis} and \eqref{pezzo_misto} imply
that $1-\alpha_{\bar k,\eps}^2=O(\eps^2)$ as $\eps\to 0^+$, hence for
all $j\neq\bar k$ $\alpha_{j,\eps}=O(\eps)$. Therefore, in view of
\eqref{coeff_misto_sopra} and 
\eqref{coeff_misto_sopra_bark}, we obtain that $\sum_{i\neq j}c_{i,j,R}^\eps\alpha_{i,\eps}\alpha_{j,\eps} 
= O(\eps^{3})$ as $\eps\to 0^+$, thus improving estimate
\eqref{pezzo_misto} for $N=3$. Then, 
for any dimension $N\geq 3$ we obtain
\begin{equation}\label{pezzo_misto_definitivo}
  \sum\limits_{\substack{i\neq j}}c_{i,j,R}^\eps\alpha_{i,\eps}\alpha_{j,\eps} 
  = O(\eps^{N}), \quad \text{as }\eps \to 0^+.
\end{equation}
Arguing a third time in the same way, from \eqref{eq:absubis} and the
improved estimate \eqref{pezzo_misto_definitivo} on the mixed terms,
we can improve \eqref{alpha_bark_eps} and \eqref{alpha_j_eps} obtaining
\begin{equation}\label{alpha_bark_eps_1}
1-\alpha_{\bar k,\eps}^2=O(\eps^{N}),\quad
\alpha_{j,\eps}^2=O(\eps^{N})\text{ for all $j<\bar k$}, \quad \text{as }\eps \to 0^+.
\end{equation}
From \eqref{eq:absubis}, \eqref{coeff_misto_sopra},
\eqref{coeff_misto_sopra_bark}, and \eqref{alpha_bark_eps_1}, we thereby obtain
\begin{align}\label{eq:47}
  \Big(-\max_{j<\bar
    k}(\lambda_j^0-\lambda_0)+o(1)\Big)(1-\alpha_{\bar k,\eps}^2)&\leq
  \sum\limits_{\substack{i\neq
      j}}c_{i,j,R}^\eps\alpha_{i,\eps}\alpha_{j,\eps}\\
\notag& = \sum\limits_{\substack{i\neq j\\ i,j\neq\bar k}}c_{i,j,R}^\eps\alpha_{i,\eps}\alpha_{j,\eps} 
+ \sum\limits_{\substack{j\neq \bar k}}c_{j,\bar k,R}^\eps\alpha_{j,\eps}\alpha_{\bar k,\eps} \\
\notag& = O(\eps^{N-2+N}) + O(\eps^{N-1+(N/2)})\\
\notag& = o(\eps^N), \quad \text{as }\eps \to 0^+,
\end{align}
thus implying 
\begin{equation}\label{alpha_bark_eps_definitivo} 
1-\alpha_{\bar k,\eps}^2=o(\eps^{N}),\quad
\alpha_{j,\eps}^2=o(\eps^{N})\text{ for all $j<\bar k$}, \quad
\text{as }\eps \to 0^+.
\end{equation}
From \eqref{secondo_coeff_sopra}, \eqref{alpha_bark_eps_definitivo},
\eqref{primo_coeff_sopra}, and \eqref{eq:47}, it follows that 
\begin{equation*}
\max_{\substack{(\alpha_1,\dots, \alpha_{\bar k})\in \R^{\bar k}\\
      \sum_{j=1}^{\bar k}\alpha_j^2 =1}}
  \bigg\{
  \sum_{j=1}^{\bar k}a_{j,R}^\eps\alpha_j^2
  + 
  \sum\limits_{\substack{i\neq j}}c_{i,j,R}^\eps\alpha_i\alpha_j
  \bigg\}
 = \eps^N (a_R + o(1)) (1+o(\eps^N))+o(\eps^N), 
\quad \text{as }\eps \to 0^+, 
\end{equation*}
thus proving claim \eqref{eq:11bis}.
From \eqref{eq:9bis} and \eqref{eq:11bis}, we deduce that 
\[
\lambda_0-\lambda_\eps \leq \eps^N(a_R+o(1)) ,\quad\text{as
}\eps\to0^+,
\]
and hence, 
for every $R>2$, 
\[
\frac{\lambda_0 - \lambda_\eps}{\eps^N}\leq K_2(\eps,R),
\]
where, for every $R>2$, 
\begin{equation*}
\lim_{\eps\to 0^+}K_2(\eps,R)=\bigg(\frac{\partial u_0}{\partial x_1}(\e)\bigg)^{\!\!2}\int_{\Gamma_{R}^+} \Phi(x)
\bigg(\dfrac{\partial z_R}{\partial \nu}-\dfrac{\partial
  \Phi}{\partial \nu}\bigg)(x)\,d\sigma(x).
\end{equation*}
{\bf Step 3:} asymptotic behavior. 
Up to now we have proved the following estimate
\begin{equation}\label{estimate_below_above}
K_1(\eps,R)\leq \frac{\lambda_0 - \lambda_\eps}{\eps^N} \leq  
K_2(\eps,R)
\end{equation}
for any $R \in (2,+\infty)$, where 
\begin{equation*}
 \lim_{\eps\to 0^+}K_1(\eps,R)=-b_R,\quad 
 \lim_{\eps\to 0^+}K_2(\eps,R)=a_R,
\end{equation*}
with $b_R$ and  $a_R$ defined in \eqref{eq:bR} and \eqref{eq:aR}
respectively. We now claim that 
\begin{equation}\label{eq:limiteuguali}
  \lim_{R\to+\infty}(-b_R)= \lim_{R\to+\infty}a_R=\bigg(\frac{\partial u_0}{\partial x_1}(\e)\bigg)^{\!\!2} N \int_{\SN_+} \left(\Phi(\e+\theta) - \theta_1 \right)\theta_1 \,d\sigma(\theta).
\end{equation}
As far as  $-b_R$ is concerned, we first observe that 
\begin{equation*}
 \int_{\Gamma_R^+} (x_1-1)\dfrac{\partial (x_1-1)}{\partial\nu}\, d\sigma
= R^{N-1} \int_{\SN_+} R\theta_1^2\,d\sigma = R^N \Upsilon_N^2.
\end{equation*}
Therefore, from Lemma \ref{lemma_v_R}(ii), we have that 
\begin{align*}
  -b_R &= \bigg(\frac{\partial u_0}{\partial
    x_1}(\e)\bigg)^{\!\!2}\int_{\Gamma_{R}^+} (x_1-1)
  \bigg(\dfrac{\partial
    (x_1-1)}{\partial \nu} - \dfrac{\partial v_R}{\partial \nu} \bigg)(x)\,d\sigma(x) \\
  &=\bigg(\frac{\partial u_0}{\partial x_1}(\e)\bigg)^{\!\!2}
  \left(R^N \Upsilon_N^2 - \Upsilon_N
    \frac{\Upsilon_N (R^N+N-1)-N\chi_R(1)}{1-R^{-N}}\right)\\
  &=\bigg(\frac{\partial u_0}{\partial x_1}(\e)\bigg)^{\!\!2}
  \dfrac{N\Upsilon_N}{1-R^{-N}} (\chi_R(1)-\Upsilon_N) \\
  &=\bigg(\frac{\partial u_0}{\partial x_1}(\e)\bigg)^{\!\!2}
  \dfrac{N\Upsilon_N}{1-R^{-N}} \bigg(\int_{\SN_+}v_R(\mathbf
  e_1+\theta)\Psi^+(\theta)\,d\sigma(\theta)-\Upsilon_N\bigg).
\end{align*}
Therefore, from Lemma \ref{lemma_v_R}(i), it follows that
\begin{equation}\label{limite_bR}
  \lim_{R\to+\infty}(-b_R)= \bigg(\frac{\partial u_0}{\partial x_1}(\e)\bigg)^{\!\!2} N \int_{\SN_+} \left(\Phi(\e+\theta) - \theta_1 \right)\theta_1 \,d\sigma(\theta).
\end{equation}
Let us now  study  the limit of $a_R$ as $R\to+\infty$.
We split \eqref{eq:aR} as
\begin{multline}\label{eq:26}
 a_R =\bigg(\frac{\partial u_0}{\partial x_1}(\e)\bigg)^{\!\!2}\int_{\Gamma_{R}^+} \Phi
\bigg(\dfrac{\partial z_R}{\partial \nu}-\dfrac{\partial
  \Phi}{\partial \nu}\bigg)\,d\sigma \\
= \bigg(\frac{\partial u_0}{\partial x_1}(\e)\bigg)^{\!\!2} \bigg(\int_{\Gamma_{R}^+} (x_1-1)
\bigg(\dfrac{\partial z_R}{\partial \nu}-\dfrac{\partial
  \Phi}{\partial \nu}\bigg)d\sigma+ \int_{\Gamma_{R}^+} (\Phi-(x_1-1))
\bigg(\dfrac{\partial z_R}{\partial \nu}-\dfrac{\partial
  \Phi}{\partial \nu}\bigg)d\sigma\bigg).
\end{multline}
We first prove that the second term in the sum vanishes as $R\to\infty$. Indeed,
\begin{multline}\label{eq:27}
 \int_{\Gamma_{R}^+} (\Phi-(x_1-1))
\bigg(\dfrac{\partial z_R}{\partial \nu}-\dfrac{\partial
  \Phi}{\partial \nu}\bigg)d\sigma= \int_{\Gamma_{R}^+} (\Phi-(x_1-1))
\bigg(\dfrac{\partial (x_1-1)}{\partial \nu}-\dfrac{\partial
  \Phi}{\partial \nu}\bigg)\,d\sigma \\
+ \int_{\Gamma_{R}^+} (\Phi-(x_1-1))
\bigg(\dfrac{\partial z_R}{\partial \nu}-\dfrac{\partial
  (x_1-1)}{\partial \nu}\bigg)\,d\sigma.
\end{multline}
Testing  equation $-\Delta(\Phi(x)-(x_1-1)) = 0$ in $D^+ \setminus
B_R^+$ with $\Phi(x)-(x_1-1)$, we have that 
\begin{equation}\label{eq:28}
\int_{\Gamma_{R}^+} (\Phi(x)-(x_1-1))
\bigg(\dfrac{\partial (x_1-1)}{\partial \nu}-\dfrac{\partial
  \Phi}{\partial \nu}\bigg)(x)\,d\sigma(x)
=  \int_{D^+ \setminus B_R^+} \abs{\nabla (\Phi(x)-(x_1-1))}^2 \to 0 
\end{equation}
as $R\to+\infty$ thanks to \eqref{eq_Phi_1}.
On the other hand, testing the equation $-\Delta (z_R-(x_1-1))=0$ in $B_R^+$ with the function
$\eta_R (\Phi-(x_1-1))$ (being $\eta_R$ a cut-off function as in 
\eqref{eq:cutoff}), we have that 
\begin{align}\label{eq:29}
\lefteqn{ \int_{\Gamma_{R}^+} (\Phi(x)-(x_1-1))
\bigg(\dfrac{\partial z_R}{\partial \nu}-\dfrac{\partial
  (x_1-1)}{\partial \nu}\bigg)(x)\,d\sigma(x)}\\
\notag&= \int_{B_R^+} \nabla(z_R-(x_1-1))\cdot \nabla(\eta_R (\Phi-(x_1-1)))\,dx\\
\notag&\leq \left(\int_{B_R^+} \abs{\nabla(z_R-(x_1-1))}^2dx\right)^{1/2} 
\left(\int_{B_R^+} \abs{\nabla(\eta_R (\Phi-(x_1-1)))}^2dx\right)^{1/2} \\
\notag& \leq \int_{B_R^+} \abs{\nabla(\eta_R
  (\Phi-(x_1-1)))}^2=o(1)\quad\text{as }R\to +\infty
\end{align}
thanks to the Dirichlet Principle and estimate
\eqref{eq:25}. Therefore, from \eqref{eq:26}, \eqref{eq:27},
\eqref{eq:28}, \eqref{eq:29}, Lemmas \ref{lemma_z_R} and
\ref{lemma_Phi_1}, and the fact that, in view of \eqref{eq:z_R},
\eqref{eq:phi}, and \eqref{eq:varphi}, $\phi_R(R)=\varphi(R)$ for all
$R>2$, it follows that 
\begin{align}\label{limite_aR}
  a_R &= \bigg(\frac{\partial u_0}{\partial x_1}(\e)\bigg)^{\!\!2} \left(\int_{\Gamma_{R}^+} (x_1-1)
    \bigg(\dfrac{\partial z_R}{\partial \nu}-\dfrac{\partial
      \Phi}{\partial \nu}\bigg)(x)\,d\sigma(x) \right) + o(1) \\
  &= \bigg(\frac{\partial u_0}{\partial x_1}(\e)\bigg)^{\!\!2} N
  \Upsilon_N R^N \left(\dfrac{\varphi(R)}{R} -\Upsilon_N\right)+
  o(1)\nonumber \\
  &= \bigg(\frac{\partial u_0}{\partial x_1}(\e)\bigg)^{\!\!2} N
  \Upsilon_N (\varphi(1) -\Upsilon_N)+
  o(1)\nonumber \\
&= \bigg(\frac{\partial u_0}{\partial x_1}(\e)\bigg)^{\!\!2} N \int_{\SN_+} \left(\Phi(\e+\theta) - \theta_1 \right)\theta_1\,d\sigma+ o(1) \quad\text{as $R\to+\infty$.}\nonumber
\end{align}
Combining \eqref{limite_bR} and \eqref{limite_aR} we prove claim 
\eqref{eq:limiteuguali}. The conclusion follows from
\eqref{estimate_below_above} and \eqref{eq:limiteuguali}, observing
that 
$\int_{\SN_+} (\Phi(\e+\theta) - \theta_1)\theta_1\,d\sigma>0$ due
to the fact that, by the Strong
Maximum Principle, $\Phi>(x_1-1)^+$ in $D^+$.
\end{pf}

\noindent We are now ready to prove Theorem \ref{teo_asintotico_autovalori}.

\begin{pfn}{Theorem \ref{teo_asintotico_autovalori}}
From Lemma \ref{lemma_Phi_1} it follows that 
\begin{equation}\label{eq:32}
\int_{\SN_+} \left(\Phi(\e+\theta) - \theta_1 \right)\theta_1 \,d\sigma=
\frac1{1-N}  \int_{\Gamma_1^+}\frac{\partial
  (\Phi-(x_1-1))}{\partial\nu}(x_1-1)\,d\sigma.
\end{equation}
On the other hand, testing first equation $-\Delta (\Phi-(x_1-1))=0$
in $B^+_1$ with $(x_1-1)$ and then   equation $-\Delta (x_1-1)=0$
in $B^+_1$ with $\Phi-(x_1-1)$, we obtain that 
\begin{align}\label{eq:33}
  \int_{\Gamma_1^+}\frac{\partial
  (\Phi-(x_1-1))}{\partial\nu}(x_1-1)\,d\sigma&=\int_{B_1^+}\nabla(\Phi-(x_1-1))\cdot\nabla(x_1-1)\,dx\\
\notag&=
\int_{\SN_+} \left(\Phi(\e+\theta) - \theta_1 \right)\theta_1
\,d\sigma-\int_\Sigma \Phi(1,x')\,dx'.
\end{align}
Combining \eqref{eq:32} and \eqref{eq:33}, we deduce that 
\begin{equation}\label{eq:34}
  N \int_{\SN_+} \left(\Phi(\e+\theta) - \theta_1 \right)\theta_1 \,d\sigma
=\int_\Sigma \Phi(1,x')\,dx'.
\end{equation}
The conclusion follows from Theorem \ref{t:asintotico_autovalori1},
\eqref{eq:34}, Corollary \ref{cor:min_phi1}, and \eqref{eq:legame_m_c}.  
\end{pfn}

Steiner rearrangement allows proving that the shape of the section
$\Sigma$ minimizing $m(\Sigma)$ and hence maximizing $\lim_{\eps\to
  0^+}\eps^{-N}(\lambda_0 - \lambda_\eps)$ is the spherical one.
\begin{Proposition}\label{p:steiner}
For every $\Sigma\subset \R^{N-1}$ being an open bounded domain containing $0$, let $m(\Sigma)$ be defined in 
  \eqref{eq:m(sigma)}. Then, for every $\mu>0$, 
  \begin{equation*}
    \min\Big\{m(\Sigma):\Sigma \subset \R^{N-1} \text{ is a bounded
      domain}, \, 0\in\Sigma, \, |\Sigma|=\mu\Big\}
    =m\Big(B'\Big({\mathbf 0},\big(\tfrac{\mu(N-1)}{\omega_{N-2}}\big)^{\frac1{N-1}}\Big)\Big)
  \end{equation*}
where $|\cdot|$  denotes the Lebesgue measure of $\R^{N-1}$,
$B'({\mathbf 0},r):=\{x'\in\R^{N-1}:|x'|<r\}$
 denotes the
$(N-1)$-dimensional ball of radius $r$ centered at ${\mathbf 0}$, and 
$\omega_{N-2}$ denotes the volume of the unit $(N-2)$-dimensional sphere.
\end{Proposition}
\begin{pf}
  For every $\Sigma\subset \R^{N-1}$ being an open bounded domain
  containing $0$, let us consider the domain 
$D_\Sigma:=\big((-\infty,1]\times \Sigma\big)\cup D^+$ 
and its Steiner symmetral $D_\Sigma^\sigma$  in codimension $N-1$
  defined as 
\[
D_\Sigma^\sigma=\big((-\infty,1]\times B'({\mathbf 0},r_\Sigma)
\big)\cup D^+=D_{B'({\mathbf 0},r_\Sigma)},
\]
where $r_\Sigma=\big(\tfrac{|\Sigma|(N-1)}{\omega_{N-2}}\big)^{\frac1{N-1}}$.
For every $w\in \mathcal D^{1,2}(D_\Sigma)$, $w\geq 0$ a.e., its Steiner rearrangement
in codimension $N-1$ is the function $w^\sigma\in  \mathcal
D^{1,2}(D_\Sigma^\sigma)$ defined as 
\[
w^\sigma(x_1,x')=\inf\bigg\{t>0: |\{y\in \R^{N-1}:w(x_1,y)>t |
\leq \frac{\omega_{N-2}}{N-1}|x'|^{N-1}\bigg\}.
\]
For every open bounded domain $\Sigma\subset
\R^{N-1}$ containing $0$ and $w\in \mathcal D^{1,2}(D_\Sigma)$  such that
$w\geq0$ a.e., the P\'olya-Szeg\"o inequality for the Steiner rearrangement (see e.g
\cite{BSY} and \cite{capriani}) implies that 
\[
\int_{D_\Sigma}|\nabla w(x_1,x')|^2\,dx_1\,dx'\geq 
\int_{D_\Sigma^\sigma}|\nabla w^\sigma(x_1,x')|^2\,dx_1\,dx'=
\int_{D_{B'({\mathbf 0},r_\Sigma)}}|\nabla w^\sigma(x_1,x')|^2\,dx_1\,dx'
,
\]
whereas the Cavalieri principle yields
\[
\int_{\Sigma}w(1,x')\,dx'=
\int_{
 B'({\mathbf 0},r_\Sigma)}w^\sigma(1,x')\,dx'.
\]
Therefore, letting $J_\Sigma:\mathcal D^{1,2}(D_\Sigma)\to\R$,
$J_\Sigma(w):= \frac12\int_{D_\Sigma}|\nabla w|^2\,dx-\int_{\Sigma}
w(1,x')\,dx'$, we have that
\[
J_\Sigma(w)\geq J_{B'({\mathbf 0},r_\Sigma)}(w^\sigma)\quad\text{for
  every  $w\in \mathcal D^{1,2}(D_\Sigma)$  such that
$w\geq0$ a.e..}
\]
Since the minimum of $J_\Sigma$ over $\mathcal D^{1,2}(D_\Sigma)$ is
attained by a nonnegative function, we then conclude 
that 
\[
m(\Sigma)\geq m(B'({\mathbf 0},r_\Sigma))
\]
for every open bounded domain $\Sigma\subset
\R^{N-1}$ containing $0$, thus completing the proof.
\end{pf}

\section{Rate of convergence for eigenfunctions}\label{sec:rate-conv-eigenf}

In this section we prove a sharp estimate for the rate of convergence
of eigenfunctions.
In view of Corollary \ref{stima_autofunzioni_via_blow_up}, it will be
sufficient to obtain
an estimate of $\|u_\eps - u_0\|_{\Di{D^+}}$. 
To this aim, we consider the following operator
\begin{align}\label{def_operatore_F}
  F:\  &\R \times \Di{D^+}  \longrightarrow  \R \times (\Di{D^+})^\star\\
\notag  &(\lambda,u) \quad \longmapsto \quad (\nor{u}^2 -\lambda_0, -\Delta
  u-\lambda p u),
\end{align}
   where the symbol $\nor{\cdot}$ stands for the $\Di{D^+}$-norm, i.e. 
 \[
\|u\|:=\bigg(\int_{D^+}|\nabla u|^2dx\bigg)^{1/2},
\]
 $(\mathcal D^{1,2}(D^+))^\star$ is
the dual space of $\mathcal D^{1,2}(D^+)$,  and, for all
 $u\in \Di{D^+}$,  $-\Delta
  u-\lambda pu\in (\Di{D^+})^\star$ acts as 
\[
\phantom{a}_{(\mathcal D^{1,2}(D^+))^\star}\big\langle 
-\Delta
  u-\lambda p u
,
v
\big\rangle_{\mathcal D^{1,2}(D^+)}=\int_{D^+}\nabla u(x)\cdot\nabla
v(x)\,dx-\lambda \int_{D^+} p(x)u(x)v(x)\,dx.
\]
  We recall from \eqref{eq:13} that $\int_{D^+}p {u_0}^2\,dx =1$ and
  hence $\|u_0\|^2=\lambda_0$.  
Therefore $F(\lambda_0,u_0)=(0,0)$. 

\begin{Lemma}\label{F_frechet}
  Under assumptions \eqref{eq:condsigma}, \eqref{eq:p}, \eqref{eq:p2},
  \eqref{eq:53}, and \eqref{eq:54}, let
  $\lambda_0=\lambda_{k_0}(D^+)=\lambda_{\bar k}(D^- \cup D^+)$ be the
  $\bar k$-th eigenvalue of problem \eqref{eq:30} on $D^- \cup D^+$
  (which is equal to the simple $k_0$-th eigenvalue on $D^+$) and
  $u_0$  be as in \eqref{eq:u0} and \eqref{eq:13}.
Then, the operator $F$ defined in \eqref{def_operatore_F} is
Frech\'{e}t-differentiable at $(\lambda_0,u_0)$ and its 
Frech\'{e}t-differential $dF(\lambda_0,u_0)\in \mathcal L(
\R \times \Di{D^+},\R \times (\Di{D^+})^\star)$ is invertible.
\end{Lemma}
\begin{pf}
For all $(\lambda,u)\in \R \times \Di{D^+}$, there holds
 \begin{align*}
 F(\lambda_0+\lambda,u_0+u)&=
\big(\nor{u_0+u}^2 -\lambda_0, 
-\Delta u -\lambda_0p u -\lambda p u_0  - \lambda p u\big)\\
&=\bigg( 2\int_{D^+}\nabla u_0\cdot\nabla u \,dx+\|u\|^2,-\Delta u
-\lambda_0p u -\lambda p u_0  - \lambda p u\bigg)\\
&=\bigg( 2\int_{D^+}\nabla u_0\cdot\nabla u \,dx,-\Delta u
-\lambda_0p u -\lambda p u_0\bigg)+o(|\lambda|+\|u\|)
\end{align*}
as $(\lambda,u)\to0$  in  $\R \times \Di{D^+}$. Therefore  $F$ is
Frech\'{e}t-differentiable at $(\lambda_0,u_0)$ and 
\begin{equation*}
 dF(\lambda_0,u_0)(\lambda,u)=\bigg( 2\int_{D^+}\!\!\nabla u_0\!\cdot\!\nabla u \,dx,-\Delta u
-\lambda_0p u -\lambda p u_0\bigg)
\quad\text{ for every }(\lambda,u)\in \R \times \Di{D^+}.
\end{equation*}
It remains to prove that 
$dF(\lambda_0,u_0):\R \times \Di{D^+}\to\R \times (\Di{D^+})^\star$ is invertible.
To this aim, by exploiting the 
compactness of the map 
 $\Di{D^+}\to (\Di{D^+})^\star$, $u\mapsto pu$, it is easy to prove that, if
$\mathcal R:(\Di{D^+})^\star\to \Di{D^+}$ is the Riesz isomorphism and
$\mathop{\rm Id}_\R$ denotes the identity on $\R$,
then the operator $(\mathop{\rm Id}_\R\times \mathcal R)\circ
dF(\lambda_0,u_0)\in\mathcal L(\R \times \Di{D^+})$ is a compact
perturbation of the identity. Therefore, from the Fredholm
alternative, $dF(\lambda_0,u_0)$ is invertible if and only if it is
injective. 

Let  $(\lambda,u)\in \R \times \Di{D^+}$ be such that
$dF(\lambda_0,u_0)(\lambda,u)$ vanishes, i.e.
\begin{equation*}
\begin{cases}
 2\int_{D^+}\nabla u_0\cdot\nabla u \,dx =0,\\
-\Delta u -\lambda_0 p u -\lambda p u_0 =
0,
\end{cases}
\end{equation*}
i.e. 
\begin{equation}\label{eq:38}
  \int_{D^+}(\nabla u\cdot\nabla v
-\lambda_0 p uv -\lambda p u_0v)
 \,dx=0\text{ for all }v\in \Di{D^+}. 
\end{equation}
Therefore, 
\begin{equation}\label{eq:35}
\lambda_0 \int_{D^+}puu_0 \,dx=\int_{D^+}\nabla
u_0\cdot\nabla u\,dx=0
\end{equation}
and then, choosing $v=u_0$ in \eqref{eq:38}, we obtain that $0=\lambda\int_{D^+}pu_0^2 \,dx=\lambda$. It
follows that $u$ is a weak $\Di{D^+}$-solution to $-\Delta u =\lambda_0 p
u$ in $D^+$. Since, by assumption \eqref{eq:53}, the eigenvalue $\lambda_0$
is simple on $D^+$, we conclude that $u=\alpha u_0$ for some $\alpha\in\R$.
From \eqref{eq:35} it follows that $0=\alpha\lambda_0 \int_{D^+}pu_0^2
\,dx$ which implies $\alpha=0$ and then $u\equiv 0$ in $D^+$.
We conclude that $dF(\lambda_0,u_0)$ is injective  and then
invertible.
\end{pf}

\begin{Theorem}\label{stima_teo_inversione}
Let $\widehat {u}_{\eps,R}$ be as in definition \eqref{v_j^eps_zero_al_giunto}.
Then, for every $\eps>0$  and $R>2$ there exist
$K(\eps,R), K(R)\in\R$ such that 
\begin{gather*}
\eps^{-N/2}\|\widehat {u}_{\eps,R} -
u_0\|_{\Di{D^+}}\leq K(\eps,R),\\[5pt]
\lim_{\eps\to 0^+}K(\eps,R)=K(R),\quad \text{for every }R>2,\quad
\text{and}\quad \lim_{R\to+\infty}K(R)=0.
\end{gather*}
\end{Theorem}

\begin{pf}
Let us fix $R>2$ and notice that $\widehat {u}_{\eps,R}\to u_0$ in
$\Di{D^+}$ as $\eps\to 0^+$. Indeed, from \eqref{eq:19},
\eqref{eq:20}, \eqref{x_1}, \eqref{z_h}, and \eqref{convergenza_u_0},
we deduce that  
\begin{align*}
  \int_{D^+}|\nabla(\widehat {u}_{\eps,R}-u_0)|^2\,dx=
  \int_{D^+\setminus B_{R\eps}^+}|\nabla(u_{\eps}-u_0)|^2\,dx+
  \eps^N\int_{B_{R}^+}|\nabla(Z^{R}_\eps-u_{0,\eps})|^2\,dx=o(1)
\end{align*}
as $\eps\to 0^+$.
Therefore 
\begin{equation}\label{eq:37}
 F(\lambda_\eps,\widehat {u}_{\eps,R}) = dF(\lambda_0,u_0) (\lambda_\eps - \lambda_0, \widehat {u}_{\eps,R} - u_0)
+ o(|\lambda_\eps - \lambda_0| + \|\widehat {u}_{\eps,R} -
u_0\|),\quad\text{as }\eps\to 0^+.
\end{equation}
In view of Lemma \ref{F_frechet}, the operator $dF(\lambda_0,u_0)$ is
invertible (and its inverse is continuous by the Open Mapping
Theorem), then \eqref{eq:37} implies that 
\begin{gather}\label{eq:41}
|\lambda_\eps - \lambda_0| + \|\widehat {u}_{\eps,R} -
u_0\|\\
\notag\leq \|(dF(\lambda_0,u_0))^{-1}\|_{ \mathcal L(
\R \times (\Di{D^+})^\star,\R \times \Di{D^+})}
\|dF(\lambda_0,u_0) (\lambda_\eps - \lambda_0, \widehat {u}_{\eps,R} - u_0)\|_{\R \times (\Di{D^+})^\star}\\
\notag=\|(dF(\lambda_0,u_0))^{-1}\|_{ \mathcal L(
\R \times (\Di{D^+})^\star,\R \times \Di{D^+})}
\|
F(\lambda_\eps,\widehat {u}_{\eps,R})\|_{\R \times (\Di{D^+})^\star }
(1+o(1))
\end{gather}
as $\eps\to 0^+$.
In order to prove the theorem, we are going to estimate the norm of
\begin{equation}\label{eq:42}
  F(\lambda_\eps,\widehat {u}_{\eps,R}) = (\mu_\eps,w_\eps)
  = \Big(\|\widehat {u}_{\eps,R}\|^2 -\lambda_0,
      -\Delta \widehat {u}_{\eps,R} - \lambda_\eps p \widehat {u}_{\eps,R}\Big).
              \end{equation}
 As far as $\mu_\eps$ is concerned, from Theorem
 \ref{teo_asintotico_autovalori} and \eqref{secondo_coeff_sopra}
 it follows that 
\begin{align}\label{eq:43}
\mu_\eps  &= \int_{D^+} |\nabla \widehat {u}_{\eps,R}|^2\,dx
-\lambda_0 = 
\int_{D^+\setminus B_{R\eps}^+} |\nabla u_\eps|^2dx
+ \int_{B_{R\eps}^+} |\nabla \bar {v}_{\eps,R}|^2dx -\lambda_0 \\
\notag&= \int_{B_{R\eps}^+} |\nabla \bar {v}_{\eps,R}|^2dx - \int_{\Omega^\eps_{1+R\eps}} |\nabla u_\eps|^2dx
+\lambda_\eps- \lambda_0\\
\notag&= a_{\bar k,R}^\eps+\lambda_\eps- \lambda_0=  O(\eps^N),
\quad\text{as }\eps\to 0^+,
\end{align}
where  $a_{\bar k,R}^\eps$ is as in the proof of Theorem
\ref{t:asintotico_autovalori1}. In particular $\lim_{\eps\to 0^+}\eps^{-N/2}\mu_\eps =0$.

As far as $w_\eps$ is concerned, we observe that, for every
$\varphi\in  \mathcal D^{1,2}(D^+)$,
\[
\phantom{a}_{(\mathcal D^{1,2}(D^+))^\star}\big\langle 
w_\eps, \varphi\big\rangle_{\mathcal D^{1,2}(D^+)}
= \int_{D^+} \big(\nabla \widehat {u}_{\eps,R} \nabla \varphi -\lambda_\eps p \widehat {u}_{\eps,R}\varphi\big)\,dx
= \int_{\Gamma_{R\eps}^+}\bigg( \frac{\partial \bar {v}_{\eps,R}}{\partial \nu} - \frac{\partial u_\eps}{\partial \nu}
\bigg)\varphi\,d\sigma.
\]
Thus, letting $Z^R_\eps$ and $U_\eps$ as in \eqref{eq:20}
and \eqref{eq:19} respectively, we have that 
\begin{align*}
\eps^{-N/2}\|w_\eps\|_{(\mathcal D^{1,2}(D^+))^\star}&=
\frac{1}{\eps^{N/2}} \sup_{\substack{
\varphi \in \Di{D^+} \\ 
\nor{\varphi}=1}}
\int_{\Gamma_{R\eps}^+}\bigg( \frac{\partial \bar {v}_{\bar
    k,R}^{\eps}}{\partial \nu} - \frac{\partial u_\eps}{\partial \nu}
\bigg)\varphi\,d\sigma\\
\notag&= \sup_{\substack{\varphi \in \Di{D^+}\\ \nor{\varphi}=1}}
\int_{\Gamma_{R}^+}\left( \dfrac{\partial Z_\eps^R}{\partial \nu}
- \dfrac{\partial U_\eps}{\partial \nu} \right)\dfrac{\varphi(\e+\eps(x-\e))}{\eps^{1-N/2}}\,d\sigma \\
\notag&= \sup_{\substack{\varphi \in \Di{D^+} \\\nor{\varphi}=1}}
\int_{\Gamma_{R}^+}\left( \dfrac{\partial Z_\eps^R}{\partial \nu}
- \dfrac{\partial U_\eps}{\partial \nu}\right)\mathcal T_\eps(\varphi)\,d\sigma \\
\notag&= \sup_{\substack{\varphi \in \Di{D^+} \\\nor{\varphi}=1}}
\int_{B_{R}^+}\nabla (Z_\eps^R-U_\eps)\cdot \nabla
\mathcal T_\eps(\varphi)\,dx
\end{align*}
where $\mathcal T_\eps:\Di{D^+}\to \Di{D^+}$ is defined as $\mathcal
T_\eps(\varphi) (x)=\eps^{\frac{N-2}{2}}\varphi(\e+\eps(x-\e))$. 
Since $\mathcal T_\eps$ is an isometry of $\Di{D^+}$, we deduce that 
\begin{align*}
\eps^{-N/2}\|w_\eps\|_{(\mathcal D^{1,2}(D^+))^\star}= \sup_{\substack{\varphi \in \Di{D^+} \\\nor{\varphi}=1}}
\int_{B_{R}^+}\nabla (Z_\eps^R-U_\eps)\cdot \nabla
\varphi\,dx.
\end{align*}
From the convergences \eqref{Phi} and \eqref{z_h}  established in
Lemma \ref{blow_up_vari} it follows that 
\begin{equation}\label{eq:44}
\lim_{\eps\to 0^+}\eps^{-N/2}\|w_\eps\|_{(\mathcal
  D^{1,2}(D^+))^\star}= \big(\tfrac{\partial u_0}{\partial x_1}(\e)\big)\sup_{\substack{\varphi \in \Di{D^+}
    \\\nor{\varphi}=1}} \int_{B_{R}^+}\nabla (z_R-\Phi)\cdot \nabla \varphi\,dx.
\end{equation}
We observe that, for every $\varphi \in \Di{D^+}$ such that $\|\varphi\|=1$, Lemma
\ref{l:secondo_lemma_Phi_1}(iii) implies that 
\begin{align}\label{eq:39}
  \bigg|\int_{B_{R}^+}&\nabla (z_R-\Phi)\cdot \nabla
  \varphi\,dx\bigg|\\
  \notag&= \bigg|\int_{B_{R}^+}\nabla ((x_1-1)^+-\Phi)\cdot \nabla
  \varphi\,dx +
  \int_{B_{R}^+}\nabla (z_R-(x_1-1)^+)\cdot \nabla \varphi\,dx\bigg|\\
  \notag&= \bigg| -\int_{D^+\setminus B_{R}^+}\nabla
  ((x_1-1)^+-\Phi)\cdot \nabla \varphi\,dx +
  \int_{B_{R}^+}\nabla (z_R-(x_1-1)^+)\cdot \nabla \varphi\,dx\bigg|\\
  \notag&\leq \bigg(\int_{D^+\setminus B_{R}^+}|\nabla
  ((x_1-1)^+-\Phi)|^2\,dx\bigg)^{\!\!1/2} +
  \bigg(\int_{B_{R}^+}|\nabla (z_R-(x_1-1)^+)|^2\,dx\bigg)^{\!\!1/2}.
\end{align}
Since $z_R-(x_1-1)^+$ is harmonic in $B_R^+$,
$(z_R-(x_1-1)^+)\big|_{\Gamma_R^+}=\Phi-(x_1-1)^+\big|_{\Gamma_R^+}$,
and vanishes on $\partial B_R^+\cap \partial D^+$, if $\eta_R$ is a
smooth cut-off function satisfying \eqref{eq:cutoff}, from the Dirichlet
Principle, \eqref{eq_Phi_1}, and \eqref{eq:Phiinfty}, we can estimate
\begin{align}\label{eq:40}
  \int_{B_R^+} |\nabla &(z_R-(x_1-1)^+)|^2\,dx \leq \int_{B_R^+} 
|\nabla (\eta_R(\Phi-(x_1-1)^+))|^2 \,dx\\
\notag&\leq 2 \int_{B_R^+} |\nabla \eta_R|^2(\Phi-(x_1-1)^+)^2dx 
+ 2 \int_{D^+\setminus B_{R /2}^+} \eta_R^2 |\nabla(\Phi-(x_1-1)^+)|^2 dx\\
  &\notag\leq {\rm const\,} R^{-2} R^{2-2N} R^N + o(1) = o(1)
\end{align}
as $R\to+\infty$. From \eqref{eq:39}, \eqref{eq:40}, and
\eqref{eq_Phi_1} we deduce that 
\begin{equation}\label{eq:45}
\lim_{R\to +\infty}\sup_{\substack{\varphi \in \Di{D^+}
    \\\nor{\varphi}=1}} \int_{B_{R}^+}\nabla (z_R-\Phi)\cdot \nabla \varphi\,dx=0.
\end{equation}
The conclusion follows combining \eqref{eq:36} with \eqref{eq:41},
\eqref{eq:42}, \eqref{eq:43}, \eqref{eq:44}, and \eqref{eq:45}.
\end{pf}

\begin{pfn}{Theorem \ref{teo_asintotico_autofunzioni}}
Let $\delta>0$. From Theorem \ref{stima_teo_inversione} and \eqref{eq_Phi_1}, there exists
$R_0=R_0(\delta)>2$ such that 
\[
K^2(R_0)\in [0,\delta)\quad\text{and}\quad 
\big(\tfrac{\partial
  u_0}{\partial x_1}(\e)\big)^2\int_{D^+\setminus B^+_{R_0}}|\nabla (\Phi-(x_1-1))(x)|^2\,dx<\delta.
\]
From Theorem \ref{stima_teo_inversione},
it follows that 
\begin{align*}
  \bigg|\frac1{\eps^N}&\int_{\Omega^\eps}|\nabla (u_\eps-u_0)|^2\,dx-
  \big(\tfrac{\partial u_0}{\partial x_1}(\e)\big)^2 \int_{\widetilde D} \abs{\nabla (\Phi - (x_1-1)^+)}^2dx \bigg|
  \\
  &\leq \big(\tfrac{\partial u_0}{\partial x_1}(\e)\big)^2
  \int_{D^+\setminus B_{R_0}^+} \abs{\nabla (\Phi - (x_1-1)^+)}^2dx
  +\frac1{\eps^N}\int_{D^+\setminus B^+_{R_0\eps}}|\nabla
  (u_\eps-u_0)|^2\,dx
  \\
  &\quad +\bigg|\frac{1}{\eps^N} \int_{\Omega^\eps_{1+R_0\eps}}
  \abs{\nabla (u_\eps - u_0)}^2dx- \big(\tfrac{\partial u_0}{\partial
    x_1}(\e)\big)^2 \int_{T_1^-\cup B_{R_0}^+} \abs{\nabla (\Phi -
    (x_1-1)^+)}^2dx
  \bigg|\\
  &\leq \delta+K^2(\eps,R_0)
  \\
  &\quad +\bigg|\frac{1}{\eps^N} \int_{\Omega^\eps_{1+R_0\eps}}
  \abs{\nabla (u_\eps - u_0)}^2dx- \big(\tfrac{\partial u_0}{\partial
    x_1}(\e)\big)^2 \int_{T_1^-\cup B_{R_0}^+} \abs{\nabla (\Phi -
    (x_1-1)^+)}^2dx \bigg|
\end{align*}
and hence, from Corollary \ref{stima_autofunzioni_via_blow_up} and
Theorem \ref{stima_teo_inversione}, we deduce that there exists
$\eps(\delta)>0$ such that, for all $\eps\in(0,\eps(\delta))$,
\begin{equation*}
  \bigg|\frac1{\eps^N}\int_{\Omega^\eps}|\nabla (u_\eps-u_0)|^2\,dx-
  \big(\tfrac{\partial u_0}{\partial x_1}(\e)\big)^2 \int_{\widetilde D} \abs{\nabla (\Phi - (x_1-1)^+)}^2dx \bigg|
 \leq 3\delta+K^2(R_0)<4\delta,
\end{equation*}
thus proving that 
\[
\lim_{\eps\to 0^+}
\frac1{\eps^N}\int_{\Omega^\eps}|\nabla (u_\eps-u_0)|^2\,dx=
  \big(\tfrac{\partial u_0}{\partial x_1}(\e)\big)^2 \int_{\widetilde D} \abs{\nabla (\Phi - (x_1-1)^+)}^2dx.
\]
On the other hand, Lemma \ref{l:secondo_lemma_Phi_1}(iii),  Corollary
\ref{cor:min_phi1}, and \eqref{eq:legame_m_c} imply that 
\[
\int_{\widetilde D} \abs{\nabla (\Phi - (x_1-1)^+)}^2dx=
 \int_\Sigma
\Phi(1,x')\,dx=-2m(\Sigma)={\mathfrak C}(\Sigma).
\]
The proof is thereby complete.
\end{pfn}

\section{The resonant case}\label{sec:resonant-case}

In this section we drop 
 assumption \eqref{eq:54} and treat the case in which  
$\lambda_0$ is a double eigenvalue on $D^+ \cup D^-$ and a simple
eigenvalue on each of the components $D^+$ and $D^-$. To this aim, we
exploit twice the sharp asymptotics provided by Theorem
\ref{teo_asintotico_autovalori} in two domains obtained by attaching
small handles to each chamber $D^+,D^-$.

Besides \eqref{eq:condsigma} and \eqref{eq:p}, 
we assume that 
$p\in C^1(\R^N,\R)\cap L^{\infty}(\R^N)$ satisfies 
\begin{equation}\label{eq:p_risonanza}
p\not\equiv 0\text{ in }D^-,\quad
 p\not\equiv 0\text{ in }D^+,\quad
 p(x)=0\text { for all } 
x\in 
B^-_3 \cup \mathcal C_\eps \cup B^+_3,
\end{equation}
and  that there exist $k_0^+,k_0^-\geq1$ such that
\begin{align}
& \lambda_0 \in \sigma_p(D^+) \cap \sigma_p(D^-), \label{ipotesi_risonanza1}\\
&\label{ipotesi_risonanza2} \lambda_0 = \lambda_{k_0^+}(D^+) 
\text{ is simple on $D^+$ and the
    corresponding eigenfunctions}\\
\notag&\phantom{\lambda_0 = \lambda_{k_0^+}(D^+) }\text{ have in ${\mathbf e}_1
    $ a zero of order $1$},\\
&
\label{ipotesi_risonanza3} \lambda_0 = \lambda_{k_0^-}(D^-) 
\text{ is simple on $D^-$ and the
    corresponding eigenfunctions}\\
\notag&\phantom{\lambda_0 = \lambda_{k_0^-}(D^-) }\text{ have in ${\mathbf 0}
    $ a zero of order $1$}.
\end{align}
Since $\sigma_p(D^+ \cup D^-)= \sigma_p(D^+) \cup \sigma_p(D^-)$, \eqref{ipotesi_risonanza1},
\eqref{ipotesi_risonanza2}, and \eqref{ipotesi_risonanza3}
imply that $\lambda_0$ is a double eigenvalue on $D^+ \cup D^-$ and
hence there exists $\bar k\geq 1$ such that 
\begin{equation}\label{eq:15}
\lambda_0 = \lambda_{\bar k}(D^+ \cup D^-) = \lambda_{\bar k+1}(D^+ \cup D^-) .
\end{equation}
Let $u_0^+\in{\mathcal
    D}^{1,2}(D^+)\setminus\{0\}$ and 
$u_0^-\in{\mathcal
    D}^{1,2}(D^-)\setminus\{0\}$ be the  eigenfunctions 
associated to 
$\lambda_0$ on $D^+$ and $D^-$
respectively, 
  i.e. solving
\begin{equation}\label{eq:u0+-}
\begin{cases}
-\Delta u_0^+=\lambda_0 p u_0^+,&\text{in }D^+,\\
u_0^+=0,&\text{on }\partial D^+,
\end{cases}\quad \begin{cases}
-\Delta u_0^-=\lambda_0 p u_0^-,&\text{in }D^-,\\
u_0^-=0,&\text{on }\partial D^-,
\end{cases}
\end{equation}
such that
\begin{equation}\label{eq:normalu_0+-}
\frac{\partial u_0^+}{\partial x_1}({\mathbf e}_1
)>0,\quad
 \frac{\partial u_0^-}{\partial x_1}({\mathbf 0}
)<0,\quad \int_{D^+}p(x)|u_0^+(x)|^2\,dx=\int_{D^-}p(x)|u_0^-(x)|^2\,dx=1.
\end{equation}
Let us introduce the following domains
\begin{align}
  & D^+_\eps = D^+ \cup \left\{(x_1,x')\in \R\times\R^{N-1}:
    \ \frac34\leq x_1\leq 1,\ \frac{x'}{\eps}\in\Sigma \right\}, \label{D_eps+} \\
  & D^-_\eps = D^- \cup \left\{(x_1,x')\in \R\times\R^{N-1}:
    \ 0\leq x_1\leq \frac14,\ \frac{x'}{\eps}\in\Sigma \right\}, \label{D_eps-} \\
  & \widetilde{\Omega}^\eps = D^+_\eps \cup
  D^-_\eps \label{Omega_eps_tilde}.
\end{align}
We observe that the asymptotics of
eigenvalues stated in Theorem \ref{teo_asintotico_autovalori} can be
proved,  up to minor modifications, replacing the dumbbell perturbed
domain $\Omega^\eps$ defined in \eqref{eq:31}  with either the domain $D^+_\eps$ or $D^+_\eps$, since the proof just
relies on the attachment of a shrinking handle at a point in
which the limit eigenfunction has a zero of order $1$; therefore,
arguing as in the proof of Theorem
\ref{teo_asintotico_autovalori}, we can prove that, under assumptions
\eqref{eq:condsigma}, \eqref{eq:p}, 
\eqref{eq:p_risonanza}, 
\eqref{ipotesi_risonanza1}, \eqref{ipotesi_risonanza2}, and
\eqref{ipotesi_risonanza3}, there holds 
\begin{align}
 & \lambda_{k_0^+}(D^+) = \lambda_{k_0^+}(D_\eps^+) 
+ \eps^N \mathfrak C(\Sigma) \bigg(\dfrac{\partial u_0^+}{\partial
  x_1}(\mathbf e_1)\bigg)^{\!\!2} + o(\eps^N),
 \label{conseguenza1} \\
 & \lambda_{k_0^-}(D^-) = \lambda_{k_0^-}(D_\eps^-) 
+ \eps^N \mathfrak C(\Sigma) \bigg(\dfrac{\partial u_0^-}{\partial x_1}(\mathbf 0)\bigg)^{\!\!2} + o(\eps^N), \label{conseguenza2}
\end{align}
as $\eps\to 0^+$,
where $\mathfrak C(\Sigma)$ is defined in \eqref{def_compliance}.

Under the non-symmetry condition that the normal derivatives of the
limit eigenfunctions at the two junctions are different, we observe 
that  the double eigenvalue $\lambda_0$ is approximated by two
different branches of eigenvalues in
$\widetilde\Omega^\eps$, see figure \ref{fig:2}.
\begin{Proposition}\label{prop_risonanza}
Under  assumptions \eqref{eq:condsigma}, \eqref{eq:p}, 
\eqref{eq:p_risonanza}, 
\eqref{ipotesi_risonanza1}, \eqref{ipotesi_risonanza2}, and
\eqref{ipotesi_risonanza3}, let $u_0^+$ and $u_0^-$ be as in
\eqref{eq:u0+-} and \eqref{eq:normalu_0+-}, and let $D_\eps^+$,
$D_\eps^-$ be as in 
\eqref{D_eps+}, \eqref{D_eps-} respectively. 
If 
\begin{equation}\label{derivate_risonanza}
 \abs{\dfrac{\partial u_0^+}{\partial x_1}(\mathbf e_1)} >
 \abs{\dfrac{\partial u_0^-}{\partial x_1}(\mathbf 0)}, 
\end{equation}
then, for $\eps$ sufficiently small, 
\[
\lambda_{k_0^+}(D_\eps^+) = \lambda_{\bar k}(\widetilde{\Omega}^\eps)\quad\text{and}\quad
\lambda_{k_0^-}(D_\eps^-) = \lambda_{\bar
  k+1}(\widetilde{\Omega}^\eps),
\]
where $\widetilde\Omega^\eps$ is defined in \eqref{Omega_eps_tilde}
and $\bar k$ is as in \eqref{eq:15}.
\end{Proposition}

\begin{figure}[h]
\begin{center}
\psset{unit=0.5cm}
\begin{pspicture}(0,0)(10,5)
\psaxes*[labels=none,ticks=none]{->}(10,5)
\uput[270](10,0){$\eps$}
\uput[180](0,5){$\lambda_\eps$}
\uput[180](0,4){$\lambda_0$}
\uput[0](5,3){$\lambda_{k_0^-}(D_\eps^-)$}
\uput[0](4,0.5){$\lambda_{k_0^+}(D_\eps^+)$}
\psplot{0}{4}{ x 2 exp -0.2 mul 4 add}
\psplot{0}{5}{ x 2 exp -0.05 mul 4 add}
\psline[linewidth=.5pt,linestyle=dashed](0,4)(5,4)
\psdot[dotstyle=*](0,4)
\end{pspicture}
\end{center}
\caption{Two different branches of eigenvalues approximating the same double eigenvalue.}\label{fig:2}
\end{figure}
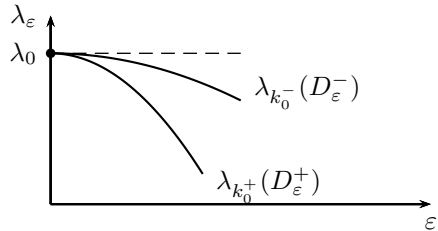

\begin{pf}
We note that 
$\sigma_p(\widetilde{\Omega}_\eps) = \sigma_p(D_\eps^+)\cup \sigma_p(D_\eps^-)$. 
Expansions \eqref{conseguenza1} and \eqref{conseguenza2} together with
assumptions \eqref{ipotesi_risonanza2}, \eqref{ipotesi_risonanza3} yield 
\[
\lambda_{k_0^+}(D_\eps^+) - \lambda_{k_0^-}(D_\eps^-) = \eps^N
\mathfrak C(\Sigma) \left( \bigg(\dfrac{\partial u_0^-}{\partial
    x_1}(\mathbf 0)\bigg)^{\!\!2} - \bigg(\dfrac{\partial
    u_0^+}{\partial x_1}(\mathbf e_1)\bigg)^{\!\!2} \right) +
o(\eps^N), \quad \text{as }\eps\to 0^+,
\]
and hence \eqref{derivate_risonanza} implies that 
\[
\lambda_{k_0^+}(D_\eps^+) - \lambda_{k_0^-}(D_\eps^-)<0
\]
for $\eps$ sufficiently small, which gives the conclusion in view of the convergence of
eigenvalues on $\widetilde\Omega^\eps$ proved by Daners in \cite{daners}.
\end{pf}

We now evaluate the difference between corresponding eigenvalues on
the dumbbell domain $\Omega^\eps$ and on the disconnected  domain
$\widetilde \Omega^\eps$.

\begin{Lemma}\label{l:tildeomegaeps}
  For $\ell=\bar k, \bar k +1$, $ \lambda_{\ell} (\widetilde{\Omega}^\eps) =
  \lambda_{\ell}(\Omega^\eps) + 
O\big
  (\eps^{\frac{N-1}2}e^{-\frac{\sqrt{\lambda_1(\Sigma)}}{32\eps}}\big)$ as $\eps \to 0^+$.
\end{Lemma}
\begin{pf}
For the sake of brevity, we prove the lemma just for $\ell=\bar k$;  the
proof for $\ell=\bar k+1$ is similar. 
By the Courant-Fisher \emph{minimax characterization} of eigenvalues
we have that 
\begin{equation}\label{eq:10bis}
  \lambda_{\bar k} (\widetilde{\Omega}^\eps) = \min\!\bigg\{\!\max_{u\in F\setminus \{0\}}\!
\dfrac{\int_{\widetilde{\Omega}^\eps} \abs{\nabla u}^2dx}{\int_{\widetilde{\Omega}^\eps}
  pu^2\,dx}:F \text{ is a subspace of $\Di{\widetilde{\Omega}^\eps}$ \text{such
    that }$\dim F= \bar k$}\bigg\}.
\end{equation}
 Let $\eta\in C^\infty(\R^N)$ be a smooth cut-off function such that 
\begin{align*}
&\eta\equiv 1 \text{ in }
 \{(x_1,x')\in\R^N: \text{either }x_1\leq 1/8\text{ or }x_1\geq
 7/8\},\\
 &\eta\equiv 0 \text{ in }
 \{(x_1,x')\in\R^N: 1/4\leq x_1\leq 3/4\},\\
 &0\leq\eta\leq1\text{ in }\R^N.
\end{align*}
For every 
$j=1,2,\dots,\bar k$ and $\eps$ small, we fix an eigenfunction $v_j^\eps\in \mathcal
D^{1,2}(\Omega^\eps)$ associated to $\lambda_j(\Omega^\eps)$  on
$\Omega^\eps$  such that
$\int_{\Omega^\eps}p|v_j^\eps|^2dx =1$
and $\int_{\Omega^\eps}\nabla v_j^\eps\cdot \nabla v_i^\eps\,dx=0$ if
$i\neq j$.
Choosing the $\bar k$-dimesional space $F=
\mathop{\rm span} \{ \eta v_1^\eps,\ldots,\eta v_{\bar k}^\eps \}$ in
\eqref{eq:10bis} we obtain that
\begin{align}\label{eq:48}
 0\leq \lambda_{\bar k} (\widetilde{\Omega}^\eps) - \lambda_{\bar k}(\Omega^\eps) &\leq
\max_{\substack{(\alpha_1,\dots,\alpha_{\bar k})\in
    \R^{\bar k}\\
\sum_{j=1}^{\bar k}\alpha_j^2=1}}
  \dfrac{\int_{\widetilde{\Omega}^\eps} |\nabla (\sum_{j=1}^{\bar k}\alpha_j \eta v_j^\eps )|^2dx}
 {\int_{\widetilde{\Omega}^\eps} p
(\sum_{j=1}^{\bar k}\alpha_j \eta v_j^\eps )^2dx} 
- \int_{\Omega^\eps} \abs{\nabla v_{\bar k}^\eps}^2dx \\
\notag& = \max_{\substack{(\alpha_1,\dots,\alpha_{\bar k})\in
    \R^{\bar k}\\
\sum_{j=1}^{\bar k}\alpha_j^2=1}}
\dfrac{\sum_{i,j=1}^{\bar k}\alpha_i\alpha_j
  \int_{\widetilde{\Omega}^\eps}
 \nabla(\eta v_i^\eps)\cdot\nabla(\eta v_j^\eps)\,dx}
{\sum_{i,j=1}^{\bar k}\alpha_i\alpha_j \int_{\widetilde{\Omega}^\eps} p \eta^2 v_i^\eps v_j^\eps\,dx}
- \int_{\Omega^\eps} \abs{\nabla v_{\bar k}^\eps}^2dx .
\end{align}
From Corollary \ref{stime_integrali_vjeps} it follows that, for every 
$j=1,\ldots,\bar k$, 
\begin{equation}\label{eq:49}
\int_{\widetilde{\Omega}^\eps} p \eta^2 |v_j^\eps|^2 dx
= 1 - \int_{
(1/8,7/8)\times(\eps\Sigma)} p (1-\eta^2) |v_j^\eps|^2dx =
 1+O\Big(\eps^{N-1}e^{-\frac{\sqrt{\lambda_1(\Sigma)}}{16\eps}}\Big) \quad \text{as }\eps\to 0^+,
\end{equation}
whereas, exploiting the orthogonality of eigenfunctions,  if $i\neq j$
we have that 
\begin{equation}\label{eq:50}
\int_{\widetilde{\Omega}^\eps} p \eta^2 v_i^\eps v_j^\eps \,dx
=- \int_{
(1/8,7/8)\times(\eps\Sigma)} p (1-\eta^2) v_i^\eps v_j^\eps\,dx =
O \Big (\eps^{N-1}e^{-\frac{\sqrt{\lambda_1(\Sigma)}}{16\eps}}\Big) \quad
\text{as }\eps\to 0^+.
\end{equation}
From \eqref{eq:48}, \eqref{eq:49}, and \eqref{eq:50} it follows that 
\begin{align}\label{eq:57}
  0&\leq \lambda_{\bar k} (\widetilde{\Omega}^\eps) - \lambda_{\bar
    k}(\Omega^\eps) \\
\notag&\leq \max_{\substack{(\alpha_1,\dots,\alpha_{\bar
        k})\in
      \R^{\bar k}\\
      \sum_{j=1}^{\bar k}\alpha_j^2=1}} \sum_{i,j=1}^{\bar
    k}\alpha_i\alpha_j \int_{\widetilde{\Omega}^\eps} \nabla(\eta
  v_i^\eps)\cdot\nabla(\eta v_j^\eps)\,dx
  - \int_{\Omega^\eps} \abs{\nabla v_{\bar k}^\eps}^2dx +O \Big
  (\eps^{N-1}e^{-\frac{\sqrt{\lambda_1(\Sigma)}}{16\eps}}\Big) \\
 \notag & = \max_{\substack{(\alpha_1,\dots,\alpha_{\bar
        k})\in
      \R^{\bar k}\\
      \sum_{j=1}^{\bar k}\alpha_j^2=1}}
\bigg\{ \alpha_{\bar k}^2
    \bigg(\int_{\widetilde{\Omega}^\eps} \abs{\nabla(\eta
        v_{\bar k}^\eps)}^2dx
      - \int_{\Omega^\eps} \abs{\nabla v_{\bar k}^\eps}^2dx\bigg) \\
\notag&\qquad\qquad+ \sum_{j=1}^{\bar k-1} \alpha_j^2 \bigg(
      \int_{\widetilde{\Omega}^\eps} \abs{\nabla(\eta
        v_j^\eps)}^2dx
      - \int_{\Omega^\eps} \abs{\nabla v_{\bar k}^\eps}^2dx \bigg) \\
 \notag &\qquad\qquad+\sum_{i\neq j}\alpha_i\alpha_j
    \int_{\widetilde{\Omega}^\eps} \nabla(\eta
    v_i^\eps)\cdot\nabla(\eta v_j^\eps)\,dx \bigg\} +O \Big
  (\eps^{N-1}e^{-\frac{\sqrt{\lambda_1(\Sigma)}}{16\eps}}\Big), \quad
  \text{as }\eps\to 0^+.
\end{align}
From Corollary \ref{stime_integrali_vjeps} it follows that 
\begin{align}\label{eq:51}
\int_{\widetilde{\Omega}^\eps} &\abs{\nabla(\eta
        v_{\bar k}^\eps)}^2dx
      - \int_{\Omega^\eps} \abs{\nabla v_{\bar k}^\eps}^2dx\\
\notag& =  \int_{\widetilde{\Omega}^\eps} \Big(\abs{\nabla\eta}^2 |v_{\bar
  k}^\eps|^2 + \eta^2 \abs{\nabla v_{\bar k}^\eps}^2 
+ 2\eta v_{\bar k}^\eps \nabla \eta\cdot\nabla v_{\bar
  k}^\eps\Big)\,dx - \int_{\Omega^\eps} \abs{\nabla v_{\bar k}^\eps}^2dx \\
\notag& \leq O \Big
  (\eps^{\frac{N-1}2}e^{-\frac{\sqrt{\lambda_1(\Sigma)}}{32\eps}}\Big)
  + \int_{\Omega^\eps} (\eta^2-1)\abs{\nabla v_{\bar k}^\eps}^2dx
\leq O \Big
  (\eps^{\frac{N-1}2}e^{-\frac{\sqrt{\lambda_1(\Sigma)}}{32\eps}}\Big),\quad \text{as }\eps\to 0^+,
\end{align}
\begin{align}\label{eq:55}
 \int_{\widetilde{\Omega}^\eps} \abs{\nabla(\eta v_j^\eps)}^2dx 
- \int_{\Omega^\eps} \abs{\nabla v_{\bar k}^\eps}^2dx 
& =  \int_{\widetilde{\Omega}^\eps} \eta^2\abs{\nabla v_j^\eps}^2dx
- \int_{\Omega^\eps} \abs{\nabla v_{\bar k}^\eps}^2dx +O\Big
  (\eps^{\frac{N-1}2}e^{-\frac{\sqrt{\lambda_1(\Sigma)}}{32\eps}}\Big)\\
\notag&\leq \lambda_j(\Omega^\eps) -\lambda_{\bar k}(\Omega^\eps) +O\Big
  (\eps^{\frac{N-1}2}e^{-\frac{\sqrt{\lambda_1(\Sigma)}}{32\eps}}\Big) \\
\notag&\leq O\Big
  (\eps^{\frac{N-1}2}e^{-\frac{\sqrt{\lambda_1(\Sigma)}}{32\eps}}\Big), \quad \text{as }\eps\to 0^+,
\end{align}
for all $j<\bar k$, and, if $i\neq j$, 
\begin{align}\label{eq:56}
\int_{\widetilde{\Omega}^\eps} &\nabla(\eta v_i^\eps)\cdot\nabla(\eta v_j^\eps)\,dx\\
\notag& = \int_{\widetilde{\Omega}^\eps} \eta^2 \nabla v_i^\eps\cdot\nabla
v_j^\eps\,dx 
+ \int_{\widetilde{\Omega}^\eps} \eta\nabla\eta\cdot(v_i^\eps\nabla
v_j^\eps+v_j^\eps\nabla
v_i^\eps)\,dx
+\int_{\widetilde{\Omega}^\eps} |\nabla\eta|^2 v_i^\eps v_j^\eps\,dx 
\\
\notag& = \int_{(1/8,7/8)\times(\eps\Sigma)} (\eta^2-1) \nabla v_i^\eps\cdot\nabla v_j^\eps\,dx+O\Big
  (\eps^{\frac{N-1}2}e^{-\frac{\sqrt{\lambda_1(\Sigma)}}{32\eps}}\Big) \\
\notag& = -2\int_{(1/8,7/8)\times(\eps\Sigma)} \eta v_i^\eps \nabla v_j^\eps\cdot \nabla \eta\,dx+O\Big
  (\eps^{\frac{N-1}2}e^{-\frac{\sqrt{\lambda_1(\Sigma)}}{32\eps}}\Big)  \\
\notag& = O\Big
  (\eps^{\frac{N-1}2}e^{-\frac{\sqrt{\lambda_1(\Sigma)}}{32\eps}}\Big), \quad \text{as }\eps\to 0^+,
\end{align}
where in the third equality we have tested  $-\Delta v_j^\eps =
\lambda_j^\eps p v_j^\eps $ 
 with 
$(\eta^2-1)v_i^\eps$ for $i\neq j$ and  integrated by parts over
$(1/8,7/8)\times(\eps\Sigma)$, using assumption \eqref{eq:p_risonanza}.

From \eqref{eq:57}, \eqref{eq:51}, \eqref{eq:55}, and \eqref{eq:56} it
follows that 
\begin{align*}
   0&\leq \lambda_{\bar k} (\widetilde{\Omega}^\eps) - \lambda_{\bar
    k}(\Omega^\eps)\leq 
O\Big
  (\eps^{\frac{N-1}2}e^{-\frac{\sqrt{\lambda_1(\Sigma)}}{32\eps}}\Big), \quad \text{as }\eps\to 0^+,
\end{align*}
thus yielding the conclusion.
\end{pf}

\begin{remark}\label{r:splitting_polinomiale}
  We observe that, under assumption \eqref{derivate_risonanza},
  expansions \eqref{conseguenza1}, \eqref{conseguenza2}, Proposition
  \ref{prop_risonanza} and Lemma \ref{l:tildeomegaeps} imply that the
  splitting of the two subsequent eigenvalues $\lambda_{\bar
    k}(\Omega^\eps),\lambda_{\bar k+1}(\Omega^\eps)$ approximating the
  same double eigenvalue $\lambda_{\bar
    k}(D^-\cup D^+)=\lambda_{\bar k+1}(D^-\cup D^+)$ has a polynomial vanishing order, i.e.
\[
\lambda_{\bar k+1}(\Omega^\eps)-\lambda_{\bar
    k}(\Omega^\eps)=\eps^N
\mathfrak C(\Sigma) \left(  \bigg(\dfrac{\partial
    u_0^+}{\partial x_1}(\mathbf e_1)\bigg)^{\!\!2}-\bigg(\dfrac{\partial u_0^-}{\partial
    x_1}(\mathbf 0)\bigg)^{\!\!2} \right) +
o(\eps^N), \quad \text{as }\eps\to 0^+.
\]  
We emphasize that non-symmetry assumption \eqref{derivate_risonanza}
 is crucial for having a polynomial splitting: indeed it was proved in
 \cite{BHM} that in the case of a symmetric dumbbell domain the
 splitting of the first two eigenvalues vanishes with exponential rate.
\end{remark}

Combining \eqref{conseguenza1}, \eqref{conseguenza2} with Proposition
\ref{prop_risonanza} and Lemma \ref{l:tildeomegaeps} we derive the
asymptotics of the eigenvalues $ \lambda_{\bar
    k}(\Omega^\eps)$ and $ \lambda_{\bar
    k+1}(\Omega^\eps)$ thus proving Theorem \ref{t:reso}.

\begin{pfn}{Theorem \ref{t:reso}}
From  Proposition
\ref{prop_risonanza} and Lemma \ref{l:tildeomegaeps} we have that 
\begin{align*}
\lambda_0-\lambda_{\bar
    k}(\Omega^\eps)&=(\lambda_0-\lambda_{\bar
    k}(\widetilde\Omega^\eps))+(\lambda_{\bar
    k}(\widetilde\Omega^\eps)-\lambda_{\bar
    k}(\Omega^\eps)\big)\\
&=(\lambda_0-\lambda_{k_0^+}(D_\eps^+))+O\big
  (\eps^{\frac{N-1}2}e^{-\frac{\sqrt{\lambda_1(\Sigma)}}{32\eps}}\big),\quad\text{as
  }\eps\to 0^+,
\end{align*}
and
\begin{align*}
\lambda_0-\lambda_{\bar
    k+1}(\Omega^\eps)&=(\lambda_0-\lambda_{\bar
    k+1}(\widetilde\Omega^\eps))+(\lambda_{\bar
    k+1}(\widetilde\Omega^\eps)-\lambda_{\bar
    k+1}(\Omega^\eps)\big)\\
&=(\lambda_0-\lambda_{k_0^-}(D_\eps^-))+O\big
  (\eps^{\frac{N-1}2}e^{-\frac{\sqrt{\lambda_1(\Sigma)}}{32\eps}}\big),\quad\text{as
  }\eps\to 0^+.
\end{align*}
Hence, by \eqref{conseguenza1} and \eqref{conseguenza2} we obtain 
\begin{align*}
\lambda_0-\lambda_{\bar
    k}(\Omega^\eps)=\eps^N \mathfrak C(\Sigma) \big(\tfrac{\partial u_0^+}{\partial
  x_1}(\mathbf e_1)\big)^{2} + o(\eps^N)\quad\text{and}\quad
\lambda_0-\lambda_{\bar
    k+1}(\Omega^\eps)=\eps^N \mathfrak C(\Sigma) \big(\tfrac{\partial u_0^-}{\partial
  x_1}(\mathbf 0)\big)^{2} + o(\eps^N)
\end{align*}
thus completing the proof.
\end{pfn}

A key ingredient for the proof of Theorem \ref{t:reso_autofun} is the
following spectral estimate.
\begin{Lemma}\label{l:spectral_estimate}
  Under the same assumptions as in Theorem \ref{t:reso}, let
\begin{align*}
&L_\eps^+:=-\Delta-\lambda_{\bar
    k+1}(\Omega^\eps)p:\Di{D_\eps^+}\to (\Di{D_\eps^+})^\star, \\
&L_\eps^-:=-\Delta-\lambda_{\bar
    k}(\Omega^\eps)p:\Di{D_\eps^-}\to (\Di{D_\eps^-})^\star,
\end{align*}
be defined as 
\begin{align*}
&\phantom{a}_{(\mathcal D^{1,2}(D_\eps^+))^\star}\big\langle  L_\eps^+ u,v
\big\rangle_{\mathcal D^{1,2}(D_\eps^+)}=\int_{D_\eps^+}\nabla
u\cdot\nabla v\,dx-\lambda_{\bar
    k+1}(\Omega^\eps)\int_{D_\eps^+} puv\,dx, \quad u,v\in
  \Di{D_\eps^+},\\
&\phantom{a}_{(\mathcal D^{1,2}(D_\eps^-))^\star}\big\langle  L_\eps^- u,v
\big\rangle_{\mathcal D^{1,2}(D_\eps^-)}=\int_{D_\eps^-}\nabla
u\cdot\nabla v\,dx-\lambda_{\bar k}(\Omega^\eps)\int_{D_\eps^-} puv\,dx, \quad u,v\in
  \Di{D_\eps^-},
\end{align*}
where $(\mathcal D^{1,2}(D_\eps^+))^\star$ is 
the dual space of $\mathcal D^{1,2}(D_\eps^+)$ and 
 $(\mathcal D^{1,2}(D_\eps^-))^\star$ is 
the dual space of $\mathcal D^{1,2}(D_\eps^-)$.
Then, for $\eps$ sufficiently small, $L_\eps^+$ and $L_\eps^-$ are
invertible and 
\begin{align*}
&  \|(L_\eps^+)^{-1}\|_{\mathcal
  L((\Di{D_\eps^+})^\star,\Di{D_\eps^+})}=O(\eps^{-N}), \quad\text{as }\eps\to0^+,\\
&  \|(L_\eps^-)^{-1}\|_{\mathcal
  L((\Di{D_\eps^-})^\star,\Di{D_\eps^-})}=O(\eps^{-N}), \quad\text{as }\eps\to0^+.
\end{align*}
\end{Lemma}
\begin{pf}
  From \eqref{conseguenza1}, \eqref{conseguenza2}, and Theorem
  \ref{t:reso}, it follows that, for $\eps$ sufficiently small, 
\begin{align}
\label{eq:59}\mathop{\rm dist\,}(\lambda_{\bar
    k+1}(\Omega^\eps),\sigma_p(D_\eps^+))&=\lambda_{\bar
    k+1}(\Omega^\eps)-\lambda_{k_0^+}(D_\eps^+)\\
\notag&=
\eps^N
\mathfrak C(\Sigma) \left(  \bigg(\dfrac{\partial
    u_0^+}{\partial x_1}(\mathbf e_1)\bigg)^{\!\!2}-\bigg(\dfrac{\partial u_0^-}{\partial
    x_1}(\mathbf 0)\bigg)^{\!\!2} \right) +
o(\eps^N), \quad \text{as }\eps\to 0^+,
\end{align}
and 
\begin{align}\label{eq:60}
\mathop{\rm dist\,}(\lambda_{\bar
    k}(\Omega^\eps),\sigma_p(D_\eps^-))&=\lambda_{k_0^-}(D_\eps^-)-\lambda_{\bar k}(\Omega^\eps)\\
\notag&=
\eps^N
\mathfrak C(\Sigma) \left(  \bigg(\dfrac{\partial
    u_0^+}{\partial x_1}(\mathbf e_1)\bigg)^{\!\!2}-\bigg(\dfrac{\partial u_0^-}{\partial
    x_1}(\mathbf 0)\bigg)^{\!\!2} \right) +
o(\eps^N), \quad \text{as }\eps\to 0^+.
\end{align}
In particular, by assumption \eqref{derivate_risonanza}, we have that $\lambda_{\bar
    k+1}(\Omega^\eps)\not\in\sigma_p(D_\eps^+)$ and $\lambda_{\bar
    k}(\Omega^\eps)\not\in\sigma_p(D_\eps^-)$, and therefore $L_\eps^+$ and $L_\eps^-$ are
invertible. Moreover, by classical spectral estimates, \eqref{eq:59},
and \eqref{eq:60}, it turns out
that 
\begin{align*}
  \|(L_\eps^+)^{-1}\|_{\mathcal
  L((\Di{D_\eps^+})^\star,\Di{D_\eps^+})}&\leq \sup_{\lambda\in \sigma_p(D_\eps^+)}\frac{\lambda}{|\lambda-\lambda_{\bar
    k+1}(\Omega^\eps)|}\\
&\leq 
1+\frac{\lambda_{\bar k+1}(\Omega^\eps)}{\mathop{\rm dist\,}(\lambda_{\bar
    k+1}(\Omega^\eps),\sigma_p(D_\eps^+))}=O(\eps^{-N}),\quad\text{as }\eps\to0^+,\\
\|(L_\eps^-)^{-1}\|_{\mathcal
  L((\Di{D_\eps^-})^\star,\Di{D_\eps^-})}&\leq \sup_{\lambda\in \sigma_p(D_\eps^-)}\frac{\lambda}{|\lambda-\lambda_{\bar
    k}(\Omega^\eps)|}\\
&\leq 
1+\frac{\lambda_{\bar k}(\Omega^\eps)}{\mathop{\rm dist\,}(\lambda_{\bar
    k}(\Omega^\eps),\sigma_p(D_\eps^-))}=O(\eps^{-N}),\quad\text{as }\eps\to0^+.
\end{align*}
The proof is thereby complete.
\end{pf}

\begin{pfn}{Theorem \ref{t:reso_autofun}}
  Due to simplicity of the eigenvalue
  $\lambda_0=\lambda_{k_0^+}(D^+)=\lambda_{k_0^-}(D^-)$ on each
  component $D^-$ and $D^+$, it is is enough to prove estimates
  \eqref{eq:62} and \eqref{eq:7} for any family of eigenfunctions 
$v_{\bar k}^\eps \in{\mathcal
    D}^{1,2}(\Omega^\eps)$ on $\Omega^\eps$
associated to 
$\lambda_{\bar k}(\Omega^\eps)$ 
and any family of eigenfunctions 
$v_{\bar k+1}^\eps \in{\mathcal
    D}^{1,2}(\Omega^\eps)$ on $\Omega^\eps$
associated to 
$\lambda_{\bar k+1}(\Omega^\eps)$ 
such that 
\[
\int_{\Omega^\eps}p(x) |v_{\bar k}^\eps(x)|^2\,dx=1,\quad
\int_{\Omega^\eps}p(x) |v_{\bar k+1}^\eps(x)|^2\,dx=1.
\]
 We prove only \eqref{eq:62}, being the
proof of \eqref{eq:7} analogous.
 Let $\eta\in C^\infty(\R^N)$ be a smooth cut-off function such that 
\begin{equation*}
\eta\equiv 1 \text{ in }
 \{(x_1,x')\in\R^N:x_1\leq 1/8\},\quad
 \eta\equiv 0 \text{ in }
 \{(x_1,x')\in\R^N: x_1\geq 1/4\},\quad 0\leq\eta\leq1\text{ in }\R^N.
\end{equation*}
A direct computation shows that, letting $L_\eps^-$ as in Lemma
\ref{l:spectral_estimate},  
\[
L_\eps^-(\eta v_{\bar k}^\eps)=h_\eps
\]
where $h_\eps\in (\Di{D_\eps^-})^\star$ acts as
\[
\phantom{a}_{(\mathcal D^{1,2}(D_\eps^-))^\star}\big\langle  h_\eps,v
\big\rangle_{\mathcal D^{1,2}(D_\eps^-)}=\int_{D_\eps^-}
(\Delta \eta v+2\nabla \eta\cdot \nabla v) v_{\bar k}^\eps\,dx,\quad
\text{for every }v\in \Di{D_\eps^-}.
\]
From Corollary  \ref{stime_integrali_vjeps} it follows that 
\[
\|h_\eps\|_{ (\Di{D_\eps^-})^\star}=O\big
  (\eps^{\frac{N-1}2}e^{-\frac{\sqrt{\lambda_1(\Sigma)}}{32\eps}}\big),\quad\text{as
  }\eps\to 0^+,
\]
and then Lemma \ref{l:spectral_estimate} implies that 
\begin{align*}
\|\eta v_{\bar
  k}^\eps\|_{\Di{D_\eps^-}}&=\|(L_\eps^-)^{-1}(h_\eps)\|_{\Di{D_\eps^-}}
\leq \|(L_\eps^-)^{-1}\|_{\mathcal
  L((\Di{D_\eps^-})^\star,\Di{D_\eps^-})}\|h_\eps\|_{
  (\Di{D_\eps^-})^\star}\\
&=O(\eps^{-N}) O\big
  (\eps^{\frac{N-1}2}e^{-\frac{\sqrt{\lambda_1(\Sigma)}}{32\eps}}\big)=
 O\big
  (\eps^{-\frac{N+1}2}e^{-\frac{\sqrt{\lambda_1(\Sigma)}}{32\eps}}\big),\quad\text{as
  }\eps\to 0^+.
\end{align*}
Estimate \eqref{eq:62} is thereby proved.
\end{pfn}

\begin{pfn}{Theorem \ref{t:reso_autofun_conv}}
We prove only \eqref{eq:6},  being the proof of \eqref{eq:16}
analogous. To this aim, we first observe that, in view of the
simplicity of the eigenvalue $\lambda_{k_0^+}(D^+)$ and by Theorem
\ref{teo_asintotico_autofunzioni} (adapted to the easier case of the
perturbed domain $D_\eps^+$), there exists  
a family of eigenfunctions 
$v_\eps^+ \in{\mathcal
    D}^{1,2}(D_\eps^+)$ on $D_\eps^+$
associated to 
$\lambda_{k_0^+}(D_\eps^+)$ such that 
\begin{equation*}
\begin{cases}
-\Delta v_\eps^+=\lambda_{k_0^+}(D_\eps^+) p v_\eps^+,&\text{in }D_\eps^+,\\
 v_\eps^+ =0,&\text{on }\partial D_\eps^+,\\
\int_{D_\eps^+}p(x) |v_\eps^+ (x)|^2\,dx=1,\end{cases}\quad
\end{equation*}
and 
\begin{equation}\label{eq:18}
\lim_{\eps\to 0^+}
\eps^{-N}\|v_\eps^+-u_0^+\|^2_{\Di{\R^N}}
=
\bigg(\frac{\partial u_0^+}{\partial x_1}(\e)\bigg)^{\!\!2}
\mathfrak C(\Sigma).
\end{equation}
In view of \eqref{eq:18}, Theorem \ref{t:reso_autofun}, and 
Corollary \ref{stime_integrali_vjeps}, to prove \eqref{eq:6} it is
enough to show that 
\begin{equation}\label{eq:63}
  \|\eta v_{\bar
    k}^\eps-v_\eps^+\|^2_{\Di{D_\eps^+}}=o(\eps^N),\quad\text{as
  }\eps\to 0^+,
\end{equation}
for some $\eta\in C^\infty(\R^N)$ being a smooth cut-off function such that 
\begin{equation*}
\eta\equiv 1 \text{ in }
 \{(x_1,x')\in\R^N:x_1\geq 7/8\},\quad
 \eta\equiv 0 \text{ in }
 \{(x_1,x')\in\R^N: x_1\leq 3/4\},\quad 0\leq\eta\leq1\text{ in }\R^N.
\end{equation*}
To prove \eqref{eq:63}, we argue as in section
\ref{sec:rate-conv-eigenf} and  consider the operator
\begin{align*}
  F_\eps:\  &\R \times \Di{D_\eps^+}  \longrightarrow  \R \times
  (\Di{D_\eps^+})^\star\\
 &(\lambda,u) \quad \longmapsto \quad (\|u\|_{\Di{D_\eps^+}}^2 -\lambda_{k_0^+}(D_\eps^+), -\Delta
  u-\lambda p u),
\end{align*}
   where, for all
 $u\in \Di{D_\eps^+}$,  $-\Delta
  u-\lambda pu\in (\Di{D_\eps^+})^\star$ acts as 
\[
\phantom{a}_{(\mathcal D^{1,2}(D_\eps^+))^\star}\big\langle 
-\Delta
  u-\lambda p u
,
v
\big\rangle_{\mathcal D^{1,2}(D_\eps^+)}=\int_{D_\eps^+}\nabla u(x)\cdot\nabla
v(x)\,dx-\lambda \int_{D_\eps^+} p(x)u(x)v(x)\,dx.
\]
We observe that 
$F_\eps(\lambda_{k_0^+}(D_\eps^+),v_\eps^+)=(0,0)$ and $F_\eps$ is
Frech\'{e}t-differentiable at $(\lambda_{k_0^+}(D_\eps^+),v_\eps^+)$;
moreover, arguing as in the proof of Lemma \ref{F_frechet}, we can
prove that, since $\lambda_{k_0^+}(D_\eps^+)$ is simple per small
$\eps$, the 
Frech\'{e}t-differential $dF_\eps(\lambda_{k_0^+}(D_\eps^+),v_\eps^+)\in \mathcal L(
\R \times \Di{D_\eps^+},\R \times (\Di{D_\eps^+})^\star)$ is
invertible. Due to the fact that $\lambda_{k_0^+}(D_\eps^+)$ converges
to a simple eigenvalue on $D^+$, it is also easy to verify that 
\begin{equation}\label{eq:64}
\|(dF_\eps(\lambda_{k_0^+}(D_\eps^+),v_\eps^+)^{-1}\|_{ \mathcal L(
\R \times (\Di{D_\eps^+})^\star,\R \times
\Di{D_\eps^+})}=O(1),\quad\text{as }\eps\to 0^+.
\end{equation}
Since 
\begin{multline*}
 F_\eps(\lambda_{\bar k}(\Omega^\eps),\eta v_{\bar k}^{\eps}) =
 dF_\eps(\lambda_{k_0^+}(D_\eps^+),v_\eps^+) 
(\lambda_{\bar k}(\Omega^\eps)-\lambda_{k_0^+}(D_\eps^+), \eta v_{\bar
  k}^{\eps}-v_\eps^+)\\
+ o\big(|\lambda_{\bar k}(\Omega^\eps)-\lambda_{k_0^+}(D_\eps^+)| + \|
 \eta v_{\bar
  k}^{\eps}-v_\eps^+\|_{\Di{D_\eps^+}}\big),\quad\text{as }\eps\to 0^+,
\end{multline*}
from \eqref{eq:64} we deduce that, for $\eps$ small, 
\begin{equation}\label{eq:65}
|\lambda_{\bar k}(\Omega^\eps)-\lambda_{k_0^+}(D_\eps^+)| + \| \eta
v_{\bar k}^{\eps}-v_\eps^+\|_{\Di{D_\eps^+}}\leq {\rm
  const\,}\|F_\eps(\lambda_{\bar k}(\Omega^\eps),\eta v_{\bar
  k}^{\eps})\|_{\R \times (\Di{D_\eps^+})^\star }.
\end{equation}
We have that 
\begin{equation}\label{eq:66}
F_\eps(\lambda_{\bar k}(\Omega^\eps),\eta v_{\bar
  k}^{\eps}) = (\mu_\eps,w_\eps)
  = 
\big(\|\eta v_{\bar
  k}^{\eps}\|_{\Di{D_\eps^+}}^2 -\lambda_{k_0^+}(D_\eps^+), -\Delta
  (\eta v_{\bar
  k}^{\eps})-\lambda_{\bar k}(\Omega^\eps) p \eta v_{\bar
  k}^{\eps}\big).
    \end{equation}
By direct calculations and using Proposition \ref{prop_risonanza},
Lemma \ref{l:tildeomegaeps}, Corollary  \ref{stime_integrali_vjeps},
and Theorem  \ref{t:reso_autofun}, we have that 
\begin{align}\label{eq:67}
  \mu_\eps&=\|\eta v_{\bar k}^{\eps}\|_{\Di{D_\eps^+}}^2
  -\lambda_{k_0^+}(D_\eps^+)=\lambda_{\bar
    k}(\Omega^\eps)-\lambda_{k_0^+}(D_\eps^+)+\int_{\Omega^{\eps}}|\nabla\eta|^2|v_{\bar
    k}^{\eps}|^2\,dx\\
\notag&\quad+2 \int_{\Omega^{\eps}}
v_{\bar k}^{\eps}\eta\nabla\eta\cdot\nabla v_{\bar
    k}^{\eps}\,dx+\int_{D^-\cup
    ([0,7/8)\times(\eps\Sigma))}(\eta^2-1)|\nabla v_{\bar
    k}^{\eps}|^2\,dx\\
\notag&=\lambda_{\bar
    k}(\Omega^\eps)-\lambda_{k_0^+}(D_\eps^+)+\int_{D^-\cup
    ((3/4,7/8)\times(\eps\Sigma))}|\nabla\eta|^2|v_{\bar
    k}^{\eps}|^2\,dx\\
\notag&\quad+\lambda_{\bar
    k}(\Omega^\eps)\int_{D^-\cup
    ([0,7/8)\times(\eps\Sigma))}p(\eta^2-1)|v_{\bar k}^{\eps}|^2dx=o(\eps^N),\quad\text{as }\eps\to 0^+.
\end{align}
In order to estimate $\|w_\eps\|_{(\Di{D_\eps^+})^\star}$ we observe
that $w_\eps\in (\Di{D_\eps^+})^\star$ acts as
\[
\phantom{a}_{(\mathcal D^{1,2}(D_\eps^+))^\star}\big\langle  w_\eps,v
\big\rangle_{\mathcal D^{1,2}(D_\eps^+)}=\int_{D_\eps^+}
(\Delta \eta v+2\nabla \eta\cdot \nabla v) v_{\bar k}^\eps\,dx,\quad
\text{for every }v\in \Di{D_\eps^+}.
\]
From Corollary  \ref{stime_integrali_vjeps} it follows that 
\begin{equation}\label{eq:68}
\|w_\eps\|_{ (\Di{D_\eps^+})^\star}=o(\eps^N),\quad\text{as
  }\eps\to 0^+.
\end{equation}
Combining \eqref{eq:65}, \eqref{eq:66}, \eqref{eq:67}, and
\eqref{eq:68}, we obtain \eqref{eq:63}, thus completing the proof. 
\end{pfn}

\end{document}